\documentclass[10pt,a4paper,oneside]{article}
\usepackage[english]{babel}
\usepackage{a4wide,amsmath,amsthm,amssymb,url,graphicx,xspace,algorithm,algorithmic,pgf}
\usepackage[latin1]{inputenc}
\usepackage{enumitem}
\usepackage{multirow}
\usepackage{hyperref}
\usepackage{tabularx}
\usepackage{tikz}
\usepackage{cases}
\usepackage{bm}
\allowdisplaybreaks

\def\F{\mathbb{F}}

\def\eps{\varepsilon}

\DeclareMathOperator{\PG}{PG}

\DeclareMathOperator{\PGO}{PGO}

\newtheorem{theorem}{Theorem}[section]

\newtheorem{lemma}[theorem]{Lemma}

\newtheorem{corollary}[theorem]{Corollary}
\newtheorem{example}[theorem]{Example}

\theoremstyle{definition}
\newtheorem{definition}[theorem]{Definition}
\newtheorem{remark}[theorem]{Remark}
\newtheorem{notation}[theorem]{Notation}

\newtheoremstyle{dotless}{}{}{\itshape}{}{\bfseries}{}{ }{}

  \theoremstyle{dotless}

\newcommand{\comments}[1]{}
\newcommand{\gs}[3]{\genfrac{[}{]}{0pt}{}{#1}{#2}_{#3}}

\author{Maarten De Boeck\thanks{Department of Mathematical Sciences, University of Memphis, Dunn Hall, 3725 Norriswood Ave, Memphis, TN 38152, USA. ORCID: 0000-0001-8399-9064. \href{mailto:mdeboeck@memphis.edu}{\url{mdeboeck@memphis.edu}}\newline Department of Mathematics: Algebra and Geometry, Ghent University, Gent, Flanders, Belgium}~ and Geertrui Van de Voorde\thanks{School of Mathematics and Statistics, University of Canterbury, Private Bag 4800, 8140 Christchurch, New Zealand. ORCID: 0000-0002-4957-6911 \href{mailto:geertrui.vandevoorde@canterbury.ac.nz}{geertrui.vandevoorde@canterbury.ac.nz}}}
\title{Anzahl theorems for disjoint subspaces generating a non-degenerate subspace II: quadratic forms}
\date{}
\begin{document}
\maketitle

\begin{abstract}
        In this paper, we solve a classical counting problem for non-degenerate quadratic forms defined on a vector space in odd characteristic; given a subspace $\pi$, we determine the number of non-singular subspaces that are trivially intersecting with $\pi$ and span a non-singular subspace with $\pi$. Lower bounds for the quantity of such pairs where $\pi$ is non-singular were first studied in ``Glasby, Niemeyer, Praeger (Finite Fields Appl., 2022)'', which was later improved for even-dimensional subspaces in ``Glasby, Ihringer, Mattheus (Des. Codes Cryptogr., 2023)'' and  generalised in ``Glasby, Niemeyer, Praeger (Linear Algebra Appl., 2022)''. The explicit formulae, which allow us to give the exact proportion and improve the known lower bounds were derived in the symplectic and Hermitian case in ``De Boeck and Van de Voorde (Linear Algebra Appl. 2024)''. This paper deals with the more complicated quadratic case. 
\end{abstract}

\section{Introduction}
Counting theorems are at the heart of finite geometry, and so-called \emph{Anzahl theorems} have been studied for geometries defined over finite fields for many decades already. For a more comprehensive introduction to the topic, we refer to the introduction of Part I of this work, \cite{dbvdv24}.
\par The research in this paper has been instigated by the problem posed and discussed in \cite{gim,gnp,gnp2}: given a form on $\F^{n}_{q}$ it asks for the proportion of pairs of non-singular subspaces that are trivially intersecting and span a non-singular subspace. The motivation for this problem comes from computational group theory, more precisely the algorithms for recognising classical groups. For more details we refer to \cite[Section 1.1]{gnp2}. In our previous work, \cite{dbvdv24}, we determined these proportions exactly for the case of Hermitian and symplectic forms. The current paper deals with the (much more involved) case of quadratic forms over fields of odd characteristic. In order to derive those proportions, we first count, given a fixed subspace $\pi$ of $\F^{n}_{q}$, the exact number of non-singular subspaces trivially intersecting $\pi$ and spanning a non-singular subspace with it.  We will not restrict ourselves to the case where $\pi$ is non-singular. The number we obtain will depend on multiple parameters, which include but are not limited to the dimension and type of the subspace (see later). The main result is given in Theorem \ref{th:gammageneral}: there we derive the explicit formula for the quantity $\gamma_{(i,j,\delta(,\lambda)),(n,\eps),(k,\zeta),\eta}$. This number counts the following: consider a quadratic form of type $\eps$ on $\mathbb{F}_q^n$, $q$ odd, and let $\pi$ be an $i$-singular $j$-space of type $\delta$, with perp type $\lambda$ if $n(j-i)$ odd, then this is the number of non-singular $k$-spaces $\sigma$ of type $\zeta$ such that $\sigma \cap \pi$ is trivial and $\langle \sigma,\pi\rangle$ is non-singular of type $\eta$. 
\par The paper is organised as follows. In Section \ref{sec:prelim} we will introduce the necessary background, and in Section \ref{sec:computations} introduce the functions that we will use to describe the results. Section \ref{sec:proofofmain} is devoted to the proof of four main theorems: Theorem \ref{th:gammaquadratic}, Corollary \ref{cor:gammaquadratic0odd} and Theorem \ref{th:gammaquadratic0even} deal with the case of complementary subspaces, and Theorem \ref{th:gammageneral} settles the general case. Finally, in Section \ref{sec:proportion} we will determine the proportion that was investigated in \cite{gim,gnp,gnp2}. The technical calculations used in the proof of Theorem \ref{th:gammaquadratic} are collected in Appendix \ref{ap:calculations} and the calculations used in the proof Theorem \ref{th:gammaquadratic0even} are collected in Appendix \ref{ap:morecalculations}.

\section{Preliminaries}\label{sec:prelim}
We start by recalling the background on quadratic and orthogonal forms, and the associated polarities.

\begin{definition}
    A \emph{bilinear form} on the vector space $\F^{n}_{q}$ is a map $f:\F^{n}_{q}\times \F^{n}_{q}\to\F_{q}$ such that for all $u,v,w\in\F^{n}_{q}$ and all $\lambda,\mu\in\F_{q}$
    \[
        f(u+v,w)=f(u,w)+f(v,w),\qquad f(u,v+w)=f(u,v)+f(u,w)\quad\text{and}\quad f(\lambda v,\mu w)=\lambda\mu f(v,w).
    \]
    
\end{definition}

\begin{definition}
    A \emph{quadratic form} on the vector space $\F^{n}_{q}$ is a map $Q:\F^{n}_{q}\to\F_{q}$ such that for all $v\in\F^{n}_{q}$ and all $\lambda\in\F_{q}$
    \[
        Q(\lambda v)=\lambda^{2} f(v).
    \]
\end{definition}

A bilinear form $f$ on $\F^{n}_{q}$ is called \emph{symmetric} if $f(v,w)=f(w,v)$ for all $v,w\in\F^{n}_q$. We can associate a {\em quadratic} form $Q(v)=f(v,v)$ to the symmetric bilinear form $f$ and when the characteristic of $q$ is odd, $f$ is uniquely determined by $Q$. 

\begin{definition}
    The \emph{radical} of a symmetric bilinear form $f$ on $\F^{n}_{q}$ is the set $\{v\in\F^{n}_{q}\mid\forall w\in\F^{n}_{q}:f(v,w)=0\}$. If the radical of $f$ only contains the zero vector, then $f$ is \emph{non-degenerate}. 

With the radical of a quadratic form we mean the radical of the associated symmetric bilinear form.\end{definition}

It is a classical result in linear algebra that every reflexive sesquilinear form is either symmetric, symplectic or a scalar multiple of a Hermitian form. In this article we will discuss the symmetric case. We now introduce the associated polarity.

\begin{definition}
    Given a non-degenerate symmetric bilinear form $f$ on $\F^{n}_{q}$, $q$ odd, we define the corresponding  \emph{orthogonal polarity} $\perp$ as follows:
    \begin{align*}
        \perp:S\to S:\pi\mapsto \{x\in\F^{n}_{q}\mid \forall y\in \pi:f(x,y)=0\},
    \end{align*}
    where $S$ is the set of all subspaces of $\F^{n}_{q}$.
\end{definition}

A polarity on $\F^{n}_{q}$ is an involutory inclusion-reversing permutation of the subspaces. We have that $\dim(\pi)+\dim(\pi^\perp)=n$.

\begin{definition}
    Given a quadratic form $Q$ on a vector space $V$, the subspaces $\pi$ in $V$ such that the restriction on $\pi$ has an $i$-dimensional radical are called \emph{$i$-singular}. A 0-singular space is also called \emph{non-singular}. If $Q$ vanishes on $\pi$ (equivalently, the radical of $\pi$ restricted to $f$ is $\pi$ itself), then $\pi$ is \emph{totally isotropic}.
\end{definition}

The totally isotropic subspaces with respect to an orthogonal form are precisely the subspaces $\pi$ such that $\pi\subseteq\pi^{\perp}$ where $\perp$ is the corresponding polarity.

\begin{definition}
    The incidence structure consisting of all totally isotropic subspaces with respect to an orthogonal form is an \emph{orthogonal polar space}. 
\end{definition}

The non-degenerate quadratic and orthogonal forms defined on $\F_q^n$ fall into three equivalence classes called {\em types} determined by a parameter $\eps$: one for odd $n$ (the {\em parabolic} type, $\eps=0$) and two for even $n$. For even $n$, the orthogonal form is called {\em hyperbolic} ($\eps=1$) if the maximum dimension of a totally isotropic subspace is $\frac{n}{2}$, and {\em elliptic} ($\eps=-1$) if this dimension is $\frac{n}{2}-1$. The {\em type} of a subspace is the type of the restriction of the form to that subspace.

\begin{remark} The type of the $0$-dimensional subspace (projectively, the empty space) is hyperbolic. 
\end{remark}

The orthogonal polar spaces are a class of \emph{classical polar spaces}. A classical polar space has a natural embedding in a vector space, and can therefore also be naturally embedded in the corresponding projective space. The type of the polar space defined by a non-degenerate quadratic form is the type of the form.

Looking at the orthogonal polar space $\mathcal{P}$ embedded in the projective space $\PG(n-1,q)$, an $i$-singular $j$-space in $\F^{n}_{q}$ corresponds to a $(j-1)$-space of $\PG(n-1,q)$ meeting $\mathcal{P}$ in a cone with an $(i-1)$-dimensional vertex and as base a polar space of the same type as $\mathcal{P}$ in a $(j-i-1)$-space disjoint to this vertex (where we use vectorial dimension for the subspaces of $\F^{n}_{q}$, and projective dimension for their interpretation in $\PG(n-1,q)$). The type of an $i$-singular $j$-space in $\F_q^n$ is the type of non-singular polar space forming the base of this cone.

In this paper, we will use the following shorthand conventions.
\begin{definition}\label{def:functionsphipsichi}
    For integers $b\geq a\geq0$ we define
    \begin{align*}
        \psi_{a,b}^{+}(q)=\prod_{k=a}^{b}\left(q^{k}+1\right),\qquad
        \psi_{a,b}^{-}(q)=\prod_{k=a}^{b}\left(q^{k}-1\right),\qquad\chi_{a,b}(q)=\prod_{k=a}^{b}\left(q^{2k-1}-1\right)
    \end{align*}
  
    Furthermore, we set $\psi_{a,a-1}^{+}(q)=\psi_{a,a-1}^{-}(q)=\chi_{a,a-1}(q)=1$ for $a\geq0$; this corresponds to the empty product.
\end{definition}

The classical \emph{Gaussian binomial coefficient} can be described using the function $\psi^{-}$.
\begin{definition}\label{def:gaussianbinomialcoeffient}
    For integers $b\geq a\geq 0$ and prime powers $q$ we define
    \begin{align*}
	   \gs{b}{a}{q}=\frac{\psi^{-}_{b-a+1,b}(q)}{\psi^-_{1,a}(q)}\:.
    \end{align*}
    Furthermore, we set $\gs{b}{a}{q}=0$ if $a<0$ or $0\leq b<a$, and we set $\gs{-1}{0}{q}=1$.
\end{definition}

It is well-known that the number of $a$-dimensional subspaces in $\F^{b}_{q}$ is given by $\gs{b}{a}{q}$.

We also have the following analogues of Pascal's identity which will be used throughout the computations in the appendix of this paper.

\begin{lemma}\label{lem:pascalgeneral}
For integers $b\geq a\geq 0$ we have

\begin{align*}
    \gs{b}{a}{q}=q^{a}\gs{b-1}{a}{q}+\gs{b-1}{a-1}{q}\qquad \text{and}\qquad
    \gs{b}{a}{q}=\gs{b-1}{a}{q}+q^{b-a}\gs{b-1}{a-1}{q}\:.
\end{align*}
Note that also the equalities $\gs{0}{0}{q}=q^{0}\gs{-1}{0}{q}+\gs{-1}{-1}{q}$ and $\gs{0}{0}{q}=\gs{-1}{0}{q}+q^{0}\gs{-1}{-1}{q}$ hold, since we set $\gs{-1}{0}{q}=1$, while $\gs{-1}{-1}{q}=0$.

\end{lemma}

\section{Preliminary counting results}\label{sec:computations}%General computations

We now introduce some further notations that will be used to count the quantities in our main theorem.

\subsection{\texorpdfstring{$\alpha_{(i,j,\delta),(n,\eps)}$ and $\alpha_{(i,j,0,\lambda),(n,\eps)}$}{The alphas}}

\begin{definition}\label{def:alphaquadratic}
    Given a quadratic form $f$ of type $\eps$ on $\F^{n}_{q}$, we define $\alpha_{(i,j,\delta),(n,\eps)}$ as the number of $i$-singular $j$-spaces of type $\delta$ with respect to $f$.
\end{definition}

We first look at the action of the orthogonal groups $\PGO(n,q)$, $\PGO^{+}(n,q)$ and $\PGO^{-}(n,q)$ on the $i$-singular subspaces, where we abuse notation to let $\PGO^\eps$ correspond to $\PGO^{+}$, $\PGO$ and $\PGO^{-}$ when $\eps=1,0,-1$, respectively.

\begin{theorem}[{\cite[Theorem 1.49 and Section 1.8]{ht}}]\label{th:orbitsPGO}
    Let $f$ be a quadratic form of type $\eps$ on $\F^{n}_{q}$, $q$ odd, and consider the action of the group $\PGO^{\eps}(n,q)$ on the $i$-singular $j$-spaces of $\F^{n}_{q}$. This action is transitive if $n$ is even, or $j-i$ is even or $n-j-i=0$.
    \par If $n$ and $j-i$ are both odd and $n-i-j>0$, the action has two orbits. In this case both $f$ and the $j$-spaces have type $0$ (parabolic). Let $\perp$ be the orthogonal polarity corresponding to $f$. The orbits correspond to the $j$-spaces that under $\perp$ are mapped to $(n-j)$-spaces that have type $1$ or $-1$.
\end{theorem}

For these $i$, $j$ and $n$ where the action of $\PGO^{\eps}(n,q)$ on the $i$-singular $j$-spaces has two orbits, we define the perp type and introduce the corresponding $\alpha$-notation. Since we assume that $q$ is odd, there is always a polarity associated to the quadratic form.

\begin{definition}
    Let $n$ and $j-i$ be both odd, and let $f$ be a quadratic form (of type $0$) on $\F^{n}_{q}$, $q$ odd. The \emph{perp type} of a $j$-space $\pi$ (of type 0) is the type of $\pi^\perp$, where $\perp$ is the polarity associated to $f$.
\end{definition}

Note that the perp type is always $1$ or $-1$. Given Theorem \ref{th:orbitsPGO}, the following definition is well defined.

\begin{definition}\label{def:alphaperp}
    Given a quadratic form $f$ of type $0$ on $\F^{n}_{q}$, $q$ odd, we define $\alpha_{(i,j,0,\lambda),(n,0)}$ as the number of $i$-singular $j$-spaces with perp type $\lambda$.
\end{definition}

\begin{remark}\label{rem:n=i+jalpha}
    In the case that $n$ and $(j-i)$ are both odd, and $n=i+j$, there is only a single orbit of $i$-singular $j$-spaces with respect to a quadratic form in $\F_{q}^{n}$, as mentioned in Theorem \ref{th:orbitsPGO}. All these $i$-singular $j$-spaces have perp type 1, since their image under the polarity is their radical, which by definition has hyperbolic type. So, in case $n$ and $(j-i)$ are both odd, we have $\alpha_{(i,j,0),(i+j,0)}=\alpha_{(i,j,0,1),(i+j,0)}$ and $\alpha_{(i,j,0,-1),(i+j,0)}=0$.
\end{remark}

The formulae for $\alpha_{(i,j,\delta),(n,\eps)}$ and $\alpha_{(i,j,0,\lambda),(n,0)}$ (see Definitions \ref{def:alphaquadratic} and \ref{def:alphaperp}) are known from the literature. The formula for $\alpha_{(i,j,\delta),(n,\eps)}$ was first determined by Dai and Feng (\cite{df1,df2}).
\begin{theorem}[{\cite[Theorem 1.57]{ht}}]\label{th:alphaorthogonal}
    Let $n,j,i$ be integers with $0\leq i\leq j$ and $i+j\leq n$, and let $\delta,\eps\in\{0,\pm1\}$. If $j-i\equiv\delta\pmod{2}$ or $n\equiv\eps\pmod{2}$, then $\alpha_{(i,j,\delta),(n,\eps)}=0$. For the case $j-i-\delta\equiv1\equiv n-\eps\pmod{2}$ , with $\delta=1$ if $i=j$, we have:
    \begin{align*}
        \alpha_{(i,j,\delta),(n,\eps)}&=q^{\frac{1}{2}(n-j-i)\left((j-i)+(1-\delta^2)(1-\eps^2)\right)-\frac{1}{2}(1-\delta^2)\eps^2}\cdot\\
        &\quad\;\ \frac{\psi^{+}_{\frac{1}{2}(n-i-j+(\delta^2-1)(\eps^2-1)+1-\eps\delta),\frac{1}{2}(n-1-\eps)}(q)\ \psi^{-}_{\frac{1}{2}(n-i-j+(\delta^2-1)(\eps^2-1)+1+\eps\delta),\frac{1}{2}(n-1+\eps)}(q)}{\psi^{+}_{1-\delta^2,\frac{1}{2}(j-i-\delta-1)}(q)\ \psi^{-}_{1,\frac{1}{2}(j-i+\delta-1)}(q)\ \psi^{-}_{1,i}(q)}\:.
    \end{align*}
\end{theorem}

We determine simplified formulae for the case where $\alpha$ describes a number of hyperplanes, which will appear frequently in the proof of the main theorem.

\begin{corollary}\label{cor:alphaorthogonal}
    Let $n\geq1$ be an integer, and let $\delta,\eps\in\{0,\pm1\}$, with $n-\eps\equiv1\pmod{2}$. If $n\equiv\delta\pmod{2}$, we have
    \begin{align*}
        \alpha_{(0,n-1,\delta),(n,\eps)}&=
        \begin{cases}
            \frac{1}{2}q^{\frac{1}{2}(n-1)}\left(q^{\frac{1}{2}(n-1)}+\delta\right)&\delta\in\{\pm1\},\eps=0\\
            q^{\frac{1}{2}n-1}\left(q^{\frac{1}{2}n}-\eps\right)&\delta=0,\eps\in\{\pm1\}
        \end{cases}\:.
    \end{align*}
    If $n\equiv\delta+1\pmod{2}$, we have:
    \begin{align*}
        \alpha_{(1,n-1,\delta),(n,\eps)}&=
        \begin{cases}
            \frac{q^{n-1}-1}{q-1}+\eps q^{\frac{1}{2}n-1}&\delta=\eps\\
            0&\text{else}
        \end{cases}\:.
    \end{align*}
\end{corollary}

\par The formula for $\alpha_{(i,j,0,\lambda),(n,0)}$ follows immediately from the formula for $\alpha_{(i,j,\delta),(n,0)}$ by considering the image of the $i$-singular $j$-spaces under the polarity.

\begin{theorem}[{\cite[Theorem 1.65]{ht}}]\label{th:alphaorthogonalodd}
    Let $n,j,i$ be integers with $0\leq i\leq j$ and $i+j\leq n$, and assume that $n(j-i)$ is odd.% Let $\lambda\in\{0,\pm1\}$. If $\lambda=0$, then $\alpha_{(i,j,0,\lambda),(n,0)}=0$.
    For $\lambda\in\{\pm1\}$, then we have
    \begin{align*}
        \alpha_{(i,j,0,\lambda),(n,0)}&=\alpha_{(i,n-j,\lambda),(n,0)}%\\
        =q^{\frac{1}{2}(j-i)\left(n-j-i\right)}\cdot\frac{\psi^{+}_{\frac{1}{2}(j-i+1),\frac{1}{2}(n-1)}(q)\ \psi^{-}_{\frac{1}{2}(j-i+1),\frac{1}{2}(n-1)}(q)}{\psi^{+}_{0,\frac{1}{2}(n-j-i-\lambda-1)}(q)\ \psi^{-}_{1,\frac{1}{2}(n-j-i+\lambda-1)}(q)\ \psi^{-}_{1,i}(q)}\:.
    \end{align*}
\end{theorem}

\begin{example}
    When $n=3$, the totally isotropic subspaces of dimension 1 with respect to $f$ correspond to points of a conic in $\PG(2,q)$. Their number is given by $\alpha_{(1,1,1),(3,0)}$, which, by Theorem \ref{th:alphaorthogonal}, is indeed $\psi^+_{1,1}(q)=q+1$. The number $\alpha_{(0,2,1),(3,0)}$ is the number of $0$-singular lines of hyperbolic type, that is, secant lines. 
    Indeed, Corollary \ref{cor:alphaorthogonal} with $n=3, \delta=1,\eps=0$ shows that $\alpha_{(0,2,1),(3,0)}=\frac{1}{2}q(q+1)$. 
    Similarly, $\alpha_{(0,1,0,-1),(3,0)}$ is the number of internal points, not on the conic: these are the points whose image under the polarity is of elliptic type (i.e. a passant to the conic). Indeed, Theorem \ref{th:alphaorthogonalodd}, with $n=3, i=0,j=1,\lambda=-1$ shows that $\alpha_{(0,1,0,-1),(3,0)}=q^{\frac{1}{2}2}\frac{\psi^+_{1,1}(q)\psi^-_{1,1}(q)}{\psi^+_{0,1}(q)\psi^-_{1,0}(q)\psi^-_{1,0}(q)}=q\frac{(q+1)(q-1)}{2(q+1)}=\frac{1}{2}q(q-1).$
\end{example}

\begin{example}\label{ex:alfa}
    Consider $i=0,j=2,\delta=1,n=4,\eps=1$, then $\alpha_{(0,2,1),(4,1)}$ counts the number of secant lines to a hyperbolic quadric $\mathcal{Q}^+(3,q).$ This number is $q^2\frac{(q+1)^2}{2}$, which is indeed equal to $q^2\frac{\psi^+_{1,1}\psi^-_{2,2}}{\psi^+_{0,0}\psi^-_{1,1}\psi^-_{1,0}}$ given in Theorem \ref{th:alphaorthogonal}.
\end{example}

\subsection{\texorpdfstring{$\beta_{(i,j,\delta),(n,\eps),(k,\zeta)}$ and its variations}{The betas}}

Given Theorem \ref{th:orbitsPGO} the following definition is well defined, i.e~independent of the chosen subspace $\pi$.

\begin{definition}\label{def:betaquadratic}
    Assume that $n(j-i)$ even or $n-j-i=0$. Consider a quadratic form $f$ of type $\eps$ on $\F^{n}_{q}$, $q$ odd, and let $\pi$ be an $i$-singular $j$-space of type $\delta$. Then, $\beta_{(i,j,\delta),(n,\eps),(k,\zeta)}$ is the number of non-singular $k$-spaces $\sigma\supseteq\pi$ of type $\zeta$ in $\F_q^n$.
\end{definition}

\begin{remark}\label{rem:projectiveinterpretationbeta}
    Using projective dimension and notation, we can say that $\beta_{(i,j,\delta),(n,\eps),(k,\zeta)}$ is the number of $(k-1)$-spaces in $\PG(n-1,q)$ meeting $\mathcal{Q}^\eps(n-1,q)$ in a $\mathcal{Q}^\zeta(k-1,q)$ and containing a fixed $(j-1)$-space $\pi$ that meets $\mathcal{Q}^\eps(n-1,q)$ in a cone $\Pi_{i-1}\mathcal{Q}^\delta(j-i-1,q)$.
\end{remark}

We now introduce two analogues of $\beta_{(i,j,\delta),(n,\eps),(k,\zeta)}$ for the case that there is more than one orbit on the $i$-singular $j$-spaces, i.e.~for the case that the perp types are to be taken into account. Given Theorem \ref{th:orbitsPGO}, the statements in Definition \ref{def:betaquadraticodd} are well defined, i.e~independent of the chosen subspace $\pi$.

\begin{definition}\label{def:betaquadraticodd}
    Consider a quadratic form $f$ of type $\eps$ on $\F^{n}_{q}$, $q$ odd. Let $\pi$ be an $i$-singular $j$-space $\pi$ of type $\delta$. Let $\lambda$ be the perp type of $\pi$ in case $n(j-i)$ is odd.
    \begin{itemize}
        \item For $n$ even, $k(j-i)$ odd (and thus $\delta=0$), $\beta^{\nu}_{(i,j,0),(n,\eps),(k,0)}$ is the number of non-singular $k$-spaces $\sigma\supseteq\pi$ (necessarily of type $0$) in $\F_{q^{n}}$ such that $\pi$ has perp type $\nu$ in $\sigma$.
        \item For $n(j-i)$ odd (and thus $\eps=\delta=0$), $\beta_{(i,j,0,\lambda),(n,0),(k,\zeta)}$ is the number of non-singular $k$-spaces $\sigma\supseteq\pi$ of type $\zeta$ in $\F_{q^{n}}$.
        \item  For $nk(j-i)$ odd (and thus $\eps=\delta=\zeta=0$), $\beta^{\nu}_{(i,j,0,\lambda),(n,0),(k,0)}$ is the number of non-singular $k$-spaces $\sigma\supseteq\pi$ (necessarily of type $0$) in $\F_{q^{n}}$ such that $\pi$ has perp type $\nu$ in $\sigma$.
    \end{itemize}
\end{definition}

The betas in Definition \ref{def:betaquadraticodd} and gammas in Definition \ref{def:gammaquadraticodd} can be interpreted in projective space, similar to the interpretation in Remarks \ref{rem:projectiveinterpretationbeta} and \ref{rem:projectiveinterpretationgamma}.

\begin{remark}\label{rem:n=i+jbeta}
    The case with $n$ and $j-i$ both odd and $n=i+j$ was discussed earlier in Remark \ref{rem:n=i+jalpha}. Similarly, we now find $\beta_{(i,j,0),(i+j,0),(k,\zeta)}=\beta_{(i,j,0,1),(i+j,0),(k,\zeta)}$. Furthermore, $\beta_{(i,j,0,-1),(i+j,0),(k,\zeta)}$ is not defined.
\end{remark}

We now establish a formula for $\beta_{(i,j,\delta),(n,\eps),(k,\zeta)}$.

\begin{theorem}\label{th:betaorthogonal}
    Let $n,j,i,k$ be integers with $0\leq i\leq j$ and $i+j\leq k\leq n$, and let $\delta,\eps,\zeta\in\{0,\pm1\}$ with $j-i-\delta\equiv1\equiv n-\eps\pmod{2}$, such that $n(j-i)$ is even or $n-j-i=0$. If $k\equiv\zeta\pmod{2}$, then $\beta_{(i,j,\delta),(n,\eps),(k,\zeta)}=0$. For the case $k-\zeta\equiv1\pmod{2}$ we have:
    \begin{align*}
        &\beta_{(i,j,\delta),(n,\eps),(k,\zeta)}\\
        &=q^{\frac{1}{2}\left((n-k)(k-j+i)+n(\delta^2-\zeta^2)(1-\eps^2)+(i+j)(1-\delta^2)(\zeta^2-\eps^2)+k(1-\zeta^2)(\eps^2-\delta^2)-\eps^2\delta^2+\zeta^2(\delta^2+\eps^2-1)\right)}\cdot\\
        &\quad\;\ \frac{\psi^{+}_{\frac{1}{2}(n-k+(\zeta^2-1)(\eps^2-1)+1-\eps\zeta),\frac{1}{2}(n-i-j+(\delta^2-1)(\eps^2-1)-1-\eps\delta)}(q)}{\psi^{+}_{1-\zeta^2,\frac{1}{2}(k-i-j+(\delta^2-1)(\zeta^2-1)-1-\zeta\delta)}(q)}\cdot\\
        &\quad\;\ \frac{\psi^{-}_{\frac{1}{2}(n-k+(\zeta^2-1)(\eps^2-1)+1+\eps\zeta),\frac{1}{2}(n-i-j+(\delta^2-1)(\eps^2-1)-1+\eps\delta)}(q)}{\psi^{-}_{1,\frac{1}{2}(k-i-j+(\delta^2-1)(\zeta^2-1)-1+\zeta\delta)}(q)}
    \end{align*}
\end{theorem}
\begin{proof}
    Given a quadratic form of type $\eps$ on $\F_{q}^{n}$, we count the tuples $(\pi,\sigma)$ such that $\pi$ is an $i$-singular $j$-space of type $\delta$ and $\sigma\supseteq\pi$ is a non-singular $k$-space of type $\zeta$, in two ways. We find that
	\[
		\alpha_{(i,j,\delta),(n,\eps)}\:\beta_{(i,j,\delta),(n,\eps),(k,\zeta)}=\alpha_{(0,k,\zeta),(n,\eps)}\:\alpha_{(i,j,\delta),(k,\zeta)}\:.
	\]
	Hence, using Theorem \ref{th:alphaorthogonal} we find
    \begin{align*}
        &\beta_{(i,j,\delta),(n,\eps),(k,\zeta)}\\&=\frac{\alpha_{(0,k,\zeta),(n,\eps)}\:\alpha_{(i,j,\delta),(k,\zeta)}}{\alpha_{(i,j,\delta),(n,\eps)}}\\
        &=q^{\frac{1}{2}(n-k)\left(k+(1-\zeta^2)(1-\eps^2)\right)-\frac{1}{2}(1-\zeta^2)\eps^2}\cdot\\
        &\quad\;\ \frac{\psi^{+}_{\frac{1}{2}(n-k+(\zeta^2-1)(\eps^2-1)+1-\eps\zeta),\frac{1}{2}(n-1-\eps)}(q)\ \psi^{-}_{\frac{1}{2}(n-k+(\zeta^2-1)(\eps^2-1)+1+\eps\zeta),\frac{1}{2}(n-1+\eps)}(q)}{\psi^{+}_{1-\zeta^2,\frac{1}{2}(k-\zeta-1)}(q)\ \psi^{-}_{1,\frac{1}{2}(k+\zeta-1)}(q)}\cdot\\
        &\quad\ q^{\frac{1}{2}(k-j-i)\left((j-i)+(1-\delta^2)(1-\zeta^2)\right)-\frac{1}{2}(1-\delta^2)\zeta^2}\cdot\\
        &\quad\;\ \frac{\psi^{+}_{\frac{1}{2}(k-i-j+(\delta^2-1)(\zeta^2-1)+1-\zeta\delta),\frac{1}{2}(k-1-\zeta)}(q)\ \psi^{-}_{\frac{1}{2}(k-i-j+(\delta^2-1)(\zeta^2-1)+1+\zeta\delta),\frac{1}{2}(k-1+\zeta)}(q)}{\psi^{+}_{1-\delta^2,\frac{1}{2}(j-i-\delta-1)}(q)\ \psi^{-}_{1,\frac{1}{2}(j-i+\delta-1)}(q)\ \psi^{-}_{1,i}(q)}\cdot\\
        &\quad\ q^{-\frac{1}{2}(n-j-i)\left((j-i)+(1-\delta^2)(1-\eps^2)\right)+\frac{1}{2}(1-\delta^2)\eps^2}\cdot\\
        &\quad\;\ \frac{\psi^{+}_{1-\delta^2,\frac{1}{2}(j-i-\delta-1)}(q)\ \psi^{-}_{1,\frac{1}{2}(j-i+\delta-1)}(q)\ \psi^{-}_{1,i}(q)}{\psi^{+}_{\frac{1}{2}(n-i-j+(\delta^2-1)(\eps^2-1)+1-\eps\delta),\frac{1}{2}(n-1-\eps)}(q)\ \psi^{-}_{\frac{1}{2}(n-i-j+(\delta^2-1)(\eps^2-1)+1+\eps\delta),\frac{1}{2}(n-1+\eps)}(q)}\\
        &=q^{\frac{1}{2}\left((n-k)(k-j+i)+n(\delta^2-\zeta^2)(1-\eps^2)+(i+j)(1-\delta^2)(\zeta^2-\eps^2)+k(1-\zeta^2)(\eps^2-\delta^2)-\eps^2\delta^2+\zeta^2(\delta^2+\eps^2-1)\right)}\cdot\\
        &\quad\;\ \frac{\psi^{+}_{\frac{1}{2}(n-k+(\zeta^2-1)(\eps^2-1)+1-\eps\zeta),\frac{1}{2}(n-1-\eps)}(q)\ \psi^{+}_{\frac{1}{2}(k-i-j+(\delta^2-1)(\zeta^2-1)+1-\zeta\delta),\frac{1}{2}(k-1-\zeta)}(q)}{\psi^{+}_{1-\zeta^2,\frac{1}{2}(k-\zeta-1)}(q)\ \psi^{+}_{\frac{1}{2}(n-i-j+(\delta^2-1)(\eps^2-1)+1-\eps\delta),\frac{1}{2}(n-1-\eps)}(q)}\cdot\\
        &\quad\;\ \frac{\psi^{-}_{\frac{1}{2}(n-k+(\zeta^2-1)(\eps^2-1)+1+\eps\zeta),\frac{1}{2}(n-1+\eps)}(q)\ \psi^{-}_{\frac{1}{2}(k-i-j+(\delta^2-1)(\zeta^2-1)+1+\zeta\delta),\frac{1}{2}(k-1+\zeta)}(q)}{\psi^{-}_{1,\frac{1}{2}(k+\zeta-1)}(q)\ \psi^{-}_{\frac{1}{2}(n-i-j+(\delta^2-1)(\eps^2-1)+1+\eps\delta),\frac{1}{2}(n-1+\eps)}(q)}\\
        &=q^{\frac{1}{2}\left((n-k)(k-j+i)+n(\delta^2-\zeta^2)(1-\eps^2)+(i+j)(1-\delta^2)(\zeta^2-\eps^2)+k(1-\zeta^2)(\eps^2-\delta^2)-\eps^2\delta^2+\zeta^2(\delta^2+\eps^2-1)\right)}\cdot\\
        &\quad\;\ \frac{\psi^{+}_{\frac{1}{2}(n-k+(\zeta^2-1)(\eps^2-1)+1-\eps\zeta),\frac{1}{2}(n-i-j+(\delta^2-1)(\eps^2-1)-1-\eps\delta)}(q)}{\psi^{+}_{1-\zeta^2,\frac{1}{2}(k-i-j+(\delta^2-1)(\zeta^2-1)-1-\zeta\delta)}(q)}\cdot\\
        &\quad\;\ \frac{\psi^{-}_{\frac{1}{2}(n-k+(\zeta^2-1)(\eps^2-1)+1+\eps\zeta),\frac{1}{2}(n-i-j+(\delta^2-1)(\eps^2-1)-1+\eps\delta)}(q)}{\psi^{-}_{1,\frac{1}{2}(k-i-j+(\delta^2-1)(\zeta^2-1)-1+\zeta\delta)}(q)}\qedhere
    \end{align*}
\end{proof}

We determine a few simplified formulae for the case where $\beta$ gives a number of hyperplanes, which will appear frequently in the proof of the main theorem.

\begin{corollary}\label{cor:betas}
    Let $n,j,i$ be integers with $0\leq i\leq j$ and $i+j\leq n-1$, and let $\delta,\eps,\zeta\in\{0,\pm1\}$ with $j-i-\delta\equiv1\equiv n-\eps\pmod{2}$ and $n\equiv\zeta\pmod{2}$, such that $n(j-i)$ is even or $n-j-i=0$.
    For $\eps\in\{\pm1\}$ we have
    \begin{align*}
        \beta_{(i,j,\delta),(n,\eps),(n-1,0)}&=q^{\frac{1}{2}(n-j+i)-1}\left(q^{\frac{1}{2}(n-j-i)}-\delta\eps\right)\:.
    \end{align*}
    For $\eps=0$ and $\delta\in\{\pm1\}$ we have
    \begin{align*}
        \beta_{(i,j,\delta),(n,0),(n-1,\zeta)}&=\frac{1}{2}q^{\frac{1}{2}(n-j+i-1)}\left(q^{\frac{1}{2}(n-j-i-1)}+\delta\zeta\right).
    \end{align*}
    
\end{corollary}

\begin{example}
    The number $\beta_{(0,2,1),(4,1),(3,0)}$ is the number of conic planes in $\PG(3,q)$, through a given secant line to a hyperbolic quadric. Corollary \ref{cor:betas}, with $i=0,j=2,\delta=1, n=4,\eps=1$ indeed says that this number is $q^0(q-1)=q-1.$
\end{example}

We will derive the formulae to compute $\beta^{\nu}_{(i,j,0),(n,\eps),(k,0)}$, $\beta_{(i,j,0,\lambda),(n,0),(k,\zeta)}$ and $\beta^{\nu}_{(i,j,0,\lambda),(n,0),(k,\zeta)}$ in Theorem \ref{th:betavariationsorthogonal}. But first we prove the following lemma and its corollary which will be useful later.

\begin{lemma}\label{lem:halvedorbits}
    Let $i,j,k,n$ be integers with $0\leq i\leq j$, $i+j\leq k\leq n$ and $j-i$ odd, and let $\lambda\in\{\pm1\}$. Let $Q$ be a quadratic form on $\F^{n}_{q}$, $q$ odd and $n$ odd, and let $\pi$ be a non-singular $k$-space, with $k$ even. Precisely half of the $i$-singular $j$-spaces in $\pi$ (all of type 0) have perp type $\lambda$ with respect to $Q$. 
\end{lemma}
\begin{proof}
    Let $\mathcal{S}$ be the set of $i$-singular $j$-spaces in $\pi$. Since $k$ is even, $Q\vert_\pi$ has type $\eps\in\{1,-1\}$. Let $\mathcal{S}_{\mu}$ be the subset of $\mathcal{S}$ consisting of the subspaces having perp type $\mu$ with respect to $Q$. 
    Let $\mathrm{PO}_n$ denote the projective isometry group of $Q$. Then it is easy to see (e.g. \cite[p17]{kleidman}) that the stabiliser of $\pi$ contains the projective isometry group of $Q\vert_\pi$, $\mathrm{PO}^{\eps}_k$, as a subgroup . 
    
    We know that $\mathcal{S}_1$ and $\mathcal{S}_2$ are orbits under the action of $\mathrm{PO}^{\eps}_k$. 
    Since $q$ is odd, the projective isometry group $\mathrm{PO}^{\eps}_k$ is a subgroup of index $2$ of the projective similarity group $\mathrm{PGO}^{\eps}_k$ (see e.g. \cite[Table 2.1D]{kleidman}), and we know that the action of $\mathrm{PGO}^{\eps}_k$ is transtive on $\mathcal{S}$. It follows that the $\mathrm{PO}^{\eps}_k$-orbits $S_1$ and $S_2$ have the same size. 
\end{proof}

\begin{corollary}
    \label{cor:halvedorbits} Let $i,j,n$ be integers with $0\leq i\leq j$, $i+j\leq n$, $j,i$ odd, and let $\lambda\in\{\pm1\}$. Let $Q$ be a quadratic form on $\F^{n}_{q}$, $q$ odd and $n$ odd, and let $\pi$ be an $i$-singular $j$-space with radical $\overline{\pi}$. Precisely half of the hyperplanes through $\bar{\pi}$ in $\pi$ (all of type 0) have perp type $\lambda$ with respect to $Q$. 
\end{corollary}
\begin{proof} 
Consider a $(j-i)$-space $\mu$ of $\pi$ which is complementary to $\overline{\pi}$, and let $H$ be a hyperplane of $\pi$ through $\overline{\pi}$. 
The perp type of $H$ is the same as the perp type of $H\cap \mu$ since $H^\perp$ is an $i$-singular $(n-j-i+1)$-space in the non-singular $(n-j+i+1)$-space $(H\cap\mu)^\perp$. Note that $j-i$ is even and the restriction of $Q$ to $\mu$ is non-singular. By Lemma \ref{lem:halvedorbits} we know that half of the hyperplanes of $\mu$ have perp type $\lambda$, therefore, half of the hyperplanes of $\pi$ through $\overline{\pi}$ have perp type $\lambda$.
\end{proof}

\begin{theorem}\label{th:betavariationsorthogonal}
    Let $n,j,i,k$ be integers with $0\leq i\leq j$ and $i+j\leq k\leq n$, let $\eps\in\{0,\pm1\}$ and let $\zeta,\lambda,\nu\in\{\pm1\}$, with $j-i$ odd and $n-\eps\equiv1\pmod{2}$.
    \par If $n$ is even and $k$ is odd, then
    \begin{align*}
        \alpha_{(i,j,0),(n,\eps)}\beta^{\nu}_{(i,j,0),(n,\eps),(k,0)}&=\alpha_{(0,k,0),(n,\eps)}\alpha_{(i,j,0,\nu),(k,0)}\:.
    \end{align*}
    \par If $n$ is odd (and thus $\eps=0$) and $k$ is even, then
    \begin{align*}
        \alpha_{(i,j,0,\lambda),(n,0)}\beta_{(i,j,0,\lambda),(n,0),(k,\zeta)}&=\frac{1}{2}\alpha_{(0,k,\zeta),(n,0)}\alpha_{(i,j,0),(k,\zeta)}
    \end{align*}
    \par  If $n$ and $k$ are both odd, then
    \begin{align*}
        \beta^{\nu}_{(i,j,0,\lambda),(n,0),(k,0)}&=\alpha_{(0,k-j-i,\nu),(n-j-i,\lambda)}\beta_{(i,k-i,0,\lambda\nu),(n,0),(k,0)}\:.
    \end{align*}
\end{theorem}
\begin{proof}
     The first result can be derived from a double counting argument analogous to the one in the proof of Theorem \ref{th:betaorthogonal}. The same is true for the second one, but here we need to be more cautious, so we present the argument in detail.
     \par Given a quadratic form $Q$ of type $0$ on $\F_{q}^{n}$, $n$ odd, we count the tuples $(\pi,\sigma)$ such that $\pi$ is an $i$-singular $j$-space of type $0$ and perp type $\lambda$, and $\sigma\supseteq\pi$ is a non-singular $k$-space of type $\zeta$, in two ways. We know that the number of $i$-singular $j$-spaces of type $0$ and perp type $\lambda$ equals $\alpha_{(i,j,0,\lambda),(n,0)}$ and that through each such subspace there are $\beta_{(i,j,0,\lambda),(n,0),(k,\zeta)}$ non-singular $k$-spaces of type $\zeta$. On the other hand there are $\alpha_{(0,k,\zeta),(n,0)}$ non-singular $k$-spaces of type $\zeta$. Such a space contains $\alpha_{(i,j,0),(k,\zeta)}$ $i$-singular $j$-spaces of type 0. Precisely half of those have perp type $\lambda$ with respect to $Q$ by Lemma \ref{lem:halvedorbits}.
     \par For the third result we follow another approach. Denote the polarity corresponding to the quadratic form by $\perp$. Recall that we assume that $n$ and $k$ are both odd. Let $\pi$ be an $i$-singular $j$-space of type 0 and perp type $\lambda$. Let $\overline{\pi}$ be the $i$-space that is the radical of $\pi$ (the vertex of the cone geometrically). We know that $\pi^\perp$ is an $i$-singular $(n-j)$-space of type $\lambda$ with radical $\overline{\pi}$. Note that $\langle\pi,\pi^\perp\rangle=\overline{\pi}^{\perp}$ is an $i$-singular $(n-i)$-space.
     \par Let $\sigma$ be a non-singular $k$-space through $\pi$ such that $\pi$ has perp type $\nu$ in $\sigma$. We know that $k'=\dim(\sigma\cap\pi^{\perp})\geq k-j$ by the Grassmann identity, and that $\dim(\sigma\cap\langle\pi,\pi^{\perp}\rangle)=k'+j-i\geq k-i$. We also know that $\overline{\pi}$ is contained in the radical of $\sigma\cap\pi^{\perp}$ and of $\sigma\cap\langle\pi,\pi^{\perp}\rangle$, so the radical of $\sigma\cap\langle\pi,\pi^{\perp}\rangle$ has dimension at least $i$. As in general an $a$-singular subspace of a non-singular $k$-space can have dimension at most $k-a$, we find that $k'+j-i\leq k-i$, and hence that $k'\leq k-j$, and we conclude that $k'=k-j$.
     \par Let $\perp'$ be the orthogonal polarity corresponding to the restriction of the quadratic form to $\sigma$. It is immediate that $\pi^{\perp'}=\pi^{\perp}\cap\sigma$, and thus this subspace has type $\nu$.
     \par So, on the one hand, for any non-singular $k$-space $\sigma$ through $\pi$ such that $\pi$ has perp type $\nu$ with respect to it, we find a unique $(k-j)$-space of type $\nu$ through $\overline{\pi}$ in $\pi^\perp$ (namely $\sigma\cap \pi^\perp$). On the other hand, if $\sigma'$ is a $(k-j)$-space of type $\nu$ through $\overline{\pi}$ in $\pi^\perp$, then $\langle\sigma',\pi\rangle$ is an $i$-singular $(k-i)$-space and in any non-singular $k$-space through $\langle\sigma',\pi\rangle$ the subspace $\pi$ has perp type $\nu$. Note that $\langle\sigma',\pi\rangle$ necessarily has type 0; its perp type is the perp type of $\sigma'$ in $\pi^\perp$. Recall that $Q_{|\pi^\perp}$ has type $\lambda$, and that $\sigma'$ has type $\nu$ in $\pi^\perp$, therefore the perp type of $\sigma'$ in $\pi^\perp$ is given by $\lambda\nu$.
     \par We conclude that we can count the non-singular $k$-spaces $\sigma$ through $\pi$ such that $\pi$ has perp type $\nu$ in $\sigma$ by counting the number of pairs $(\sigma',\sigma)$ such that $\sigma'$ is a $(k-j)$-space of type $\nu$ through $\overline{\pi}$ in $\pi^\perp$ and $\sigma\supseteq\langle\sigma',\pi\rangle$ is a non-singular $k$-space. The number of choices for $\sigma'$ equals the number of non-singular $(k-j-i)$-spaces of type $\nu$ in an $(n-j-i)$-space of type $\lambda$, namely $\alpha_{(0,k-j-i,\nu),(n-j-i,\lambda)}$. The number of choices for $\sigma$ is then the number of non-singular $k$-spaces through an $i$-singular $(k-i)$-space of type 0 and perp type $\lambda\nu$, namely $\beta_{(i,k-i,0,\lambda\nu),(n,0),(k,0)}$.
\end{proof}

\begin{corollary}\label{cor:betavariations}
    Let $n,j,i$ be integers with $0\leq i\leq j$ and $i+j\leq n-1$, let $\eps\in\{0,\pm1\}$ and let $\zeta,\lambda,\nu\in\{\pm1\}$, with $j-i$ odd and $n-\eps\equiv1\pmod{2}$.%
    \par\noindent If $n$ is even (and thus $\eps\in\{\pm1\}$), we have
    \begin{align*}
        \beta^{\nu}_{(i,j,0),(n,\eps),(n-1,0)}&=\frac{1}{2}q^{\frac{1}{2}(n-j+i-1)}\left(q^{\frac{1}{2}(n-j-i-1)}+\nu\right).
    \end{align*}
    If $n$ is odd, we have
    \begin{align*}
        \beta_{(i,j,0,\lambda),(n,0),(n-1,\zeta)}=\frac{1}{2}q^{\frac{1}{2}(n-j+i)-1}\left(q^{\frac{1}{2}(n-j-i)}-\lambda\right).
    \end{align*}
\end{corollary}
\begin{proof}
    First assume that $n$ is even. From Theorem \ref{th:alphaorthogonal}, Corollary \ref{cor:alphaorthogonal}, Theorem \ref{th:alphaorthogonalodd} and Theorem \ref{th:betavariationsorthogonal} it follows that
    \begin{align*}
        &\beta^{\nu}_{(i,j,0),(n,\eps),(n-1,0)}\\
        &=\frac{\alpha_{(0,n-1,0),(n,\eps)}\alpha_{(i,j,0,\nu),(n-1,0)}}{\alpha_{(i,j,0),(n,\eps)}}\\
        &=q^{\frac{1}{2}n-1}\left(q^{\frac{1}{2}n}-\eps\right)q^{\frac{1}{2}(j-i)\left(n-j-i-1\right)}\frac{\psi^{+}_{\frac{1}{2}(j-i+1),\frac{1}{2}n-1}(q)\ \psi^{-}_{\frac{1}{2}(j-i+1),\frac{1}{2}n-1}(q)}{\psi^{+}_{0,\frac{1}{2}(n-j-i-\nu)-1}(q)\ \psi^{-}_{1,\frac{1}{2}(n-j-i+\nu)-1}(q)\ \psi^{-}_{1,i}(q)}\\&\qquad q^{-\frac{1}{2}(n-j-i)\left(j-i\right)+\frac{1}{2}}\frac{\psi^{+}_{1,\frac{1}{2}(j-i-1)}(q)\ \psi^{-}_{1,\frac{1}{2}(j-i-1)}(q)\ \psi^{-}_{1,i}(q)}{\psi^{+}_{\frac{1}{2}(n-i-j+1),\frac{1}{2}(n-1-\eps)}(q)\ \psi^{-}_{\frac{1}{2}(n-i-j+1),\frac{1}{2}(n-1+\eps)}(q)}\\
        &=q^{\frac{1}{2}(n-j+i-1)}\left(q^{\frac{1}{2}n}-\eps\right)\\&\qquad\frac{\psi^{+}_{1,\frac{1}{2}n-1}(q)\ \psi^{-}_{1,\frac{1}{2}n-1}(q)}{\psi^{+}_{0,\frac{1}{2}(n-j-i-\nu)-1}(q)\ \psi^{+}_{\frac{1}{2}(n-i-j+1),\frac{1}{2}(n-1-\eps)}(q)\ \psi^{-}_{1,\frac{1}{2}(n-j-i+\nu)-1}(q)\ \psi^{-}_{\frac{1}{2}(n-i-j+1),\frac{1}{2}(n-1+\eps)}(q)}\\
        &=\frac{1}{2}q^{\frac{1}{2}(n-j+i-1)}\left(q^{\frac{1}{2}n}-\eps\right)\left(\begin{cases}
            \frac{1}{q^{\frac{1}{2}n}-1}&\eps=1\\
            \frac{1}{q^{\frac{1}{2}n}+1}&\eps=-1\\
        \end{cases}\right)\left(\begin{cases}
            q^{\frac{1}{2}(n-j-i-1)}+1&\nu=1\\
            q^{\frac{1}{2}(n-j-i-1)}-1&\nu=-1
        \end{cases}\right)\\
        &=\frac{1}{2}q^{\frac{1}{2}(n-j+i-1)}\left(q^{\frac{1}{2}(n-j-i-1)}+\nu\right).
    \end{align*}
    The proof for the second formula, with $n$ odd, is similar.
\end{proof}

\section{Proof of the main result}\label{sec:proofofmain}

In this section we will prove the main result of this article. We start by introducing the following notation. Given Theorem \ref{th:orbitsPGO}, Definition \ref{def:gammaquadratic} is well defined, i.e~independent of the chosen subspace $\pi$.

\begin{definition}\label{def:gammaquadratic}
    Assume that $n(j-i)$ even or $n-j-i=0$. Consider a quadratic form $f$ of type $\eps$ on $\F^{n}_{q}$, $q$ odd, an let $\pi$ be an $i$-singular $j$-space of type $\delta$. Then, $\gamma_{(i,j,\delta),(n,\eps),(k,\zeta),\eta}$ is the number of non-singular $k$-spaces $\sigma$ of type $\zeta$ in $\F_q^n$ such that $\sigma\cap\pi$ is trivial and $\langle\pi,\sigma\rangle$ is a non-singular $(k+j)$-space of type $\eta$. When additionally $nk$ is odd (and thus $\eps=\zeta=0$), then $\gamma_{(i,j,\delta),(n,0),(k,0,\mu),\eta}$ is the number of non-singular $k$-spaces $\sigma$ of type $0$ and perp type $\mu$ in $\F_{q^{n}}$ such that $\sigma\cap\pi$ is trivial and $\langle\pi,\sigma\rangle$ is a non-singular $(k+j)$-space of type $\eta$.
\end{definition}

\begin{remark}\label{rem:projectiveinterpretationgamma}
    Using projective dimension and notation, we can say that $\gamma_{(i,j,\delta),(n,\eps),(k,\zeta),\eta}$ is the number of projective $(k-1)$-spaces in $\PG(n-1,q)$ meeting $\mathcal{Q}^\eps(n-1,q)$ in a $\mathcal{Q}^\zeta(k-1,q)$, disjoint from a fixed $(j-1)$-space $\pi$ that meets $\mathcal{Q}^\eps(n-1,q)$ in a cone $\Pi_{i-1}\mathcal{Q}^\delta(j-i-1,q)$ and such that the $(k+j-1)$-space spanned by both meets $\mathcal{Q}^\eps(n-1,q)$ in a $\mathcal{Q}^\eta(k+j-1,q)$.
\end{remark}

\begin{remark}
    Note that $\gamma_{(i,j,\delta),(n,\eps),(k,\zeta),\eta}$ is not defined if no $i$-singular $j$-space of type $\delta$ exists with respect to a quadratic form of type $\eps$ on $\F_{q}^{n}$, $q$ odd. So, we must have that $j-i-\delta\equiv 1\equiv n-\eps\pmod{2}$ for $\gamma_{(i,j,\delta),(n,\eps),(k,\zeta),\eta}$ to exist.
\end{remark}

We now introduce an analogue of $\gamma_{(i,j,\delta),(n,\eps),(k,\zeta),\eta}$ for the case where a perp type needs to be taken into account. Given Theorem \ref{th:orbitsPGO},  Definition \ref{def:gammaquadraticodd} is well defined, i.e~independent of the chosen subspace $\pi$.

\begin{definition}\label{def:gammaquadraticodd}
    Consider a quadratic form $f$ of type $0$ on $\F^{n}_{q}$, $q$ odd. Let $\pi$ be an $i$-singular $j$-space $\pi$ of type $0$ and perp type $\lambda$ (necessarily $n(j-i)$ is odd). Then $\gamma_{(i,j,0,\lambda),(n,0),(k,\zeta),\eta}$ is the number of non-singular $k$-spaces $\sigma$ of type $\zeta$ in $\F_{q^{n}}$ such that $\sigma\cap\pi$ is trivial and $\langle\pi,\sigma\rangle$ is a non-singular $(k+j)$-space of type $\eta$. When additionally $k$ is odd (and thus $\zeta=0$), then $\gamma_{(i,j,0,\lambda),(n,0),(k,0,\mu),\eta}$ is the number of non-singular $k$-spaces $\sigma$ of type $0$ and perp type $\mu$ in $\F_{q^{n}}$ such that $\sigma\cap\pi$ is trivial and $\langle\pi,\sigma\rangle$ is a non-singular $(k+j)$-space of type $\eta$.
\end{definition}

The gammas in Definition \ref{def:gammaquadraticodd} can be interpreted in projective space, similar to the interpretation in Remark \ref{rem:projectiveinterpretationgamma}.

\begin{remark}\label{rem:n=i+jgamma}
    The case with $n$ and $j-i$ both odd and $n=i+j$ was discussed earlier in Remarks \ref{rem:n=i+jalpha} and \ref{rem:n=i+jbeta}. Similarly, we now find that $\gamma_{(i,j,0),(i+j,0),(k,\zeta),\eta}=\gamma_{(i,j,0,1),(i+j,0),(k,\zeta),\eta}$ and that  $\gamma_{(i,j,0,-1),(i+j,0),(k,\zeta),\eta}$ is not defined.
\end{remark}

\begin{notation}\label{not:ambiguous}
    In the main theorems we will discuss $\gamma_{(i,j,\delta),(n,\eps),(k,\zeta),\eta}$ and $\gamma_{(i,j,0,\lambda),(n,0),(k,\zeta),\eta}$ simultaneously and use the notation $\gamma_{(i,j,\delta(,\lambda)),(n,\eps),(k,\zeta),\eta}$. Based on the parity of $n(j-i)$, it should always be clear which one is intended. Similarly, we will also use $\beta_{(i,j,\delta(,\lambda)),(n,\eps),(k,\zeta)}$.
\end{notation}

{\bf Structure of the proof of the Main Theorems}

We first deal with the case that $k=n-j$, that is, that the subspaces we are counting are complementary to the fixed $j$-space in $\mathbb{F}_q^n$; in this case we may assume $\eta=\eps$. The general case where $k<n-j$ (Theorem \ref{th:gammageneral}) then follows relatively easily from the expressions in the complementary case $k=n-j$, together with a double counting argument. To tackle the case $k=n-j$, we derive four lemmas with recursion relations on $\gamma_{(i,j,\delta),(n,\eps),(n-j,\zeta),\eps}$ and $\gamma_{(i,j,0,\lambda),(n,0),(n-j,\zeta),0}$. The proof of each of the lemmas is essentially done in the same way; we double count the set of tuples $(\sigma,\tau)$ with $\sigma$ a non-singular hyperplane (possibly of a certain type), and $\tau\subseteq \sigma$ a non-singular $(n-j)$-space of a fixed type disjoint from the given space $\pi$. However, the arguments used to count this quantity strongly depend on the parity of $n,i,j$, so these cases need to be treated separately. The different cases we treat are: $n,j$ odd, $i$ odd (Lemma \ref{lem:recursion3}), $n,j$ odd, $i$ even (Lemma \ref{lem:recursion4}), $n$ even, $j-i$ even (Lemma \ref{lem:recursion1}) and $n$ even, $j-i$ odd (Lemma \ref{lem:recursion2}).   
Overall, the four lemmas cover all cases with $n-j$ even which allows us, for $n-j$ even, to derive the expression for $\gamma_{(i,j,\delta,(\lambda)), (n,\eps),(n-j,\zeta),\eps}$ in Theorem \ref{th:gammaquadratic}, using an induction argument. Using the associated polarity, we are able to use the expressions for $n-j$ even to deal with the case that $n$ is odd and $j$ is even as well. This is done in Theorem \ref{th:gammaquadratic0odd}. This leaves only the case that $n$ is even and $j$ is odd. We deal with this case in Theorem \ref{th:gammaquadratic0even} by using Lemma \ref{lem:recursion1} and \ref{lem:recursion2} for the second time, because at this point we have already derived the expressions for $n$ even and $j$ odd that are occurring in the recursion formulae.

\begin{remark}
    One may wonder why do not simply derive a recursion formula for all possible parities of $n,j,i$, thereby avoiding having to go through the steps explained above. However, there is a conceptual obstacle to this approach: the non-singular $(n-j)$-spaces with respect to a non-degenerate quadratic form on $\F_{q}^{n}$, $n$ and $q$ odd, split in two different orbits (see Theorem \ref{th:orbitsPGO}). The number of non-singular hyperplanes containing a given non-singular $(n-j)$-space (necessary to set up the induction argument) then depends on the perp type of the given non-singular $(n-j)$-space, so it would be necessary to take care of both classes separately. But when trying to take care of both classes separately, one notices that the perp type gets lost when going one dimension down.
\end{remark}

\begin{lemma}\label{lem:recursion3}
    Let $n,j,i$ be integers with $0\leq i\leq j$ and $i+j\leq n$, and with $n$, $j$ and $i$ odd. Let $\delta,\zeta\in\{\pm1\}$. If $j\geq1$, we have
    \begin{align*}
        &\gamma_{(i,j,\delta),(n,0),(n-j,\zeta),0}\beta_{(0,n-j,\zeta),(n,0),(n-1,\zeta)}\\
        &=\frac{1}{2}q^{n-i-1}\left(q^{i}-1\right)\gamma_{(i-1,j-1,\delta),(n-1,\zeta),(n-j,\zeta),\zeta}\\
        &\qquad+\frac{1}{2}q^{\frac{1}{2}n-\frac{3}{2}}\left((q-1)q^{\frac{1}{2}(n-1)-i}-\delta\zeta q^{\frac{1}{2}(j-i)}-\delta(q-1)q^{\frac{1}{2}(n-j-i-1)}+\zeta\right)\\&\qquad\qquad\gamma_{(i,j-1,0),(n-1,\zeta),(n-j,\zeta),\zeta}\\
        &\qquad+\frac{1}{2}q^{\frac{1}{2}(n-j+i-1)}\gamma_{(i+1,j-1,\delta),(n-1,\zeta),(n-j,\zeta),\zeta}\\&\qquad\qquad\left(\left(q^{j-i-1}-1\right)q^{\frac{1}{2}(n-j-i-1)}+\delta\left(q-1\right)q^{\frac{1}{2}n-i-\frac{3}{2}}+\delta\zeta\left(q^{j-i-1}-1\right)+\zeta\left(q-1\right)q^{\frac{1}{2}(j-i)-1}\right).
    \end{align*}
\end{lemma}
\begin{proof}
    Consider a quadratic form $f$ (necessarily of type $0$) on $\F^{n}_{q}$ and a fixed $i$-singular $j$-space $\pi$ of type $\delta$ with respect to it, with $j\geq1$. We denote the $i$-space that is the radical of $\pi$ (the vertex of the cone geometrically) by $\overline{\pi}$. We count the tuples $(\sigma,\tau)$ with $\sigma$ a non-singular hyperplane of type $\zeta$, and $\tau\subseteq\sigma$ a non-singular $(n-j)$-space of type $\zeta$ disjoint from $\pi$. Note that $\langle\pi,\tau\rangle=\F^{n}_{q}$ and thus that the type of the restriction of the quadratic form to $\langle\pi,\tau\rangle$ is $0$. On the one hand there are $\gamma_{(i,j,\delta),(n,0),(n-j,\zeta),0}$ choices for $\tau$, and for each of them $\beta_{(0,n-j,\zeta),(n,0),(n-1,\zeta)}$ corresponding tuples. Now, we look at the non-singular hyperplanes of type $\zeta$. Since we consider hyperplanes containing an $(n-j)$-space disjoint from $\pi$, none of these hyperplanes goes through $\pi$. Therefore, each non-singular hyperplane that contains an $(n-j)$-space disjoint from $\pi$ meets $\pi$ in a $(j-1)$-space. There are three possibilities for a $(j-1)$-space $\pi'$ of $\pi$, depending on whether $\pi'$ is $(i-1)$-, $i$- or $(i+1)$-singular. For each of those possibilities, we will first count the number of such $(j-1)$-spaces $\pi'$, then multiply this number by the number of non-singular hyperplanes $\sigma$ meeting $\pi$ exactly in $\pi'$, and finally, multiply by the number of choices for a subspace $\tau$ disjoint from $\pi'$ (and thus from $\pi$) in $\sigma$.
	\begin{itemize}
		\item We first look at a $(j-1)$-space $\pi'$ which meets $\overline{\pi}$ in precisely an $(i-1)$-space. Necessarily, $\pi'$ is $(i-1)$-singular. There are $\gs{j}{j-1}{q}-\gs{j-i}{j-i-1}{q}=\frac{q^{j}-q^{j-i}}{q-1}$ choices for $\pi'$. Through $\pi'$ there are $\beta_{(i-1,j-1,\delta),(n,0),(n-1,\zeta)}$ non-singular hyperplanes of type $\zeta$, of which $\beta_{(i,j,\delta),(n,0),(n-1,\zeta)}$ contain $\pi$. Therefore, for every $\pi'$ we find $\beta_{(i-1,j-1,\delta),(n,0),(n-1,\zeta)}-\beta_{(i,j,\delta),(n,0),(n-1,\zeta)}$ admissible hyperplanes $\sigma$. In each such hyperplane $\sigma$, we have $\gamma_{(i-1,j-1,\delta),(n-1,\zeta),(n-j,\zeta),\zeta}$ choices for $\tau$.
        \item Secondly, we look at a $(j-1)$-space $\pi'$ through $\overline{\pi}$ that is $i$-singular. The number of choices for $\pi'$ is precisely the number of non-singular $(j-i-1)$-spaces (necessarily of type 0) with respect to a non-singular quadratic form of type $\delta$ on a $(j-i)$-space, which is $\alpha_{(0,j-i-1,0),(j-i,\delta)}$. Such a space $\pi'$ has a perp type $\lambda$ with respect to $f$. It follows from Corollary \ref{cor:halvedorbits} that precisely half of the $i$-singular $(j-1)$-spaces in $\pi$ have perp type $1$ and half of them have perp type $-1$.
        \par Through $\pi'$ there are $\beta_{(i,j-1,0,\lambda),(n,0),(n-1,\zeta)}$ non-singular hyperplanes of type $\zeta$, of which $\beta_{(i,j,\delta),(n,0),(n-1,\zeta)}$ contain $\pi$. Therefore, for every $i$-singular $(j-1)$-space $\pi'$ with perp type $\lambda$ we find $\beta_{(i,j-1,0,\lambda),(n,0),(n-1,\zeta)}-\beta_{(i,j,\delta),(n,0),(n-1,\zeta)}$ admissible hyperplanes $\sigma$. In each such hyperplane $\sigma$, we have $\gamma_{(i,j-1,0),(n-1,\zeta),(n-j,\zeta),\zeta}$ choices for $\tau$.
        \item Finally, we look at a $(j-1)$-space $\pi'$ through $\overline{\pi}$ that is $(i+1)$-singular. The number of choices for $\pi'$ is the number of $1$-singular $(j-i-1)$-spaces (necessarily of type $\delta$) with respect to a non-singular quadratic form of type $\delta$ on a $(j-i)$-space, which is given by $\alpha_{(1,j-i-1,\delta),(j-i,\delta)}$. Through $\pi'$ there are $\beta_{(i+1,j-1,\delta),(n,0),(n-1,\zeta)}-\beta_{(i,j,\delta),(n,0),(n-1,\zeta)}$ non-singular hyperplanes $\sigma$ of type $\zeta$ not containing $\pi$. The number of choices for $\tau$ in $\sigma$ is $\gamma_{(i+1,j-1,\delta),(n-1,\zeta),(n-j,\zeta),\zeta}$.
	\end{itemize}
We find the following result, which we simplify using Corollaries \ref{cor:alphaorthogonal}, \ref{cor:betas} and \ref{cor:betavariations}:
    \begin{align*}
        &\gamma_{(i,j,\delta),(n,0),(n-j,\zeta),0}\beta_{(0,n-j,\zeta),(n,0),(n-1,\zeta)}\\
        &=\frac{q^{j}-q^{j-i}}{q-1}\left(\beta_{(i-1,j-1,\delta),(n,0),(n-1,\zeta)}-\beta_{(i,j,\delta),(n,0),(n-1,\zeta)}\right)\gamma_{(i-1,j-1,\delta),(n-1,\zeta),(n-j,\zeta),\zeta}\\
        &\qquad+\frac{1}{2}\alpha_{(0,j-i-1,0),(j-i,\delta)}\\&\qquad\quad\sum_{\lambda\in\{\pm1\}}\left(\beta_{(i,j-1,0,\lambda),(n,0),(n-1,\zeta)}-\beta_{(i,j,\delta),(n,0),(n-1,\zeta)}\right)\gamma_{(i,j-1,0),(n-1,\zeta),(n-j,\zeta),\zeta}\\
        &\qquad+\alpha_{(1,j-i-1,\delta),(j-i,\delta)}\left(\beta_{(i+1,j-1,\delta),(n,0),(n-1,\zeta)}-\beta_{(i,j,\delta),(n,0),(n-1,\zeta)}\right)\gamma_{(i+1,j-1,\delta),(n-1,\zeta),(n-j,\zeta),\zeta}\\
        &=q^{j-i}\frac{q^{i}-1}{q-1}\gamma_{(i-1,j-1,\delta),(n-1,\zeta),(n-j,\zeta),\zeta}\\&\qquad\qquad\left(\frac{1}{2}q^{\frac{1}{2}(n-j+i-1)}\left(q^{\frac{1}{2}(n-j-i+1)}+\delta\zeta\right)-\frac{1}{2}q^{\frac{1}{2}(n-j+i-1)}\left(q^{\frac{1}{2}(n-j-i-1)}+\delta\zeta\right)\right)\\
        &\qquad+\frac{1}{2}q^{\frac{1}{2}(j-i)-1}\left(q^{\frac{1}{2}(j-i)}-\delta\right)\gamma_{(i,j-1,0),(n-1,\zeta),(n-j,\zeta),\zeta}\\&\qquad\qquad\sum_{\lambda\in\{\pm1\}}\left(\frac{1}{2}q^{\frac{1}{2}(n-j+i-1)}\left(q^{\frac{1}{2}(n-j-i+1)}-\lambda\right)-\frac{1}{2}q^{\frac{1}{2}(n-j+i-1)}\left(q^{\frac{1}{2}(n-j-i-1)}+\delta\zeta\right)\right)\\
        &\qquad+\left(\frac{q^{j-i-1}-1}{q-1}+\delta q^{\frac{1}{2}(j-i)-1}\right)\gamma_{(i+1,j-1,\delta),(n-1,\zeta),(n-j,\zeta),\zeta}\\&\qquad\qquad\left(\frac{1}{2}q^{\frac{1}{2}(n-j+i+1)}\left(q^{\frac{1}{2}(n-j-i-1)}+\delta\zeta\right)-\frac{1}{2}q^{\frac{1}{2}(n-j+i-1)}\left(q^{\frac{1}{2}(n-j-i-1)}+\delta\zeta\right)\right)\\
        &=\frac{1}{2}q^{n-i-1}\left(q^{i}-1\right)\gamma_{(i-1,j-1,\delta),(n-1,\zeta),(n-j,\zeta),\zeta}\\
        &\qquad+\frac{1}{2}q^{\frac{1}{2}n-\frac{3}{2}}\left(q^{\frac{1}{2}(j-i)}-\delta\right)\left((q-1)q^{\frac{1}{2}(n-j-i-1)}-\delta\zeta\right)\gamma_{(i,j-1,0),(n-1,\zeta),(n-j,\zeta),\zeta}\\
        &\qquad+\frac{1}{2}q^{\frac{1}{2}(n-j+i-1)}\left(q^{j-i-1}-1+\delta\left(q-1\right)q^{\frac{1}{2}(j-i)-1}\right)\gamma_{(i+1,j-1,\delta),(n-1,\zeta),(n-j,\zeta),\zeta}\\&\qquad\qquad\left(q^{\frac{1}{2}(n-j-i-1)}+\delta\zeta\right)\\
        &=\frac{1}{2}q^{n-i-1}\left(q^{i}-1\right)\gamma_{(i-1,j-1,\delta),(n-1,\zeta),(n-j,\zeta),\zeta}\\
        &\qquad+\frac{1}{2}q^{\frac{1}{2}n-\frac{3}{2}}\left((q-1)q^{\frac{1}{2}(n-1)-i}-\delta\zeta q^{\frac{1}{2}(j-i)}-\delta(q-1)q^{\frac{1}{2}(n-j-i-1)}+\zeta\right)\\&\qquad\qquad\gamma_{(i,j-1,0),(n-1,\zeta),(n-j,\zeta),\zeta}\\
        &\qquad+\frac{1}{2}q^{\frac{1}{2}(n-j+i-1)}\gamma_{(i+1,j-1,\delta),(n-1,\zeta),(n-j,\zeta),\zeta}\\&\qquad\qquad\left(\left(q^{j-i-1}-1\right)q^{\frac{1}{2}(n-j-i-1)}+\delta\left(q-1\right)q^{\frac{1}{2}n-i-\frac{3}{2}}+\delta\zeta\left(q^{j-i-1}-1\right)+\zeta\left(q-1\right)q^{\frac{1}{2}(j-i)-1}\right).
    \end{align*}
    Note that this equality is also valid if $i=j$ (and thus $\delta=1)$. Then, only one of the three possible cases for a $(j-1)$-space in $\pi$ occurs. But the cases that do not appear, indeed correspond to terms with a factor $0$ in the given equation.
\end{proof}

\begin{lemma}\label{lem:recursion4}
    Let $n,j,i$ be integers with $0\leq i\leq j$ and $i+j\leq n$, and with $n$ and $j$ odd, and $i$ even. Let $\lambda,\zeta\in\{\pm1\}$. If $j\geq1$, we have
    \begin{align*}
        &\gamma_{(i,j,0,\lambda),(n,0),(n-j,\zeta),0}\beta_{(0,n-j,\zeta),(n,0),(n-1,\zeta)}\\
        &=\frac{1}{2}q^{n-i-1}\left(q^{i}-1\right)\gamma_{(i-1,j-1,0),(n-1,\zeta),(n-j,\zeta),\zeta}\\
        &\qquad+\frac{1}{4}q^{\frac{1}{2}n-\frac{3}{2}}\sum_{\kappa\in\{\pm1\}}\gamma_{(i,j-1,\kappa),(n-1,\zeta),(n-j,\zeta),\zeta}\\&\qquad\qquad\left((q-1)q^{\frac{1}{2}(n-1)-i}+\lambda q^{\frac{1}{2}(j-i-1)}+\zeta\kappa q^{\frac{1}{2}(j-i+1)}+\kappa(q-1)q^{\frac{1}{2}(n-j-i)}+\kappa\lambda+\zeta q\right)\\
        &\qquad+\frac{1}{2}q^{\frac{1}{2}(n-j+i)-1}\left(q^{j-i-1}-1\right)\left(q^{\frac{1}{2}(n-j-i)}-\lambda\right)\gamma_{(i+1,j-1,0),(n-1,\zeta),(n-j,\zeta),\zeta}\:.
    \end{align*}
\end{lemma}
\begin{proof}
    Consider a quadratic form $f$ (necessarily of type $0$) on $\F^{n}_{q}$ and a fixed $i$-singular $j$-space $\pi$ of type $0$ and perp type $\lambda$ with respect to it, with $j\geq1$. We denote the $i$-space that is the radical of $\pi$ (the vertex of the cone geometrically) by $\overline{\pi}$. We count the tuples $(\sigma,\tau)$ with $\sigma$ a non-singular hyperplane of type $\zeta$, and $\tau\subseteq\sigma$ a non-singular $(n-j)$-space of type $\zeta$ disjoint from $\pi$. Note that $\langle\pi,\tau\rangle=\F^{n}_{q}$ and thus that the type of the restriction of the quadratic form to $\langle\pi,\tau\rangle$ is $0$. On the one hand there are $\gamma_{(i,j,0,\lambda),(n,0),(n-j,\zeta),0}$ choices for $\tau$, and for each of them $\beta_{(0,n-j,\zeta),(n,0),(n-1,\zeta)}$ corresponding tuples. Now, we look at the non-singular hyperplanes of type $\zeta$. Since we consider hyperplanes containing an $(n-j)$-space disjoint from $\pi$, none of these hyperplanes goes through $\pi$. Therefore, each non-singular hyperplane that contains an $(n-j)$-space disjoint from $\pi$ meets $\pi$ in a $(j-1)$-space. There are three possibilities for a $(j-1)$-space $\pi'$ of $\pi$, depending on whether $\pi'$ is $(i-1)$-, $i$- or $(i+1)$-singular. For each of those possibilities, we will first count the number of such $(j-1)$-spaces $\pi'$, then multiply this number by the number of non-singular hyperplanes $\sigma$ meeting $\pi$ exactly in $\pi'$, and finally, multiply by the number of choices for a subspace $\tau$ disjoint from $\pi'$ (and thus from $\pi$) in $\sigma$.
	\begin{itemize}
		\item We first look at a $(j-1)$-space $\pi'$ which meets $\overline{\pi}$ in precisely an $(i-1)$-space. Necessarily, $\pi'$ is $(i-1)$-singular. There are $\gs{j}{j-1}{q}-\gs{j-i}{j-i-1}{q}=\frac{q^{j}-q^{j-i}}{q-1}$ choices for $\pi'$. Note that $\pi'$ has type 0 and perp type $\lambda$. Through $\pi'$ there are $\beta_{(i-1,j-1,0,\lambda),(n,0),(n-1,\zeta)}$ non-singular hyperplanes of type $\zeta$, of which $\beta_{(i,j,0,\lambda),(n,0),(n-1,\zeta)}$ contain $\pi$. Therefore, for every $\pi'$ we find $\beta_{(i-1,j-1,0,\lambda),(n,0),(n-1,\zeta)}-\beta_{(i,j,0,\lambda),(n,0),(n-1,\zeta)}$ admissible hyperplanes $\sigma$. In each such hyperplane $\sigma$, we have $\gamma_{(i-1,j-1,0),(n-1,\zeta),(n-j,\zeta),\zeta}$ choices for $\tau$.
        \item Secondly, we look at a $(j-1)$-space $\pi'$ through $\overline{\pi}$ that is $i$-singular. Such a space $\pi'$ has a type $\kappa\in\{\pm1\}$. The number of choices for $\pi'$ depends on $\kappa$, and is precisely the number of non-singular $(j-i-1)$-spaces of type $\kappa$ with respect to a non-singular quadratic form of type 0 on a $(j-i)$-space, which is $\alpha_{(0,j-i-1,\kappa),(j-i,0)}$. Through $\pi'$ there are $\beta_{(i,j-1,\kappa),(n,0),(n-1,\zeta)}$ non-singular hyperplanes of type $\zeta$, of which $\beta_{(i,j,0,\lambda),(n,0),(n-1,\zeta)}$ contain $\pi$. Therefore, for every $\pi'$ with perp type $\lambda$ we find $\beta_{(i,j-1,\kappa),(n,0),(n-1,\zeta)}-\beta_{(i,j,0,\lambda),(n,0),(n-1,\zeta)}$ admissible hyperplanes $\sigma$. In each such hyperplane $\sigma$, we have $\gamma_{(i,j-1,\kappa),(n-1,\zeta),(n-j,\zeta),\zeta}$ choices for $\tau$.
        \item Finally, we look at a $(j-1)$-space $\pi'$ through $\overline{\pi}$ that is $(i+1)$-singular. Note that $\pi'$ has type 0 and perp type $\lambda$. The number of choices for $\pi'$ is the number of $1$-singular $(j-i-1)$-spaces (necessarily of type 0 and perp type $\lambda$) with respect to a non-singular quadratic form of type $0$ on a $(j-i)$-space, which is given by $\alpha_{(1,j-i-1,0),(j-i,0)}$. Through $\pi'$ there are $\beta_{(i+1,j-1,0,\lambda),(n,0),(n-1,\zeta)}-\beta_{(i,j,0,\lambda),(n,0),(n-1,\zeta)}$ non-singular hyperplanes $\sigma$ of type $\zeta$ not containing $\pi$. The number of choices for $\tau$ in $\sigma$ is $\gamma_{(i+1,j-1,0),(n-1,\zeta),(n-j,\zeta),\zeta}$.
	\end{itemize}
   	We find the following result, which we simplify using Corollaries \ref{cor:alphaorthogonal}, \ref{cor:betas} and \ref{cor:betavariations}:
    \begin{align*}
        &\gamma_{(i,j,0,\lambda),(n,0),(n-j,\zeta),0}\beta_{(0,n-j,\zeta),(n,0),(n-1,\zeta)}\\
        &=\frac{q^{j}-q^{j-i}}{q-1}\gamma_{(i-1,j-1,0),(n-1,\zeta),(n-j,\zeta),\zeta}\left(\beta_{(i-1,j-1,0,\lambda),(n,0),(n-1,\zeta)}-\beta_{(i,j,0,\lambda),(n,0),(n-1,\zeta)}\right)\\
        &\qquad+\sum_{\kappa\in\{\pm1\}}\alpha_{(0,j-i-1,\kappa),(j-i,0)}\gamma_{(i,j-1,\kappa),(n-1,\zeta),(n-j,\zeta),\zeta}\\&\qquad\qquad\qquad\left(\beta_{(i,j-1,\kappa),(n,0),(n-1,\zeta)}-\beta_{(i,j,0,\lambda),(n,0),(n-1,\zeta)}\right)\\
        &\qquad+\alpha_{(1,j-i-1,0),(j-i,0)}\gamma_{(i+1,j-1,0),(n-1,\zeta),(n-j,\zeta),\zeta}\\&\qquad\qquad\qquad\left(\beta_{(i+1,j-1,0,\lambda),(n,0),(n-1,\zeta)}-\beta_{(i,j,0,\lambda),(n,0),(n-1,\zeta)}\right)\\
        &=q^{j-i}\frac{q^{i}-1}{q-1}\gamma_{(i-1,j-1,0),(n-1,\zeta),(n-j,\zeta),\zeta}\\&\qquad\qquad\left(\frac{1}{2}q^{\frac{1}{2}(n-j+i)-1}\left(q^{\frac{1}{2}(n-j-i)+1}-\lambda\right)-\frac{1}{2}q^{\frac{1}{2}(n-j+i)-1}\left(q^{\frac{1}{2}(n-j-i)}-\lambda\right)\right)\\
        &\qquad+\sum_{\kappa\in\{\pm1\}}\frac{1}{2}q^{\frac{1}{2}(j-i-1)}\left(q^{\frac{1}{2}(j-i-1)}+\kappa\right)\gamma_{(i,j-1,\kappa),(n-1,\zeta),(n-j,\zeta),\zeta}\\&\qquad\qquad\left(\frac{1}{2}q^{\frac{1}{2}(n-j+i)}\left(q^{\frac{1}{2}(n-j-i)}+\kappa\zeta\right)-\frac{1}{2}q^{\frac{1}{2}(n-j+i)-1}\left(q^{\frac{1}{2}(n-j-i)}-\lambda\right)\right)\\
        &\qquad+\frac{q^{j-i-1}-1}{q-1}\gamma_{(i+1,j-1,0),(n-1,\zeta),(n-j,\zeta),\zeta}\\&\qquad\qquad\left(\frac{1}{2}q^{\frac{1}{2}(n-j+i)}\left(q^{\frac{1}{2}(n-j-i)}-\lambda\right)-\frac{1}{2}q^{\frac{1}{2}(n-j+i)-1}\left(q^{\frac{1}{2}(n-j-i)}-\lambda\right)\right)\\
        &=\frac{1}{2}q^{n-i-1}\left(q^{i}-1\right)\gamma_{(i-1,j-1,0),(n-1,\zeta),(n-j,\zeta),\zeta}\\
        &\qquad+\frac{1}{4}q^{\frac{1}{2}n-\frac{3}{2}}\sum_{\kappa\in\{\pm1\}}\gamma_{(i,j-1,\kappa),(n-1,\zeta),(n-j,\zeta),\zeta}\left(q^{\frac{1}{2}(j-i-1)}+\kappa\right)\left((q-1)q^{\frac{1}{2}(n-j-i)}+\zeta\kappa q+\lambda\right)\\
        &\qquad+\frac{1}{2}q^{\frac{1}{2}(n-j+i)-1}\left(q^{j-i-1}-1\right)\left(q^{\frac{1}{2}(n-j-i)}-\lambda\right)\gamma_{(i+1,j-1,0),(n-1,\zeta),(n-j,\zeta),\zeta}\\
        &=\frac{1}{2}q^{n-i-1}\left(q^{i}-1\right)\gamma_{(i-1,j-1,0),(n-1,\zeta),(n-j,\zeta),\zeta}\\
        &\qquad+\frac{1}{4}q^{\frac{1}{2}n-\frac{3}{2}}\sum_{\kappa\in\{\pm1\}}\gamma_{(i,j-1,\kappa),(n-1,\zeta),(n-j,\zeta),\zeta}\\&\qquad\qquad\left((q-1)q^{\frac{1}{2}(n-1)-i}+\lambda q^{\frac{1}{2}(j-i-1)}+\zeta\kappa q^{\frac{1}{2}(j-i+1)}+\kappa(q-1)q^{\frac{1}{2}(n-j-i)}+\kappa\lambda+\zeta q\right)\\
        &\qquad+\frac{1}{2}q^{\frac{1}{2}(n-j+i)-1}\left(q^{j-i-1}-1\right)\left(q^{\frac{1}{2}(n-j-i)}-\lambda\right)\gamma_{(i+1,j-1,0),(n-1,\zeta),(n-j,\zeta),\zeta}\:.
    \end{align*}
   Note that this equality is also valid if $i=j-1$ and if $i=0$. Then (in both instances), only two of the three possible cases for a $(j-1)$-space in $\pi$ occur. But the cases that do not appear, indeed correspond to terms with a factor $0$ in the given equation.
\end{proof}

\begin{lemma}\label{lem:recursion1}
    Let $n,j,i$ be integers with $0\leq i\leq j$ and $i+j\leq n$, and with $n$ and $j-i$ both even. Let $\delta,\eps\in\{\pm1\}$ and $\zeta\in\{0,\pm1\}$. 
    If $j\geq1$, we have
    \begin{align*}
        &\gamma_{(i,j,\delta),(n,\eps),(n-j,\zeta),\eps}\beta_{(0,n-j,\zeta),(n,\eps),(n-1,0)}\\
        &=\left(q^{i}-1\right)q^{n-i-1}\gamma_{(i-1,j-1,\delta),(n-1,0),(n-j,\zeta),0}\\
        &\qquad+\frac{1}{2}q^{\frac{1}{2}n-2}\left(q^{\frac{1}{2}(j-i)}-\delta\right)\sum_{\nu\in\{\pm1\}}\left((q-1)q^{\frac{1}{2}(n-j-i)}+\nu q+\delta\eps\right)\gamma_{(i,j-1,0,\nu),(n-1,0),(n-j,\zeta),0}\\
        &\qquad+q^{\frac{1}{2}(n-j+i)-1}\left(q^{\frac{1}{2}(j-i-1-\delta)}+1\right)\left(q^{\frac{1}{2}(j-i-1+\delta)}-1\right)\left(q^{\frac{1}{2}(n-j-i)}-\delta\eps\right)\\&\qquad\qquad\qquad\gamma_{(i+1,j-1,\delta),(n-1,0),(n-j,\zeta),0}\:.
    \end{align*}
\end{lemma}
\begin{proof}
    Consider a quadratic form $f$ of type $\eps$ on $\F^{n}_{q}$ and a fixed $i$-singular $j$-space $\pi$ of type $\delta$ with respect to it, with $j\geq1$. We denote the $i$-space that is the radical of $\pi$ (the vertex of the cone geometrically) by $\overline{\pi}$. We count the tuples $(\sigma,\tau)$ with $\sigma$ a non-singular hyperplane, and $\tau\subseteq\sigma$ a non-singular $(n-j)$-space of type $\zeta$ disjoint from $\pi$. Note that $\sigma$ is necessarily of parabolic type since $n$ is even. Also, note that $\langle\pi,\tau\rangle=\F^{n}_{q}$ and thus that the type of the restriction of the quadratic form to $\langle\pi,\tau\rangle$ is $\eps$. On the one hand there are $\gamma_{(i,j,\delta),(n,\eps),(n-j,\zeta),\eps}$ choices for $\tau$, and for each of them $\beta_{(0,n-j,\zeta),(n,\eps),(n-1,0)}$ corresponding tuples.
    Now, we look at the non-singular hyperplanes. Since we consider hyperplanes $\sigma$ containing an $(n-j)$-space disjoint from $\pi$, none of these hyperplanes goes through $\pi$. Therefore, each non-singular hyperplane that contains an $(n-j)$-space disjoint from $\pi$ meets $\pi$ in a $(j-1)$-space. There are three possibilities for such a $(j-1)$-space $\pi'$ of $\pi$, depending on whether $\pi'$ is $(i-1)$-, $i$- or $(i+1)$-singular.
	\begin{itemize}
		\item We first look at a $(j-1)$-space $\pi'$ which meets $\overline{\pi}$ in precisely an $(i-1)$-space. Necessarily, $\pi'$ is $(i-1)$-singular. There are $\gs{j}{j-1}{q}-\gs{j-i}{j-i-1}{q}=\frac{q^{j}-q^{j-i}}{q-1}$ choices for $\pi'$. Through $\pi'$ there are $\beta_{(i-1,j-1,\delta),(n,\eps),(n-1,0)}$ non-singular hyperplanes, of which $\beta_{(i,j,\delta),(n,\eps),(n-1,0)}$ contain $\pi$. Therefore, for every $\pi'$ we find $\beta_{(i-1,j-1,\delta),(n,\eps),(n-1,0)}-\beta_{(i,j,\delta),(n,\eps),(n-1,0)}$ admissible hyperplanes $\sigma$. In each such hyperplane $\sigma$, we have  $\gamma_{(i-1,j-1,\delta),(n-1,0),(n-j,\zeta),0}$ choices for $\tau$.
        \item Secondly, we look at a $(j-1)$-space $\pi'$ through $\overline{\pi}$ that is $i$-singular. The number of choices for $\pi'$ is precisely the number of non-singular $(j-i-1)$-spaces (necessarily of type 0) with respect to a non-singular quadratic form of type $\delta$ on a $(j-i)$-space, which is $\alpha_{(0,j-i-1,0),(j-i,\delta)}$. Given a hyperplane $\sigma$ through $\pi'$ not containing $\pi$, the number of choices for $\tau$ depends on the perp type $\nu$ of $\pi'$ with respect to the quadratic form $f\vert_{\sigma}$: it equals $\gamma_{(i,j-1,0,\nu),(n-1,0),(n-j,\zeta),0}$.
        \par Therefore, we need to count the number of non-singular hyperplanes $\sigma$ through $\pi'$, not containing $\pi$, such that $\pi'$ has perp type $\nu$ with respect to the restricted quadratic form. 
        The number of non-singular hyperplanes through $\pi'$ such that $\pi'$ has perp type $\nu$ with respect to the restricted quadratic form equals $\beta^{\nu}_{(i,j-1,0),(n,\eps),(n-1,0)}$ by Definition \ref{def:betaquadraticodd}. To find the correct number we now need to subtract the number of such hyperplanes containing $\pi$.
        \par We know that $\pi^{\perp}$ is an $i$-singular $(n-j)$-space of type $\delta\eps$ with radical $\overline{\pi}$ and that $\pi'^{\perp}$ is an $i$-singular $(n-j+1)$-space of type 0 with radical $\overline{\pi}$, containing $\pi^\perp$. Note that $\langle\pi,\pi^\perp\rangle=\overline{\pi}^{\perp}=\langle\pi',\pi'^\perp\rangle$ is an $i$-singular $(n-i)$-space.
        \par Let $\sigma$ be a non-singular (necessarily parabolic) hyperplane through $\pi$ in which $\pi'$ has perp type $\nu$, and let $\perp'$ be the orthogonal polarity on $\sigma$ corresponding to $f\vert_{\sigma}$. Note that $\sigma$ does not contain $\langle\pi,\pi^\perp\rangle=\langle\pi',\pi'^\perp\rangle$, since there cannot be an $i$-singular $(n-i)$-space with respect to a non-singular form in $\F^{n-1}_{q}$. Since $\sigma$ contains $\pi'$, the subspace $\sigma\cap\pi'^\perp$ must have dimension $n-j$, and since $\sigma\supseteq\pi$, the subspace $\sigma\cap\pi'^\perp$ contains the $i$-singular $(i+1)$-space $\pi'^{\perp}\cap\pi$ through $\overline{\pi}$. Moreover, $\sigma\cap\pi'^\perp$ must be $i$-singular as it contains $\overline{\pi}$ but its radical cannot have dimension larger than $i$, for if it were then also $\sigma\cap\langle\pi',\pi'^{\perp}\rangle$ would have a radical of dimension at least $i+1$, which is impossible, since $\dim(\sigma)-\dim(\sigma\cap\langle\pi',\pi'^{\perp}\rangle)=i$. Furthermore, we note that $\pi'^{\perp'}=\sigma\cap\pi'^{\perp}$, and so $\nu$ is the type of $\sigma\cap\pi'^{\perp}$.
        \par So, on the one hand, for any non-singular hyperplane through $\pi$ such that $\pi'$ has perp type $\nu$ with respect to it, we find a unique $(n-j)$-space of type $\nu$ through $\pi'^{\perp}\cap\pi\supset\overline{\pi}$ in $\pi'^\perp$. On the other hand, if $\sigma'$ is an $(n-j)$-space of type $\nu$ through $\pi'^{\perp}\cap\pi$ in $\pi'^\perp$, then $\langle\sigma',\pi'\rangle$ is an $i$-singular $(n-i-1)$-space and in any non-singular hyperplane through $\langle\sigma',\pi'\rangle$ the subspace $\pi'$ has perp type $\nu$. Note that $\langle\sigma',\pi'\rangle$ necessarily has type 0.
        \par  So, we first count the number of $i$-singular $(n-j)$-spaces of type $\nu$ in an $i$-singular $(n-j+1)$-space through a fixed $i$-singular $(i+1)$-space. Taking the quotient by the radical, we find that this number is $\beta_{(0,1,0,\delta\eps),(n-j-i+1,0),(n-j-i,\nu)}$. And given such a space $\sigma'$ we now count the number of non-singular hyperplanes through $\sigma'$ also containing $\pi'$ (and then automatically $\pi$). Note that $\langle\sigma',\pi'\rangle$ is an $(n-i-1)$-space in $\langle\pi',\pi'^\perp\rangle=\langle\pi,\pi^\perp\rangle$ containing $\overline{\pi}$ that is $i$-singular. So, we are counting the number of hyperplanes through an $i$-singular $(n-i-1)$-space (necessarily of type 0): this number is $\beta_{(i,n-i-1,0),(n,\eps),(n-1,0)}$.
        \par We conclude that, given an $i$-singular $(j-1)$-space in $\pi$ through $\overline{\pi}$, the number of choices for $\tau$ equals
        \begin{align*}
            &\sum_{\nu\in\{\pm1\}}\left(\beta^{\nu}_{(i,j-1,0),(n,\eps),(n-1,0)}-\beta_{(0,1,0,\delta\eps),(n-j-i+1,0),(n-j-i,\nu)}\beta_{(i,n-i-1,0),(n,\eps),(n-1,0)}\right)\\&\qquad\qquad\qquad\qquad\gamma_{(i,j-1,0,\nu),(n-1,0),(n-j,\zeta),0}\:.
        \end{align*}
		\item Finally, we look at a $(j-1)$-space $\pi'$ through $\overline{\pi}$ that is $(i+1)$-singular. The number of choices for $\pi'$ is the number of $1$-singular $(j-i-1)$-spaces (necessarily of type $\delta$) with respect to a non-singular quadratic form of type $\delta$ on a $(j-i)$-space, which is given by $\alpha_{(1,j-i-1,\delta),(j-i,\delta)}$. Through $\pi'$ there are $\beta_{(i+1,j-1,\delta),(n,\eps),(n-1,0)}-\beta_{(i,j,\delta),(n,\eps),(n-1,0)}$ non-singular hyperplanes $\sigma$ not containing $\pi$. The number of choices for $\tau$ in $\sigma$ is $\gamma_{(i+1,j-1,\delta),(n-1,0),(n-j,\zeta),0}$.
	\end{itemize}
	We find the following result, which we simplify using Corollaries \ref{cor:alphaorthogonal}, \ref{cor:betas} and \ref{cor:betavariations}:
    \begin{align*}
        &\gamma_{(i,j,\delta),(n,\eps),(n-j,\zeta),1}\beta_{(0,n-j,\zeta),(n,\eps),(n-1,0)}\\
        &=\frac{q^{j}-q^{j-i}}{q-1}\left(\beta_{(i-1,j-1,\delta),(n,\eps),(n-1,0)}-\beta_{(i,j,\delta),(n,\eps),(n-1,0)}\right)\gamma_{(i-1,j-1,\delta),(n-1,0),(n-j,\zeta),0}\\
        &\qquad+\alpha_{(0,j-i-1,0),(j-i,\delta)}\sum_{\nu\in\{\pm1\}}\gamma_{(i,j-1,0,\nu),(n-1,0),(n-j,\zeta),0}\\&\qquad\qquad\qquad\qquad\left(\beta^{\nu}_{(i,j-1,0),(n,\eps),(n-1,0)}-\beta_{(0,1,0,\delta\eps),(n-j-i+1,0),(n-j-i,\nu)}\beta_{(i,n-i-1,0),(n,\eps),(n-1,0)}\right)\\
        &\qquad+\alpha_{(1,j-i-1,\delta),(j-i,\delta)}\left(\beta_{(i+1,j-1,\delta),(n,\eps),(n-1,0)}-\beta_{(i,j,\delta),(n,\eps),(n-1,0)}\right)\gamma_{(i+1,j-1,\delta),(n-1,0),(n-j,\zeta),0}\\
        &=\frac{q^{j}-q^{j-i}}{q-1}\left(q^{\frac{1}{2}(n-j+i)-1}\left(q^{\frac{1}{2}(n-j-i)+1}-\delta\eps\right)-q^{\frac{1}{2}(n-j+i)-1}\left(q^{\frac{1}{2}(n-j-i)}-\delta\eps\right)\right)\\&\qquad\qquad\qquad\gamma_{(i-1,j-1,\delta),(n-1,0),(n-j,\zeta),0}\\
        &\qquad+q^{\frac{1}{2}(j-i)-1}\left(q^{\frac{1}{2}(j-i)}-\delta\right)\sum_{\nu\in\{\pm1\}}\gamma_{(i,j-1,0,\nu),(n-1,0),(n-j,\zeta),0}\\&\qquad\qquad\qquad\left(\frac{1}{2}q^{\frac{1}{2}(n-j+i)}\left(q^{\frac{1}{2}(n-j-i)}+\nu\right)-q^{i}\frac{1}{2}q^{\frac{1}{2}(n-j-i)-1}\left(q^{\frac{1}{2}(n-j-i)}-\delta\eps\right)\right)\\
        &\qquad+\frac{1}{q-1}\left(q^{\frac{1}{2}(j-i-1-\delta)}+1\right)\left(q^{\frac{1}{2}(j-i-1+\delta)}-1\right)\gamma_{(i+1,j-1,\delta),(n-1,0),(n-j,\zeta),0}\\&\qquad\qquad\qquad\left(q^{\frac{1}{2}(n-j+i)}\left(q^{\frac{1}{2}(n-j-i)}-\delta\eps\right)-q^{\frac{1}{2}(n-j+i)-1}\left(q^{\frac{1}{2}(n-j-i)}-\delta\eps\right)\right)\\
        &=\left(q^{i}-1\right)q^{n-i-1}\gamma_{(i-1,j-1,\delta),(n-1,0),(n-j,\zeta),0}\\
        &\qquad+\frac{1}{2}q^{\frac{1}{2}n-2}\left(q^{\frac{1}{2}(j-i)}-\delta\right)\sum_{\nu\in\{\pm1\}}\left((q-1)q^{\frac{1}{2}(n-j-i)}+\nu q+\delta\eps\right)\gamma_{(i,j-1,0,\nu),(n-1,0),(n-j,\zeta),0}\\
        &\qquad+q^{\frac{1}{2}(n-j+i)-1}\left(q^{\frac{1}{2}(j-i-1-\delta)}+1\right)\left(q^{\frac{1}{2}(j-i-1+\delta)}-1\right)\left(q^{\frac{1}{2}(n-j-i)}-\delta\eps\right)\\&\qquad\qquad\qquad\gamma_{(i+1,j-1,\delta),(n-1,0),(n-j,\zeta),0}\:.
    \end{align*}
    Note that this equality is also valid if $i=j$ (and thus $\delta=1$) and if $i=0$. Then, only one or two of the three possible cases for a $(j-1)$-space in $\pi$ occur, respectively. But the cases that do not appear, indeed correspond to terms with a factor $0$ in the given equation.
\end{proof}

\begin{lemma}\label{lem:recursion2}
    Let $n,j,i$ be integers with $0\leq i\leq j$ and $i+j\leq n$, and with $n$ even and $j-i$ odd. Let $\delta,\eps\in\{\pm1\}$ and $\zeta\in\{0,\pm1\}$. If $j\geq1$, we have
    \begin{align*}
        &\gamma_{(i,j,0),(n,\eps),(n-j,\zeta),\eps}\beta_{(0,n-j,\zeta),(n,\eps),(n-1,0)}\\
        &=\frac{1}{2}q^{n-i-1}\left(q^{i}-1\right)\sum_{\nu\in\{\pm1\}}\gamma_{(i-1,j-1,0,\nu),(n-1,0),(n-j,\zeta),0}\\
        &\qquad+\frac{1}{2}q^{\frac{1}{2}n-1}\sum_{\kappa\in\{\pm1\}}\left(q^{\frac{1}{2}n-i-1}(q-1)-\eps+\kappa q^{\frac{1}{2}(j-i-1)}\left((q-1)q^{\frac{1}{2}n-j}-\eps\right)\right)\\&\qquad\qquad\qquad\qquad\qquad\gamma_{(i,j-1,\kappa),(n-1,0),(n-j,\zeta),0}\\
        &\qquad+\frac{1}{2}\left(q^{j-i-1}-1\right)q^{\frac{1}{2}(n-j+i-1)}\sum_{\nu\in\{\pm1\}}\gamma_{(i+1,j-1,0,\nu),(n-1,0),(n-j,\zeta),0}\left(q^{\frac{1}{2}(n-j-i-1)}+\nu\right)\:.
    \end{align*}
\end{lemma}
\begin{proof}
    Consider a quadratic form $f$ of type $\eps$ on $\F^{n}_{q}$ and a fixed $i$-singular $j$-space $\pi$ with respect to it, with $j\geq1$. Let $\perp$ be the orthogonal polarity corresponding to $f$. The subspace $\pi$ has type 0 since $j-i$ is odd. We denote the $i$-space that is the radical of $\pi$ (the vertex of the cone geometrically) by $\overline{\pi}$. We count the tuples $(\sigma,\tau)$ with $\sigma$ a non-singular hyperplane, and $\tau\subseteq\sigma$ a non-singular $(n-j)$-space of type $\zeta$ disjoint from $\pi$. Note that $\sigma$ is necessarily of parabolic type. Also, note that $\langle\pi,\tau\rangle=\F^{n}_{q}$ and thus that the type of the restriction of the quadratic form to $\langle\pi,\tau\rangle$ is $\eps$. On the one hand there are $\gamma_{(i,j,0),(n,\eps),(n-j,\zeta),\eps}$ choices for $\tau$, and for each of them $\beta_{(0,n-j,\zeta),(n,\eps),(n-1,0)}$ corresponding tuples. Now, we look at the non-singular hyperplanes. Each non-singular hyperplane that contains an $(n-j)$-space disjoint from $\pi$ meets $\pi$ in a $(j-1)$-space, so we look at the $(j-1)$-spaces of $\pi$. There are three possibilities.
	\begin{itemize}
		\item We first look at the $(j-1)$-spaces that meet $\overline{\pi}$ in precisely an $(i-1)$-space. Necessarily, they are $(i-1)$-singular. The number of such $(j-1)$-spaces equals $\gs{j}{j-1}{q}-\gs{j-i}{j-i-1}{q}=\frac{q^{j}-q^{j-i}}{q-1}$. Let $\pi'$ be such a $(j-1)$-space. Given a hyperplane $\sigma$ through $\pi'$ not containing $\pi$, the number of tuples depends on the perp type $\nu$ of $\pi'$ with respect to the quadratic form $f\vert_{\sigma}$: it equals $\gamma_{(i-1,j-1,0,\nu),(n-1,0),(n-j,\zeta),0}$.
        \par We now count the number of non-singular hyperplanes through $\pi'$, not containing $\pi$, such that $\pi'$ has perp type $\nu$ with respect to the restricted quadratic form. The number of non-singular hyperplanes through $\pi'$ such that $\pi'$ has perp type $\nu$ with respect to the restricted quadratic form equals $\beta^{\nu}_{(i-1,j-1,0),(n,\eps),(n-1,0)}$ by Definition \ref{def:betaquadraticodd}. To find the correct number we now need to subtract the number of such hyperplanes containing $\pi$.
        \par Set $\pi'\cap\overline{\pi}=\overline{\pi}'$. We know that $\pi^{\perp}$ is an $i$-singular $(n-j)$-space of type 0 with radical $\overline{\pi}$ and that $\pi'^{\perp}$ is an $(i-1)$-singular $(n-j+1)$-space of type 0 with radical $\overline{\pi}'$, containing $\pi^\perp$. Note that $\langle\pi,\pi^\perp\rangle=\overline{\pi}^{\perp}$ is an $i$-singular $(n-i)$-space and that $\langle\pi',\pi'^\perp\rangle=\overline{\pi}'^{\perp}$ is an $(i-1)$-singular $(n-i+1)$-space, containing $\overline{\pi}^{\perp}$.
        \par Let $\sigma$ be a non-singular (necessarily parabolic) hyperplane through $\pi$ in which $\pi'$ has perp type $\nu$, and let $\perp'$ be the orthogonal polarity on $\sigma$ corresponding to $f\vert_{\sigma}$. Note that $\sigma$ contains neither $\langle\pi,\pi^\perp\rangle$ nor $\langle\pi',\pi'^\perp\rangle$, since a there cannot be an $i$-singular $(n-i)$-space or $(i-1)$-singular $(n-i+1)$-space with respect to a non-singular form in $\F^{n-1}_{q}$. Since $\sigma$ contains $\pi'$, the subspace $\sigma\cap\pi'^\perp$ must have dimension $n-j$, and be different from $\pi^\perp$. Moreover, $\sigma\cap\pi'^\perp$ must be $(i-1)$-singular as it contains $\overline{\pi}'$ but its radical cannot have dimension larger than $i-1$, for if it were than also $\sigma\cap\langle\pi',\pi'^{\perp}\rangle$ would have a radical of dimension at least $i$, which is impossible. Furthermore, we note that $\pi'^{\perp'}=\sigma\cap\pi'^{\perp}$, and so $\nu$ is the type of $\sigma\cap\pi'^{\perp}$.
        \par So, we first count the number of $(i-1)$-singular $(n-j)$-spaces of type $\nu$ in an $(i-1)$-singular $(n-j+1)$-space through a totally singular $i$-space containing the radical. Taking the quotient by the radical, we find that this number is $\beta_{(1,1,1),(n-j-i+2,0),(n-j-i+1,\nu)}$. And given such a space $\sigma'$ we now count the number of non-singular hyperplanes also containing $\pi$. Note that $\langle\sigma',\pi\rangle$ is an $(n-i)$-space in $\langle\pi',\pi'^\perp\rangle$ containing $\overline{\pi}'$ and so its radical has dimension at least $i-1$; if $\langle\sigma',\pi\rangle$ were $i$-singular, then $\overline{\pi}$ has to be its radical, but that is impossible since $\sigma'$ is $(i-1)$-singular. So, we are counting the number of hyperplanes through an $(i-1)$-singular $(n-i)$-space (necessarily of type 0): this number is $\beta_{(i-1,n-i,0),(n,\eps),(n-1,0)}$.
        \par So, given a $(j-1)$-space in $\pi$ that meets $\overline{\pi}$ in an $(i-1)$-space, the number of choices for $\tau$ equals
        \begin{align*}
            &\sum_{\nu\in\{\pm1\}}\left(\beta^{\nu}_{(i-1,j-1,0),(n,\eps),(n-1,0)}-\beta_{(1,1,1),(n-j-i+2,0),(n-j-i+1,\nu)}\beta_{(i-1,n-i,0),(n,\eps),(n-1,0)}\right)\\&\qquad\qquad\qquad\qquad\gamma_{(i-1,j-1,0,\nu),(n-1,0),(n-j,\zeta),0}\:.
        \end{align*}
        \item Secondly, we look at the $(j-1)$-spaces through $\overline{\pi}$ that are $i$-singular. Such a space has a type $\kappa\in\{\pm1\}$. The number of $i$-singular $(j-1)$-spaces of type $\kappa$ through $\overline{\pi}$ corresponds to the number of non-singular $(j-i-1)$-spaces of type $\kappa$ with respect to a non-singular quadratic form of type $0$ on a $(j-i)$-space, which is $\alpha_{(0,j-i-1,\kappa),(j-i,0)}$. Through such a $(j-1)$-space there are $\beta_{(i,j-1,\kappa),(n,\eps),(n-1,0)}-\beta_{(i,j,0),(n,\eps),(n-1,0)}$ non-singular hyperplanes not containing $\pi$. The number of tuples we have for each of these hyperplanes is $\gamma_{(i,j-1,\kappa),(n-1,0),(n-j,\zeta),0}$.
		\item Finally, we look at the $(j-1)$-spaces through $\overline{\pi}$ that are $(i+1)$-singular. The number of such $(j-1)$-spaces corresponds to the number of 1-singular $(j-i-1)$-spaces (necessarily of type $0$) with respect to a non-singular quadratic form of type $0$ on a $(j-i)$-space, which is $\alpha_{(1,j-i-1,0),(j-i,0)}$. 
        Let $\pi'$ be such a $(j-1)$-space. Given a hyperplane $\sigma$ through $\pi'$ not containing $\pi$, the number of tuples depends on the perp type $\nu$ of $\pi'$ with respect to the quadratic form $f\vert_{\sigma}$: it equals $\gamma_{(i+1,j-1,0,\nu),(n-1,0),(n-j,\zeta),0}$.
        \par We now count the number of non-singular hyperplanes through $\pi'$, not containing $\pi$, such that $\pi'$ has perp type $\nu$ with respect to the restricted quadratic form. The number of non-singular hyperplanes through $\pi'$ such that $\pi'$ has perp type $\nu$ with respect to the restricted quadratic form equals $\beta^{\nu}_{(i+1,j-1,0),(n,\eps),(n-1,0)}$ by Definition \ref{def:betaquadraticodd}. To find the correct number we now need to subtract the number of such hyperplanes containing $\pi$.
        \par We denote the radical of $\pi'$ by $\overline{\pi}'$ and we have $\overline{\pi}'\supset\overline{\pi}$. We know that $\pi^{\perp}$ is an $i$-singular $(n-j)$-space of type 0 with radical $\overline{\pi}$ and that $\pi'^{\perp}$ is an $(i+1)$-singular $(n-j+1)$-space of type 0 with radical $\overline{\pi}'$, containing $\pi^\perp$. Note that $\pi'^\perp=\langle\pi^{\perp},\overline{\pi}'\rangle$. Furthermore, note that $\langle\pi,\pi^\perp\rangle=\overline{\pi}^{\perp}$ is an $i$-singular $(n-i)$-space and that $\langle\pi',\pi'^\perp\rangle=\overline{\pi}'^{\perp}$ is an $(i+1)$-singular $(n-i-1)$-space, contained in $\overline{\pi}^{\perp}$.
        \par Let $\sigma$ be a non-singular (necessarily parabolic) hyperplane through $\pi$ in which $\pi'$ has perp type $\nu$, and let $\perp'$ be the orthogonal polarity on $\sigma$ corresponding to $f\vert_{\sigma}$. Note that $\sigma$ contains neither $\langle\pi,\pi^\perp\rangle$ nor $\langle\pi',\pi'^\perp\rangle$, since a there cannot be an $i$-singular $(n-i)$-space or $(i+1)$-singular $(n-i-1)$-space with respect to a non-singular form in $\F^{n-1}_{q}$. Since $\sigma$ contains $\pi'$, the subspace $\sigma\cap\pi'^\perp$ must have dimension $n-j$, and be different from $\pi^\perp$. Moreover, $\sigma\cap\pi'^\perp$ must be $(i+1)$-singular as it contains $\overline{\pi}'$ but its radical cannot have dimension larger than $i+1$, for if it were than also $\sigma\cap\langle\pi',\pi'^{\perp}\rangle$ would have a radical of dimension at least $i+2$, which is impossible. Furthermore, we note that $\pi'^{\perp'}=\sigma\cap\pi'^{\perp}$, and so $\nu$ is the type of $\sigma\cap\pi'^{\perp}$.
        \par So, we first count the number of $(i+1)$-singular $(n-j)$-spaces of type $\nu$ in the $(i+1)$-singular $(n-j+1)$-space $\pi'^{\perp}$ through its radical $\overline{\pi}'$. Taking the quotient by the radical, we find that this number is $\alpha_{(0,n-j-i-1,\nu),(n-j-i,0)}$. And given such a space $\sigma'$ we now count the number of non-singular hyperplanes also containing $\pi$. Note that $\langle\sigma',\pi\rangle$ is an $(n-i-1)$-space in $\langle\pi,\pi^\perp$ containing $\overline{\pi}$ and so its radical has dimension at least $i$; if $\langle\sigma',\pi\rangle$ were $(i+1)$-singular, then its radical has to be contained in $\pi^{\perp}$, but that is impossible since $\sigma'\cap\pi^{\perp}$ is $i$-singular. So, we are counting the number of hyperplanes through an $i$-singular $(n-i-1)$-space (necessarily of type 0): this number is $\beta_{(i,n-i-1,0),(n,\eps),(n-1,0)}$.
        \par So, given a $(j-1)$-space in $\pi$ through $\overline{\pi}$ that is $(i+1)$-singular, the number of choices for $\tau$ equals
        \begin{align*}
            &\sum_{\nu\in\{\pm1\}}\left(\beta^{\nu}_{(i+1,j-1,0),(n,\eps),(n-1,0)}-\alpha_{(0,n-j-i-1,\nu),(n-j-i-1,0)}\beta_{(i,n-i-1,0),(n,\eps),(n-1,0)}\right)\\&\qquad\qquad\qquad\qquad\gamma_{(i+1,j-1,0,\nu),(n-1,0),(n-j,\zeta),0}\:.
        \end{align*}
	\end{itemize}
    We find the following result, which we simplify using Corollaries \ref{cor:alphaorthogonal}, \ref{cor:betas} and \ref{cor:betavariations}:
    \begin{align*}
        &\gamma_{(i,j,0),(n,\eps),(n-j,\zeta),\eps}\beta_{(0,n-j,\zeta),(n,\eps),(n-1,0)}\\
        &=\frac{q^{j}-q^{j-i}}{q-1}\sum_{\nu\in\{\pm1\}}\gamma_{(i-1,j-1,0,\nu),(n-1,0),(n-j,\zeta),0}\left(\beta^{\nu}_{(i-1,j-1,0),(n,\eps),(n-1,0)}\right.\\&\qquad\qquad\qquad\qquad\qquad\qquad\left.-\beta_{(1,1,1),(n-j-i+2,0),(n-j-i+1,\nu)}\beta_{(i-1,n-i,0),(n,\eps),(n-1,0)}\vphantom{\beta^{\nu}_{(i-1,j-1,0),(n,\eps),(n-1,0)}}\right)\\
        &\qquad+\sum_{\kappa\in\{\pm1\}}\alpha_{(0,j-i-1,\kappa),(j-i,0)}\left(\beta_{(i,j-1,\kappa),(n,\eps),(n-1,0)}-\beta_{(i,j,0),(n,\eps),(n-1,0)}\right)\\&\qquad\qquad\qquad\qquad\gamma_{(i,j-1,\kappa),(n-1,0),(n-j,\zeta),0}\\
        &\qquad+\alpha_{(1,j-i-1,0),(j-i,0)}\sum_{\nu\in\{\pm1\}}\gamma_{(i+1,j-1,0,\nu),(n-1,0),(n-j,\zeta),0}\\&\qquad\qquad\qquad\left(\beta^{\nu}_{(i+1,j-1,0),(n,\eps),(n-1,0)}-\alpha_{(0,n-j-i-1,\nu),(n-j-i,0)}\beta_{(i,n-i-1,0),(n,\eps),(n-1,0)}\right)\\
        &{=\frac{q^{j}-q^{j-i}}{q-1}\sum_{\nu\in\{\pm1\}}\gamma_{(i-1,j-1,0,\nu),(n-1,0),(n-j,\zeta),0}}\\&{\qquad\qquad\qquad\left(\frac{1}{2}q^{\frac{1}{2}(n-j+i-1)}\left(q^{\frac{1}{2}(n-j-i+1)}+\nu\right)-\frac{1}{2}q^{\frac{1}{2}(n-j-i+1)}\left(q^{\frac{1}{2}(n-j-i-1)}+\nu\right)q^{i-1}\right)}\\
        &{\qquad+\sum_{\kappa\in\{\pm1\}}\frac{1}{2}q^{\frac{1}{2}(j-i-1)}\left(q^{\frac{1}{2}(j-i-1)}+\kappa\right)\left(q^{\frac{1}{2}(n-j+i-1)}\left(q^{\frac{1}{2}(n-j-i+1)}-\eps\kappa\right)-q^{n-j-1}\right)}\\&{\qquad\qquad\qquad\qquad\gamma_{(i,j-1,\kappa),(n-1,0),(n-j,\zeta),0}}\\
        &\qquad+\frac{q^{j-i-1}-1}{q-1}\sum_{\nu\in\{\pm1\}}\gamma_{(i+1,j-1,0,\nu),(n-1,0),(n-j,\zeta),0}\\&\qquad\qquad\qquad\left(\frac{1}{2}q^{\frac{1}{2}(n-j+i+1)}\left(q^{\frac{1}{2}(n-j-i-1)}+\nu\right)-\frac{1}{2}q^{\frac{1}{2}(n-j-i-1)}\left(q^{\frac{1}{2}(n-j-i-1)}+\nu\right)q^{i}\right)\\
        &=\frac{1}{2}q^{n-i-1}\left(q^{i}-1\right)\sum_{\nu\in\{\pm1\}}\gamma_{(i-1,j-1,0,\nu),(n-1,0),(n-j,\zeta),0}\\
        &\qquad+\frac{1}{2}q^{\frac{1}{2}n-1}\sum_{\kappa\in\{\pm1\}}\left(q^{\frac{1}{2}n-i-1}(q-1)-\eps+\kappa q^{\frac{1}{2}(j-i-1)}\left((q-1)q^{\frac{1}{2}n-j}-\eps\right)\right)\\&\qquad\qquad\qquad\qquad\qquad\gamma_{(i,j-1,\kappa),(n-1,0),(n-j,\zeta),0}\\
        &\qquad+\frac{1}{2}\left(q^{j-i-1}-1\right)q^{\frac{1}{2}(n-j+i-1)}\sum_{\nu\in\{\pm1\}}\gamma_{(i+1,j-1,0,\nu),(n-1,0),(n-j,\zeta),0}\left(q^{\frac{1}{2}(n-j-i-1)}+\nu\right)\:.
    \end{align*}
    Note that this equality is also valid if $i=j-1$ and if $i=0$. Then (in both instances), only two of the three possible cases for a $(j-1)$-space in $\pi$ occur. But the cases that do not appear, indeed correspond to terms with a factor $0$ in the given equation.
\end{proof}

Recall that we set $\gs{-1}{0}{q}=1$; this is important for the case of totally singular spaces (i.e.~$j=i$) in the following theorem, which deals with the case that $n-j$ is even.

\begin{theorem}\label{th:gammaquadratic}
    Let $n,j,i$ be integers with $0\leq i\leq j$ and $i+j\leq n$, such that $n-j$ is even. Let $\delta,\eps\in\{0,\pm1\}$ and $\zeta,\lambda\in\{\pm1\}$ with $j-i-\delta\equiv1\equiv n-\eps\pmod{2}$. We have that
    \begin{align*}
        &\gamma_{(i,j,\delta(,\lambda)),(n,\eps),(n-j,\zeta),\eps}\\
        &=\frac{1}{2}q^{j(n-j)-\left(\frac{j+1-\eps^2}{2}\right)^{2}}\sum_{m=0}^{\frac{j-i-1+\delta^2}{2}}\left(\chi_{1,\frac{j+1-\eps^2}{2}-m}\gs{\frac{j-i-1+\delta^2}{2}}{m}{q^2}q^{m(j+i-n+m+1-\delta^2-\eps^2)}\right.\\
        &\qquad+\zeta q^{\frac{1}{2}(j-n)+(1-\eps^2)j}\chi_{1,\frac{j-1+\eps^2}{2}-m}\gs{\frac{j-i-1-\delta^2}{2}}{m}{q^2}q^{m(j+i-n+m-1+\delta^2+\eps^2)}\\
        &\qquad-\delta\zeta q^{\frac{1}{2}(i-n)+(1-\eps^2)j}\chi_{1,\frac{j-1+\eps^2}{2}-m}\gs{\frac{j-i}{2}-1}{m}{q^2}q^{m(j+i-n+m+\eps^2)}\\
        &\qquad+\eps\zeta q^{-\frac{1}{2}j}\chi_{1,\frac{j}{2}-m}\gs{\frac{j-i-1-\delta^2}{2}}{m}{q^2}q^{m(j+i-n+m+\delta^2)}\\
        &\qquad+\delta\eps\zeta q^{-\frac{1}{2}i}\chi_{1,\frac{j}{2}-m}\gs{\frac{j-i}{2}-1}{m-1}{q^2}q^{m(j+i-n+m-1)}\\
        &\qquad\left.+\lambda q^{\frac{1}{2}(j+i-n)}\chi_{1,\frac{j-1}{2}-m}\gs{\frac{j-i-1}{2}}{m}{q^2}q^{m(j+i-n+m+1)}\right)
    \end{align*}
    where the $\lambda$ only is considered if $n(j-i)$ is odd, and treated as 0 else.
\end{theorem}
\begin{proof}
    We prove this theorem using induction on $j$. For $j=0$ we have $i=0$ and thus $\delta=1$, and we also know that $n$ is even and thus $\eps\in\{\pm1\}$. We find that $\gamma_{(0,0,1),(n,\eps),(n,\zeta),\eps}=\frac{1}{2}\left(1+\eps\zeta\right)$, which indeed equals 1 if $\eps=\zeta$ and 0 if $\eps\neq\zeta$.
    \par Now, we assume that $j\geq1$. Given that $n-j$ is even, there are four possibilities for the parities of $n$, $j$ and $i$: \{$n,j$ odd and $i$ odd\}, \{$n,j$ odd and $i$ even\}, \{$n,j$ even and $i$ even\}, \{$n,j$ even and $i$ odd\}. For these four cases, we use the recursion relations given in Lemmas \ref{lem:recursion3}, \ref{lem:recursion4}, \ref{lem:recursion1} and \ref{lem:recursion2}, respectively. To complete the induction argument we only need to make a (lenghty) calculation in each of these four cases. We defer these calculations to Appendix \ref{ap:calculations}. For ease of the calculations we give simplified formulae for each of the four cases.
    \par For $n,j$ odd, and $i$ odd:
    \begin{align}
        \gamma_{(i,j,\delta),(n,0),(n-j,\zeta),0}
        &=\frac{1}{2}q^{jn-\frac{5}{4}j^2-\frac{1}{2}j-\frac{1}{4}}\left(\sum_{m=0}^{\frac{j-i}{2}}\chi_{1,\frac{j+1}{2}-m}(q)\gs{\frac{j-i}{2}}{m}{q^{2}}q^{m(j+i-n+m)}\right.\nonumber\\
        &\qquad\ \ \left.+\zeta q^{j-\frac{1}{2}n}\left(q^{\frac{1}{2}j}-\delta q^{\frac{1}{2}i}\right)\sum_{m=0}^{\frac{j-i}{2}-1}\chi_{1,\frac{j-1}{2}-m}(q)\gs{\frac{j-i}{2}-1}{m}{q^{2}}q^{m(j+i-n+m)}\right).\label{rec3}
    \end{align}
    \par For $n,j$ odd and $i$ even:
    \begin{align}
        \gamma_{(i,j,0,\lambda),(n,0),(n-j,\zeta),0}&=\frac{1}{2}q^{jn-\frac{5}{4}j^2-\frac{1}{2}j-\frac{1}{4}}\left(\sum_{m=0}^{\frac{j-i-1}{2}}\chi_{1,\frac{j+1}{2}-m}(q)\gs{\frac{j-i-1}{2}}{m}{q^{2}}q^{m(j+i-n+m+1)}\right.\nonumber\\
        &\qquad\qquad+\zeta\:q^{\frac{3}{2}j-\frac{1}{2}n}\sum_{m=0}^{\frac{j-i-1}{2}}\chi_{1,\frac{j-1}{2}-m}(q)\gs{\frac{j-i-1}{2}}{m}{q^{2}}q^{m(j+i-n+m-1)}\nonumber\\
        &\qquad\qquad\left.+\lambda\:q^{\frac{1}{2}(j+i-n)}\sum_{m=0}^{\frac{j-i-1}{2}}\chi_{1,\frac{j-1}{2}-m}(q)\gs{\frac{j-i-1}{2}}{m}{q^{2}}q^{m(j+i-n+m+1)}\right)\label{rec4}.
    \end{align}
    \par For $n,j$ even, and $i$ even:
    \begin{align}
        \gamma_{(i,j,\delta),(n,\eps),(n-j,\zeta),\eps}
        &=\frac{1}{2}q^{jn-\frac{5}{4}j^2}\left(\sum_{m=0}^{\frac{j-i}{2}}\chi_{1,\frac{j}{2}-m}(q)\gs{\frac{j-i}{2}}{m}{q^{2}}q^{m(j+i-n+m-1)}\right.\nonumber\\
        &\qquad\qquad+\zeta\:q^{\frac{1}{2}(j-n)}\sum_{m=0}^{\frac{j-i}{2}-1}\chi_{1,\frac{j}{2}-m}(q)\gs{\frac{j-i}{2}-1}{m}{q^{2}}q^{m(j+i-n+m+1)}\nonumber\\
        &\qquad\qquad+\eps\zeta\:q^{-\frac{1}{2}j}\sum_{m=0}^{\frac{j-i}{2}}\chi_{1,\frac{j}{2}-m}(q)\gs{\frac{j-i}{2}-1}{m}{q^{2}}q^{m(j+i-n+m+1)}\nonumber\\
        &\qquad\qquad-\delta\zeta\:q^{\frac{1}{2}(i-n)}\sum_{m=0}^{\frac{j-i}{2}-1}\chi_{1,\frac{j}{2}-m}(q)\gs{\frac{j-i}{2}-1}{m}{q^{2}}q^{m(j+i-n+m+1)}\nonumber\\
        &\qquad\qquad\left.+\delta\eps\zeta\:q^{-\frac{1}{2}i}\sum_{m=0}^{\frac{j-i}{2}}\chi_{1,\frac{j}{2}-m}(q)\gs{\frac{j-i}{2}-1}{m-1}{q^{2}}q^{m(j+i-n+m-1)}\right)\label{rec11}.
    \end{align}
    \par For $n,j$ even, and $i$ odd:
    \begin{align}
        \gamma_{(i,j,0),(n,\eps),(n-j,\zeta),\eps}&=\frac{1}{2}q^{jn-\frac{5}{4}j^2}\left(\sum_{m=0}^{\frac{j-i-1}{2}}\chi_{1,\frac{j}{2}-m}(q)\gs{\frac{j-i-1}{2}}{m}{q^{2}}q^{m(j+i-n+m)}\right.\nonumber\\
        &\qquad+\left.\zeta\left(q^{\frac{1}{2}(j-n)}+\eps q^{-\frac{1}{2}j}\right)\sum_{m=0}^{\frac{j-i-1}{2}}\chi_{1,\frac{j}{2}-m}(q)\gs{\frac{j-i-1}{2}}{m}{q^{2}}q^{m(j+i-n+m)}\right).\label{rec2}
    \end{align}
\end{proof}

\begin{example}\label{ex:gamma}
    Consider the case $i=0,j=3$, $n=5$, $\lambda=-1$, $\zeta=1$. Then $\gamma_{(0,3,0,-1),(5,0),(2,1),0}$ counts the following: given a plane $\pi$ meeting a parabolic quadric $Q(4,q)$ in a conic, and such that $\pi^\perp$ is a passant to $Q(4,q)$, $\gamma_{(0,3,0,-1),(5,0),(2,1),0}$ is the number of secant lines to $Q(4,q)$, disjoint from $\pi$. Using formula \eqref{rec4}, we find that
    \begin{align*}
        \gamma_{(0,3,0,-1),(5,0),(2,1),0}&=\frac{1}{2}q^2(\chi_{1,2}+\chi_{1,1}+q^2(\chi_{1,1}+\chi_{1,0}q^{-2})-q^{-1}(\chi_{1,1}+\chi_{1,0})\\
        &=\frac{1}{2}q^2\left((q-1)(q^3-1)+(q-1)+q^2(q-1+q^{-2})-q^{-1}((q-1)+1 )\right)\\
        &=\frac{1}{2}q^4\left(q^2-1 \right).
    \end{align*}
    Similarly, when $\pi^\perp$ is a secant line (i.e. $\pi$ has perp type $1$), we find
    \begin{align*}
        \gamma_{(0,3,0,1),(5,0),(2,1),0}&=\frac{1}{2}q^2(\chi_{1,2}+\chi_{1,1}+q^2(\chi_{1,1}+\chi_{1,0}q^{-2})+q^{-1}(\chi_{1,1}+\chi_{1,0})\\
        &=\frac{1}{2}q^2\left((q-1)(q^3-1)+(q-1)+q^2(q-1+q^{-2})+q^{-1}((q-1)+1) \right)\\
        &=\frac{1}{2}q^2\left(q^4-q^2+2 \right).
    \end{align*}    
\end{example}

\begin{example}\label{ex:gamma2}
    Consider $i=0,j=2,n=4,\delta=\eps=\zeta=1$, then $\gamma_{(0,2,1),(4,1),(2,1),1}$ counts the number of secant lines to a hyperbolic quadric $\mathcal{Q}^+(3,q)$, disjoint from a given secant line to $\mathcal{Q}^+(3,q)$. Using \eqref{rec11} we find that this number is equal to 
    \begin{align*}
        \frac{1}{2}q^{3}(\chi_{1,1}+\chi_{1,0}q^{-2}+q^{-1}\chi_{1,1}+q^{-1}\chi_{1,1}-q^{-2}\chi_{1,1}+\chi_{1,0}q^{-2}) =\frac{1}{2}q(q^3+q^2-3q+3)\:.
    \end{align*}
\end{example}

As explained before, we will use Theorem \ref{th:gammaquadratic}, which dealt with the case that $n-j$ is even, together with the associate polarity, to find an expression in the case that $n$ is odd and $j$ even.

\begin{theorem}\label{th:gammaquadratic0odd}
    Let $n,j,i$ be integers with $0\leq i\leq j$ and $i+j\leq n$, such that $n$ is odd and $j$ is even. Let $\delta\in\{0,\pm1\}$ and $\lambda,\mu\in\{\pm1\}$ with $j-i-\delta\equiv1\pmod{2}$. For $i$ even (and thus $\delta\in\{\pm1\}$), we have that
    \begin{align*}
        &\gamma_{(i,j,\delta),(n,0),(n-j,0,\mu),0}\\
        &=\frac{1}{2}q^{\frac{3}{2}jn-\frac{1}{4}n^{2}-\frac{5}{4}j^{2}-\frac{1}{2}n+\frac{1}{2}j-\frac{1}{4}}\left(\sum_{m=0}^{\frac{n-j-i-1}{2}}\chi_{1,\frac{n-j+1}{2}-m}(q)\gs{\frac{n-j-i-1}{2}}{m}{q^{2}}q^{m(m-j+i+1)}\right.\\
        &\qquad\qquad+\mu\:q^{n-\frac{3}{2}j}\sum_{m=0}^{\frac{n-j-i-1}{2}}\chi_{1,\frac{n-j-1}{2}-m}(q)\gs{\frac{n-j-i-1}{2}}{m}{q^{2}}q^{m(m-j+i-1)}\\
        &\qquad\qquad\left.+\delta\:q^{\frac{1}{2}(i-j)}\sum_{m=0}^{\frac{n-j-i-1}{2}}\chi_{1,\frac{n-j-1}{2}-m}(q)\gs{\frac{n-j-i-1}{2}}{m}{q^{2}}q^{m(m-j+i+1)}\right)
    \end{align*}
    and for $i$ odd (and thus $\delta=0$ and $\lambda\in\{\pm1\}$), we have that
    \begin{align*}
        &\gamma_{(i,j,0,\lambda),(n,0),(n-j,0,\mu),0}\\
        &=\frac{1}{2}q^{\frac{3}{2}jn-\frac{1}{4}n^{2}-\frac{5}{4}j^{2}-\frac{1}{2}n+\frac{1}{2}j-\frac{1}{4}}\left(\sum_{m=0}^{\frac{n-j-i}{2}}\chi_{1,\frac{n-j+1}{2}-m}(q)\gs{\frac{n-j-i}{2}}{m}{q^{2}}q^{m(m-j+i)}\right.\\
        &\qquad\ \ \left.+\mu q^{\frac{1}{2}n-j}\left(q^{\frac{1}{2}(n-j)}-\lambda q^{\frac{1}{2}i}\right)\sum_{m=0}^{\frac{n-j-i}{2}-1}\chi_{1,\frac{n-j-1}{2}-m}(q)\gs{\frac{n-j-i}{2}-1}{m}{q^{2}}q^{m(m-j+i)}\right).
    \end{align*}
\end{theorem}
\begin{proof}
    Consider a parabolic form on $\F_{q}^{n}$ with corresponding orthogonal polarity $\perp$, and let $\pi$ be a fixed $i$-singular $j$-space of type $\delta$ (and perp type $\lambda$ in case $\delta=0$). We know that $\pi^\perp$ is an $i$-singular $(n-j)$-space of type $\lambda$ (in case $\delta=0$) or type 0 (in case $\delta^2=1$). We want to count the number of non-singular $(n-j)$-spaces (necessarily of type 0) with perp type $\mu$ disjoint to $\pi$. Applying the polarity $\perp$, it follows that this number is equal to the number of non-singular $j$-spaces of type $\mu$ disjoint from $\pi^\perp$. So, we find the following results, using Theorem \ref{th:gammaquadratic}. For $i$ even (and thus $\delta\in\{\pm1\}$), using \eqref{rec4}, we have:
    \begin{align*}
        &\gamma_{(i,j,\delta),(n,0),(n-j,0,\mu),0}\\
        &=\gamma_{(i,n-j,0,\delta),(n,0),(j,\mu),0}\\
        &=\frac{1}{2}q^{\frac{3}{2}jn-\frac{1}{4}n^{2}-\frac{5}{4}j^{2}-\frac{1}{2}n+\frac{1}{2}j-\frac{1}{4}}\left(\sum_{m=0}^{\frac{n-j-i-1}{2}}\chi_{1,\frac{n-j+1}{2}-m}(q)\gs{\frac{n-j-i-1}{2}}{m}{q^{2}}q^{m(m-j+i+1)}\right.\\
        &\qquad\qquad+\mu\:q^{n-\frac{3}{2}j}\sum_{m=0}^{\frac{n-j-i-1}{2}}\chi_{1,\frac{n-j-1}{2}-m}(q)\gs{\frac{n-j-i-1}{2}}{m}{q^{2}}q^{m(m-j+i-1)}\\
        &\qquad\qquad\left.+\delta\:q^{\frac{1}{2}(i-j)}\sum_{m=0}^{\frac{n-j-i-1}{2}}\chi_{1,\frac{n-j-1}{2}-m}(q)\gs{\frac{n-j-i-1}{2}}{m}{q^{2}}q^{m(m-j+i+1)}\right).
    \end{align*}
    For $i$ odd (and thus $\delta=0$), using \eqref{rec3}, we have:
    \begin{align*}
        &\gamma_{(i,j,0,\lambda),(n,0),(n-j,0,\mu),0}\\
        &=\gamma_{(i,n-j,\lambda),(n,0),(j,\mu),0}\\
        &=\frac{1}{2}q^{\frac{3}{2}jn-\frac{1}{4}n^{2}-\frac{5}{4}j^{2}-\frac{1}{2}n+\frac{1}{2}j-\frac{1}{4}}\left(\sum_{m=0}^{\frac{n-j-i}{2}}\chi_{1,\frac{n-j+1}{2}-m}(q)\gs{\frac{n-j-i}{2}}{m}{q^{2}}q^{m(m-j+i)}\right.\\
        &\qquad\ \ \left.+\mu q^{\frac{1}{2}n-j}\left(q^{\frac{1}{2}(n-j)}-\lambda q^{\frac{1}{2}i}\right)\sum_{m=0}^{\frac{n-j-i}{2}-1}\chi_{1,\frac{n-j-1}{2}-m}(q)\gs{\frac{n-j-i}{2}-1}{m}{q^{2}}q^{m(m-j+i)}\right).\qedhere
    \end{align*}
 
\end{proof}

\begin{corollary}\label{cor:gammaquadratic0odd}
    Let $n,j,i$ be integers with $0\leq i\leq j$ and $i+j\leq n$, such that $n$ is odd and $j$ is even. Let $\delta\in\{0,\pm1\}$ and $\lambda\in\{\pm1\}$ with $j-i-\delta\equiv1\pmod{2}$. For $i$ even (and thus $\delta\in\{\pm1\}$) we have that
    \begin{align*}
        \gamma_{(i,j,\delta),(n,0),(n-j,0),0}&=q^{\frac{3}{2}jn-\frac{1}{4}n^{2}-\frac{5}{4}j^{2}-\frac{1}{2}n+\frac{1}{2}j-\frac{1}{4}}\left(\sum_{m=0}^{\frac{n-j-i-1}{2}}\chi_{1,\frac{n-j+1}{2}-m}(q)\gs{\frac{n-j-i-1}{2}}{m}{q^{2}}q^{m(m-j+i+1)}\right.\\
        &\qquad\qquad\left.+\delta\:q^{\frac{1}{2}(i-j)}\sum_{m=0}^{\frac{n-j-i-1}{2}}\chi_{1,\frac{n-j-1}{2}-m}(q)\gs{\frac{n-j-i-1}{2}}{m}{q^{2}}q^{m(m-j+i+1)}\right)
    \end{align*}
    and for $i$ odd (and thus $\delta=0$ and $\lambda\in\{\pm1\}$) we have that
    \begin{align*}
        \gamma_{(i,j,0,\lambda),(n,0),(n-j,0),0}
        &=q^{\frac{3}{2}jn-\frac{1}{4}n^{2}-\frac{5}{4}j^{2}-\frac{1}{2}n+\frac{1}{2}j-\frac{1}{4}}\sum_{m=0}^{\frac{n-j-i}{2}}\chi_{1,\frac{n-j+1}{2}-m}(q)\gs{\frac{n-j-i}{2}}{m}{q^{2}}q^{m(m-j+i)}\:.
    \end{align*}
\end{corollary}
\begin{proof}
    It is clear that for $i$ even we have
    \begin{align*}
        \gamma_{(i,j,\delta),(n,0),(n-j,0),0}&=\gamma_{(i,j,\delta),(n,0),(n-j,0,1),0}+\gamma_{(i,j,\delta),(n,0),(n-j,0,-1),0}
    \end{align*}
    and that for $i$ odd we have
    \begin{align*}
        \gamma_{(i,j,0,\lambda),(n,0),(n-j,0),0}&=\gamma_{(i,j,0,\lambda),(n,0),(n-j,0,1),0}+\gamma_{(i,j,0,\lambda),(n,0),(n-j,0,-1),0}
    \end{align*}
    from which the statement immediately follows.
\end{proof}

Note that the above formula for $\gamma_{(i,j,0,\lambda),(n,0),(n-j,0),0}$ is independent from $\lambda$.

\begin{example}
    Consider the case $i=0,j=2,n=5,\delta=1$. Then $\gamma_{(0,2,1),(5,0),(3,0),0}$ denotes the number of planes meeting a parabolic quadric $Q(4,q)$ in a conic, and disjoint from a given secant line to this $Q(4,q)$. By Corollary \ref{cor:gammaquadratic0odd} this number is
    \[
        q^2\left(\chi_{1,2}(q)+\chi_{1,1}(q)+q^{-1}\left(\chi_{1,1}(q)+\chi_{1,0}(q)\right)\right)=q^2(q^4-q^3+1)\:.
    \]
\end{example}

\begin{example}
    Consider the case $i=1,j=2,n=5,\lambda=1$. Let $Q$ be a parabolic quadric $Q(4,q)$ and let $\pi$ be a tangent line to $Q$ whose image under the polarity is a plane meeting $Q$ in two intersecting lines; then $\gamma_{(0,2,0,1),(5,0),(3,0),0}$ counts the number of planes disjoint from this line $\pi$ meeting $Q(4,q)$ in a conic. By Corollary \ref{cor:gammaquadratic0odd} this number is
    \[
        q^2\left(\chi_{1,2}(q)+\chi_{1,1}(q)\right)=q^5(q-1)\:.
    \]
\end{example}

We are now able to derive the expression in the remaining case, namely when $n$ is even and $j$ is odd.

\begin{theorem}\label{th:gammaquadratic0even}
    Let $n,j,i$ be integers with $0\leq i\leq j$ and $i+j\leq n$, such that $n$ is even and $j$ is odd. Let $\delta\in\{0,\pm1\}$ and $\eps,\lambda\in\{\pm1\}$ with $j-i-\delta\equiv1\pmod{2}$. If $i$ is odd, then
    \begin{align*}
        &\gamma_{(i,j,\delta),(n,\eps),(n-j,0),\eps}=q^{\frac{3}{2}jn-\frac{1}{4}n^{2}-\frac{5}{4}j^{2}-\frac{1}{2}n+\frac{1}{2}j-\frac{1}{4}}\left(\sum_{m=0}^{\frac{n-j-i}{2}}\chi_{1,\frac{n-j+1}{2}-m}(q)\gs{\frac{n-j-i}{2}-1}{m}{q^{2}}q^{m(m-j+i+1)}\right.\\
        &\qquad+\left.\left(\delta q^{\frac{1}{2}i-\frac{1}{2}j}-\eps q^{\frac{1}{2}n-j}+\delta\eps q^{\frac{1}{2}n-\frac{3}{2}j+\frac{1}{2}i}\right)\sum_{m=0}^{\frac{n-j-i}{2}-1}\chi_{1,\frac{n-j-1}{2}-m}(q)\gs{\frac{n-j-i}{2}-1}{m}{q^{2}}q^{m(m-j+i+1)}.\right)
    \end{align*}

If $i$ is even, then
    \begin{align*}
        \gamma_{(i,j,0),(n,\eps),(n-j,0),\eps}=&q^{\frac{3}{2}jn-\frac{1}{4}n^{2}-\frac{5}{4}j^{2}-\frac{1}{2}n+\frac{1}{2}j-\frac{1}{4}}\left( \sum_{m=0}^{\frac{n-j-i-1}{2}}\chi_{1,\frac{n-j+1}{2}-m}(q)\gs{\frac{n-j-i-1}{2}}{m}{q^{2}}q^{m(m-j+i)}\right.\\
        &\qquad\left. -\eps q^{\frac{1}{2}n-j}\sum_{m=0}^{\frac{n-j-i-1}{2}}\chi_{1,\frac{n-j-1}{2}-m}(q)\gs{\frac{n-j-i-1}{2}}{m}{q^{2}}q^{m(m-j+i)} \right)
    \end{align*}
\end{theorem}
\begin{proof}
    We distinguish between two cases. First assume that $i$ is odd. Using Lemma \ref{lem:recursion1}, we find that
    \begin{align*}
        &\gamma_{(i,j,\delta),(n,\eps),(n-j,0),\eps}\beta_{(0,n-j,0),(n,\eps),(n-1,0)}\\
        &=\left(q^{i}-1\right)q^{n-i-1}\gamma_{(i-1,j-1,\delta),(n-1,0),(n-j,0),0}\\
        &\qquad+\frac{1}{2}q^{\frac{1}{2}n-2}\left(q^{\frac{1}{2}(j-i)}-\delta\right)\sum_{\nu\in\{\pm1\}}\left((q-1)q^{\frac{1}{2}(n-j-i)}+\nu q+\delta\eps\right)\gamma_{(i,j-1,0,\nu),(n-1,0),(n-j,0),0}\\
        &\qquad+q^{\frac{1}{2}(n-j+i)-1}\left(q^{\frac{1}{2}(j-i-1-\delta)}+1\right)\left(q^{\frac{1}{2}(j-i-1+\delta)}-1\right)\left(q^{\frac{1}{2}(n-j-i)}-\delta\eps\right)\\&\qquad\qquad\qquad\gamma_{(i+1,j-1,\delta),(n-1,0),(n-j,0),0}
    \end{align*}
    The formulae in Corollary \ref{cor:gammaquadratic0odd} now allow us to substitute the values of $\gamma_{(i-1,j-1,\delta),(n-1,0),(n-j,0),0}$,  $\gamma_{(i,j-1,0,\nu),(n-1,0),(n-j,0),0}$, and $\gamma_{(i+1,j-1,\delta),(n-1,0),(n-j,0),0}$. We collect the working in Subsection \ref{sec:appendixB2} of the Appendix. This yields:
    \begin{align*} 
        &\gamma_{(i,j,\delta),(n,\eps),(n-j,0),\eps}\beta_{(0,n-j,0),(n,\eps),(n-1,0)}\\
        &=q^{\frac{3}{2}jn-\frac{1}{4}n^{2}-\frac{5}{4}j^{2}-\frac{1}{2}n+\frac{3}{2}j-\frac{5}{4}}\left(\sum_{m=0}^{\frac{n-j-i}{2}}\chi_{1,\frac{n-j+1}{2}-m}(q)\gs{\frac{n-j-i}{2}-1}{m}{q^{2}}q^{m(m-j+i+1)}\right.\\
        &\qquad+\left.\left(\delta q^{\frac{1}{2}i-\frac{1}{2}j}-\eps q^{\frac{1}{2}n-j}+\delta\eps q^{\frac{1}{2}n-\frac{3}{2}j+\frac{1}{2}i}\right)\sum_{m=0}^{\frac{n-j-i}{2}-1}\chi_{1,\frac{n-j-1}{2}-m}(q)\gs{\frac{n-j-i}{2}-1}{m}{q^{2}}q^{m(m-j+i+1)}\right)
    \end{align*}
    Since Corollary \ref{cor:betas} shows that $\beta_{(0,n-j,0),(n,\eps),(n-1,0)}=q^{j-1}\neq0$,  the statement follows.
    \par Now assume that $i$ is even. 
    By Lemma \ref{lem:recursion2}, we have
    \begin{align*}
        &\gamma_{(i,j,0),(n,\eps),(n-j,0),\eps}\beta_{(0,n-j,0),(n,\eps),(n-1,0)}\\
        &=\frac{1}{2}q^{n-i-1}\left(q^{i}-1\right)\sum_{\nu\in\{\pm1\}}\gamma_{(i-1,j-1,0,\nu),(n-1,0),(n-j,0),0}\\
        &\qquad+\frac{1}{2}q^{\frac{1}{2}n-1}\sum_{\kappa\in\{\pm1\}}\left(q^{\frac{1}{2}n-i-1}(q-1)-\eps+\kappa q^{\frac{1}{2}(j-i-1)}\left((q-1)q^{\frac{1}{2}n-j}-\eps\right)\right)\\&\qquad\qquad\qquad\qquad\qquad\gamma_{(i,j-1,\kappa),(n-1,0),(n-j,0),0}\\
        &\qquad+\frac{1}{2}\left(q^{j-i-1}-1\right)q^{\frac{1}{2}(n-j+i-1)}\sum_{\nu\in\{\pm1\}}\gamma_{(i+1,j-1,0,\nu),(n-1,0),(n-j,0),0}\left(q^{\frac{1}{2}(n-j-i-1)}+\nu\right)
    \end{align*}
    The formulae in Corollary \ref{cor:gammaquadratic0odd} now allow us to substitute the values of $\gamma_{(i-1,j-1,0,\nu),(n-1,0),(n-j,0),0}$, $\gamma_{(i,j-1,\kappa),(n-1,0),(n-j,0),0}$, and $\gamma_{(i+1,j-1,0,\nu),(n-1,0),(n-j,0),0}$. We collect the working in Subsection \ref{sec:appendixB2} of the Appendix. This yields:
    \begin{align*}
        &\gamma_{(i,j,0),(n,\eps),(n-j,0),\eps}\beta_{(0,n-j,0),(n,\eps),(n-1,0)}\\
        &=q^{\frac{3}{2}jn-\frac{1}{4}n^{2}-\frac{5}{4}j^{2}-\frac{1}{2}n+\frac{3}{2}j-\frac{5}{4}}\left( \sum_{m=0}^{\frac{n-j-i-1}{2}}\chi_{1,\frac{n-j+1}{2}-m}(q)\gs{\frac{n-j-i-1}{2}}{m}{q^{2}}q^{m(m-j+i)}\right.\\
        &\qquad\left. -\eps q^{\frac{1}{2}n-j}\sum_{m=0}^{\frac{n-j-i-1}{2}}\chi_{1,\frac{n-j-1}{2}-m}(q)\gs{\frac{n-j-i-1}{2}}{m}{q^{2}}q^{m(m-j+i)} \right)
    \end{align*}
    Since Corollary \ref{cor:betas} shows that $\beta_{(0,n-j,0),(n,\eps),(n-1,0)}=q^{j-1}\neq0$,  the statement follows.
\end{proof}

\begin{example}
    We consider the case $i=1$, $j=3$, $n=6$ and $\delta=\eps=1$. Then $\gamma_{(1,3,1),(6,1),(3,0),1)}$ denotes the number of conic planes with respect to a hyperbolic quadric $\mathcal{Q}^+(5,q)$ that are disjoint from a given plane meeting this $\mathcal{Q}^+(5,q)$ in two intersecting lines. By Theorem \ref{th:gammaquadratic0even} this number is
    \begin{align*}
        q^5\left(\chi_{1,2}(q)+\left(q^{-1}-1+ q^{-1}\right)\chi_{1,1}(q)\right)&=q^5\left((q-1)\left(q^3-1\right)+\left(2q^{-1}-1\right)(q-1)\right)\\&=q^4\left(q^5-q^4-2q^2+4q-2\right)\:.
    \end{align*}
\end{example}

\begin{example}
    We consider the case $i=2$, $j=3$, $n=6$ and $\eps=1$. Then $\gamma_{(2,3,0),(6,1),(3,0),1)}$ denotes the number of conic planes with respect to a hyperbolic quadric $\mathcal{Q}^+(5,q)$ that are disjoint from a given plane meeting this $\mathcal{Q}^+(5,q)$ in exactly one line. By Theorem \ref{th:gammaquadratic0even} this number is
    \begin{align*}
        q^5\left(\chi_{1,2}(q)-\chi_{1,1}(q)\right)&=q^5\left((q-1)\left(q^3-1\right)-(q-1)\right)\\&=q^5\left(q-1\right)(q^3-2)\:.
    \end{align*}
\end{example}

We now derive our main result for the general case where the trivially intersecting subspaces do not necessarily span the full space.

\begin{theorem}\label{th:gammageneral}
    Let $n,j,i,k$ be integers with $0\leq i\leq j$ and $i,k\leq n-j$. Let $\delta,\eps,\zeta,\eta\in\{0,\pm1\}$ and $\lambda\in\{\pm1\}$ with $1\equiv j-i-\delta\equiv n-\eps\equiv k-\zeta\equiv k+j-\eta\pmod{2}$. If $k+j$ is even or $j-i$ is even, then
    \begin{align*}
        \gamma_{(i,j,\delta(,\lambda)),(n,\eps),(k,\zeta),\eta}=\beta_{(i,j,\delta(,\lambda)),(n,\eps),(k+j,\eta)}\gamma_{(i,j,\delta),(k+j,\eta),(k,\zeta),\eta}\:.
    \end{align*}
    
    If $k+j$ and $j-i$ are both odd, then
    \begin{align*}
        \gamma_{(i,j,\delta(,\lambda)),(n,\eps),(k,\zeta),\eta}=\sum_{\nu\in\{\pm1\}}\beta^{\nu}_{(i,j,0),(n,\eps),(k+j,0)}\gamma_{(i,j,0,\nu),(k+j,0),(k,\zeta),0}\:.
    \end{align*}
\end{theorem}
\begin{proof}
    Consider a quadratic form $f$ of type $\eps$ on $\F_{q}^{n}$, and consider a fixed $i$-singular $j$-space of type $\delta$ (and perp type $\lambda$ in case $n$ and $j-i$ are odd). Now, we count in two ways the pairs $(\sigma,\tau)$, where $\sigma$ is a non-singular $k$-space of type $\zeta$, trivially intersecting $\pi$ %and spanning a $(k+j)$-space of type $\eta$ with $\pi$
    and $\tau=\langle\pi,\sigma\rangle$ is a non-singular $(k+j)$-space of type $\eta$. 
    Note that the type of $\pi$ with respect to $f\vert_\tau$ is $\delta$. In case $k+j$ is even or $j-i$ is even, $\delta$ is sufficient to determine the number of pairs $(\sigma,\tau)$ given a fixed $\tau$. The number of choices for the subspace $\sigma$ is by definition $\gamma_{(i,j,\delta),(n,\eps),(k,\zeta),\eta}$, and the choice of $\sigma$ uniquely determines $\tau$. On the other hand, there are $\beta_{(i,j,\delta(,\lambda)),(n,\eps),(k+j,\eta)}$ choices for a subspace  $\tau$ through $\pi$, each of which contains $\gamma_{(i,j,\delta),(k+j,\eta),(k,\zeta),\eta}$ choices for $\sigma$. The first equality from the statement immediately follows. In the case that $k+j$ and $j-i$ are both odd, $\delta=0$ and the number of pairs $(\sigma,\tau)$ given a fixed $\tau$, depends on the perp type $\nu$ of $\pi$ with respect to $f\vert_{\tau}$. So, we need to sum over both possible perp types to find all pairs $(\sigma,\tau)$. We find the second equality from the statement.
\end{proof}

\begin{example} \label{ex:gamma3}
    Consider $\gamma_{(0,3,0),(6,1),(2,1),0}$; given a plane $\pi$, meeting a hyperbolic quadric $\mathcal{Q}=\mathcal{Q}^+(5,q)$ in a conic, this number counts the number of secant lines $L$ to $\mathcal{Q}$ such that $\langle L,\pi\rangle$ meets $\mathcal{Q}$ in a non-singular parabolic quadric $\mathcal{Q}(4,q)$. We find that
    \begin{align*}
       \gamma_{(0,3,0),(6,1),(2,1),0}=\sum_{\nu\in\{\pm1\}}\beta^{\nu}_{(0,3,0),(6,1),(5,0)}\gamma_{(0,3,0,\nu),(5,0),(2,1,0)}.
    \end{align*}
    In Example \ref{ex:gamma} we have derived that $\gamma_{(0,3,0,-1),(5,0),(2,1,0)}=\frac{1}{2}q^4(q^2-1)$ and $\gamma_{(0,3,0,1),(5,0),(2,1,0)}=\frac{1}{2}q^2(q^4-q^2+2)$. 
    
    We can use Corollary \ref{cor:betavariations} to find that $\beta^\nu_{(0,3,0),(6,1),(5,0)}=\frac{1}{2}q(q+\nu)$. It follows that
 
    \begin{align*}
        \gamma_{(0,3,0,1),(5,0),(2,1,0)}=\frac{1}{2}q(q+1)\frac{1}{2}q^2(q^4-q^2+2)+ \frac{1}{2}q(q-1)\frac{1}{2}q^4(q^2-1)
        =\frac{1}{2}q^3(q^5-q^3+q+1)\:.
    \end{align*}
\end{example}

\section{The proportion of non-singular trivially intersecting subspaces spanning a non-singular space}\label{sec:proportion}

In this section, we look at the proportion that motivated this research.

\begin{definition}\label{def:rhoquadratic}
    Consider a quadratic form of type $\eps$ on $\F^{n}_{q}$. Let $j,k$ be integers with $0\leq j,k\leq n-1$ and $j+k\leq n$. Let $\delta,\zeta,\eta\in\{0,\pm1\}$ be such that $1\equiv j-\delta\equiv k-\zeta\equiv k+j-\eta\pmod{2}$. Now, let $\mathcal{S}_{(j,\delta),(k,\zeta)}$ be the set of pairs $(\pi,\pi')$ with $\dim(\pi)=j$ and $\dim(\pi')=k$, both $\pi$ and $\pi'$ non-singular, $\pi$ of type $\delta$ and $\pi'$ of type $\zeta$. Let $\mathcal{T}_{(j,\delta),(k,\zeta),\eta}$ be the subset of $\mathcal{S}_{(j,\delta),(k,\zeta)}$ with pairs $(\pi,\pi')$ such that $\dim(\langle\pi,\pi'\rangle)=j+k$ and $\langle\pi,\pi'\rangle$ non-singular of type $\eta$. The proportion $\frac{|\mathcal{T}_{(j,\delta),(k,\zeta),\eta}|}{|\mathcal{S}_{(j,\delta),(k,\zeta)}|}$ is denoted by $\rho_{(j,\delta),(k,\zeta),(n,\eps),\eta}$.
\end{definition}

Note that by definition $\rho_{(j,\delta),(k,\zeta),(n,\eps),\eta}=\rho_{(k,\zeta),(j,\delta),(n,\eps),\eta}$.%

In case $k=0$, we must have that $\zeta=1$ in the previous definition, since the 0-space is hyperbolic. It follows immediately that $\rho_{(j,\delta),(0,1),(n,\eps),\eta}=\frac{1}{2}\left(1+\delta\eta\right)$, i.e.~1 if $\eta=\delta$ and 0 otherwise. So, in the theorems of this section we can assume that $j,k\geq1$.

\begin{theorem}\label{th:rhoquadratic}
	Let $n,j,k$ be integers with $0\leq j,k\leq n-1$ and $j+k\leq n$. Let $\delta,\eps,\zeta,\eta\in\{0,\pm1\}$ be such that $1\equiv n-\eps\equiv j-\delta\equiv k-\zeta\equiv k+j-\eta\pmod{2}$. We have
    \begin{align*}
        \displaystyle\rho_{(j,\delta),(k,\zeta),(n,\eps),\eta}=
        \begin{cases}
            \frac{\gamma_{(0,j,\delta),(n,\eps),(k,\zeta),\eta}}{\alpha_{(0,k,\zeta),(n,\eps)}}&nj\text{ even},\\
            \frac{\sum_{\lambda\in\{\pm1\}}\alpha_{(0,j,0,\lambda),(n,0)}\gamma_{(0,j,0,\lambda),(n,0),(k,\zeta),\eta}}{\alpha_{(0,j,0),(n,0)}\alpha_{(0,k,\zeta),(n,0)}}&nj\text{ odd}.
        \end{cases}
    \end{align*}
\end{theorem}
\begin{proof}
    There are $\alpha_{(0,j,\delta),(n,\eps)}\alpha_{(0,k,\zeta),(n,\eps)}$ couples $(\pi,\pi')$ with $\dim(\pi)=j$ and $\dim(\pi')=k$, both $\pi$ and $\pi'$ non-singular, $\pi$ of type $\delta$ and $\pi'$ of type $\zeta$. 
    \par If $n$ or $j$ is even, we know that the orthogonal group $\PGO^{\eps}(n,q)$ acts transitively on the non-singular $j$-spaces (see Theorem \ref{th:orbitsPGO}). There are $\alpha_{(0,j,\delta),(n,\eps)}$ choices for $\pi$, and given $\pi$, $\gamma_{(0,j,\delta),(n,\eps),(k,\zeta),\eta}$ choices for $\pi'$ with $\dim(\pi')=k$, non-singular of type $\zeta$, and disjoint from $\pi$. Hence, $\rho_{(j,\delta),(k,\zeta),(n,\eps),\eta}=\frac{\alpha_{(0,j,\delta),(n,\eps)}\gamma_{(0,j,\delta),(n,\eps),(k,\zeta),\eta}}{\alpha_{(0,j,\delta),(n,\eps)}\alpha_{(0,k,\zeta),(n,\eps)}}$.
    \par If $n$ and $j$ are both odd, we know that the orthogonal group $\PGO^{\eps}(n,q)$ has two orbits on the non-singular $j$-spaces (necessarily of type 0), corresponding to their perp type (see Theorem \ref{th:orbitsPGO}). Performing the same double counting argument as in the previous case, we now need to sum over the two possibilities for the perp type of the space $\pi$ and we find that $\rho_{(j,0),(k,\zeta),(n,0),0}=\frac{\sum_{\lambda\in\{\pm1\}}\alpha_{(0,j,0,\lambda),(n,0)}\gamma_{(0,j,0,\lambda),(n,\eps),(k,\zeta),\eta}}{\alpha_{(0,j,0),(n,0)}\alpha_{(0,k,\zeta),(n,0)}}$.
\end{proof}

We now provide several examples of this theorem, comparing them with previous results, where possible.

\begin{example}
    Consider the case $j=k=2$, $\delta=\zeta=\eta=1$ and $n=4$. Then we are looking for the proportion of pairs of secant lines to a hyperbolic quadric $\mathcal{Q}^+(3,q)$ which span the entire space among the pairs of secant lines to $\mathcal{Q}^+(3,q)$. By Theorem \ref{th:rhoquadratic}, together with Examples \ref{ex:alfa} and \ref{ex:gamma2} this proportion is $\rho_{(2,1),(2,1),(4,1),1}=\frac{\frac{1}{2}q(q^3+q^2-3q+3)}{\frac{1}{2}q^2(q+1)^2}$. Using this expression, it is easy to check that $\rho_{(2,1),(2,1),(4,1),1}=1-\frac{1}{q}\frac{q^2+4q-3}{q^2+2q+1}$. Since $\frac{q^2+4q-3}{q^2+2q+1}$ reaches it maximum when $q=5$, $\rho_{(2,1),(2,1),(4,1),1}\geq 1-\frac{7}{6}\frac{1}{q}$ when $q\geq 3$. From the results in \cite{gim} the bound $\rho_{(2,1),(2,1),(4,1),1}>1-\frac{3}{2}\frac{1}{q}$ followed. Note that the authors of \cite{gim} only deal with the case that both $j$ and $k$ are even, but they do not impose a condition on the characteristic of the field (only that $q>2$). 
\end{example}

\begin{example}
    Consider the case $j=2$, $\delta=1$, $k=4$, $\zeta=-1$, $n=7$, $\eps=0$, and $\eta=1$. Then we are looking for the proportion of pairs consisting of a secant line to a parabolic quadric $\mathcal{Q}(6,q)$ and a solid intersecting $\mathcal{Q}(6,q)$ in an elliptic quadric $\mathcal{Q}^-(3,q)$, such that this line and solid span a $\mathcal{Q}^+(5,q),$ among the total number of pairs consisting of a secant line and a solid meeting $\mathcal{Q}(6,q)$ in an elliptic quadric $\mathcal{Q}^-(3,q)$. 
    Using first Theorems \ref{th:rhoquadratic} and \ref{th:gammageneral}, and then Theorem \ref{th:alphaorthogonal}, Corollary \ref{cor:betas} and Theorem \ref{th:gammaquadratic}, we find that
    \begin{align*}
        \rho_{(2,1),(4,-1),(7,0),1}&=\frac{\gamma_{(0,2,1),(7,0),(4,-1),1}}{\alpha_{(0,4,-1),(7,0)}}=\frac{\beta_{(0,2,1),(7,0),(6,1)}\gamma_{(0,2,1),(6,1),(4,-1),1}}{\alpha_{(0,4,-1),(7,0)}}\\
        &=\frac{\frac{1}{2}q^{2}\left(q^{2}+1\right)}{\frac{1}{2}q^{6}\left(q^3+1\right)\left(q^3-1\right)}\frac{1}{2}q^{7}\left(\chi_{1,1}(q)\left(1-q^{-2}-q^{-1}+q^{-3}\right)\right)\\
        &=\frac{1}{2}\frac{\left(q^2+1\right)(q-1)^2}{q^{4}+q^{2}+1}\:.
    \end{align*}
    The proportion of subspaces considered in \cite{gnp2} does not make a distinction in between the cases in which the non-singular space spanned by the line and the solid is elliptic or hyperbolic. They consider the proportion of pairs consisting of a secant line to a parabolic quadric and a solid intersecting $\mathcal{Q}(6,q)$ in an elliptic quadric such that this line and solid span a non-singular space, among the total number of such pairs. This proportion equals
    \[
        \frac{\gamma_{(0,2,1),(n,0),(4,-1),1}+\gamma_{(0,2,1),(7,0),(4,-1),-1}}{\alpha_{(0,4,-1),(7,0)}}=\rho_{(2,1),(4,-1),(7,0),1}+\rho_{(2,1),(4,-1),(7,0),-1}\:.
    \]
    Therefore, to compare our results to theirs, we need to compute the proportion $\rho_{(2,1),(4,-1),(7,0),-1}$ as well. Using     again Theorem\ref{th:alphaorthogonal}, Corollary \ref{cor:betas} and Theorems \ref{th:gammaquadratic}, \ref{th:gammageneral} and \ref{th:rhoquadratic}, we find that
    \begin{align*}
        \rho_{(2,1),(4,-1),(7,0),-1}&=\frac{\frac{1}{2}q^{2}\left(q^{2}-1\right)}{\frac{1}{2}q^{6}\left(q^3+1\right)\left(q^3-1\right)}\frac{1}{2}q^{7}\left(\chi_{1,1}(q)\left(1-q^{-2}+q^{-1}+q^{-3}\right)+\chi_{1,0}(q)\cdot2q^{-4}\right)\\
        &=\frac{1}{2}\frac{q^{5}-2q^{3}+2q^{2}-q+2}{q\left(q^{4}+q^{2}+1\right)}\:.
    \end{align*}
    So, the proportion of pairs consisting of a secant line to a parabolic quadric and a solid intersecting $\mathcal{Q}(6,q)$ in an elliptic quadric such that this line and solid span a non-singular space, among the total number of such pairs is
    \begin{align*}
        \rho_{(2,1),(4,-1),(7,0),1}+\rho_{(2,1),(4,-1),(7,0),-1}=\frac{q^5-q^{4}+1}{q\left(q^{4}+q^{2}+1\right)}=1-\frac{1}{q}\:\frac{q^{4}+q^{3}+q-1}{q^{4}+q^{2}+1}\:.
    \end{align*}
    For $q\geq 3$, this proportion is at least $1-c\frac{1}{q}$, with $c=\frac{110}{91}$ an improvement on the value $c=\frac{3}{2}$ found in \cite{gnp2}.
\end{example}

    Let $n,j,k$ be integers with $1\leq j,k\leq n-1$ and $j+k\leq n$, and both $k$ and $n$ even, and $j$ odd. For $\eps,\zeta\in\{\pm1\}$we have

\begin{example}\label{ex:rho1}
    Consider the case $j=3$, $\delta=0$, $k=2$, $n=6$, $\zeta=\eps=1$ and $\eta=0$. Then $\rho_{(3,0),(2,1),(6,1),0}$ is the proportion of the pairs $(\pi,L)$ with $L$ a secant line and $\pi$ a conic plane to hyperbolic quadric $\mathcal{Q}=\mathcal{Q}^+(5,q)$, such that $\langle \pi,L\rangle$ meets $\mathcal{Q}$ a parabolic quadric $\mathcal{Q}(4,q)$, and the total number of pairs of secant lines and conic planes.
    \par We have found in Example \ref{ex:gamma3} that $\gamma_{(0,3,0),(6,1),(2,1),0}=\frac{1}{2}q^3(q^5-q^3+q+1)$. By Theorem \ref{th:alphaorthogonal}, we know that $\alpha_{(0,2,1),(6,1)}=q^4\frac{\psi^+_{2,2}\psi^-_{3,3}}{\psi^+_{0,0}\psi^-_{1,1}}=q^4\frac{(q^2+1)(q^3-1)}{2(q-1)}=\frac{1}{2}q^{4}(q^{2}+1)(q^{2}+q+1)$. Applying Theorem \ref{th:rhoquadratic}, we find that $\rho_{(3,0),(2,1),(6,1),0}=\frac{q^5-q^3+q+1}{q(q^2+1)(q^2+q+1)}$. 
    One can check that $\rho_{(3,0),(2,1),(6,1),0}=1-\frac{1}{q}\frac{q^4+3q^3+q^2-1}{(q^2+1)(q^2+q+1)}>1-\frac{17}{13}\frac{1}{q}$ for all $q\geq 3$.
\end{example}

\begin{example}\label{ex:rho3}
    We consider the case $j=k=3$, $n=7$, $\delta=\eps=\zeta=0$. Then $\rho_{(3,0),(3,0),(7,0),\eta}$ is the proportion of pairs $(\pi,\pi')$ where $\pi$ and $\pi'$ are conic planes of a parabolic quadric $\mathcal{Q}(6,q)$ such that $\langle \pi,\pi'\rangle$ is a hyperplane meeting $\mathcal{Q}(6,q)$ in a non-singular quadric of type $\eta$ among all conic plane pairs. Applying Theorem \ref{th:rhoquadratic} we find that $\rho_{(3,0),(3,0),(7,0),\eta}=\frac{\sum_{\lambda\in\{\pm1\}}\alpha_{(0,3,0,\lambda),(7,0)}\gamma_{(0,3,0,\lambda),(7,0),(3,0),\eta}}{\alpha^2_{(0,3,0),(7,0)}}$. Using Theorem \ref{th:alphaorthogonalodd}, we find that $\alpha_{(0,3,0,1),(7,0)}=\frac{1}{2}q^{6}\left(q^2+1\right)\left(q^4+q^2+1\right)$  and $\alpha_{(0,3,0,-1),(7,0)}=\frac{1}{2}q^6(q^6-1)$. Therefore (or directly via Theorem \ref{th:alphaorthogonal}), $\alpha_{(0,3,0),(7,0)}=q^8\left(q^{4}+q^{2}+1\right)$. Using Theorem \ref{th:gammageneral}, we find that $$\gamma_{(0,3,0,\lambda),(7,0),(3,0),\eta}=\beta_{(0,3,0,\lambda),(7,0),(6,\eta)}\gamma_{(0,3,0)(6,\eta),(3,0),\eta}\:.$$
    By Corollary \ref{cor:betavariations}, $\beta_{(0,3,0,\lambda),(7,0),(6,\eta)}=\frac{1}{2}q(q^2-\lambda)$, and by Theorem \ref{th:gammaquadratic0even},
    \begin{align*}
        \gamma_{(0,3,0)(6,\eta),(3,0),\eta}=q^5\left(\chi_{1,2}(q)+\chi_{1,1}(q)q^{-2}-\eta \chi_{1,1}(q)-\eta \chi_{1,0}(q)q^{-2}\right)\:.
    \end{align*}
    This expression equals $q^3(q^6-q^5-2q^3+2q^2+q-2)$ for $\eta=1$  and $q^4(q^5-q^4+1)$ for $\eta=-1$. 
    We find that 
    \begin{align*}
        &\rho_{(3,0),(3,0),(7,0),1}\\&=\frac{\frac{1}{2}q^6(q^2+1)(q^4+q^2+1)\frac{1}{2}q(q^2-1)+\frac{1}{2}q^6(q^6-1)\frac{1}{2}q(q^2+1)}{q^{16}(q^4+q^2+1)^2}q^3(q^6-q^5-2q^3+2q^2+q-2)\\&=\frac{q^{10}(q^2+1)(q^6-1)(q^6-q^5-2q^3+2q^2+q-2)}{2q^{16}(q^4+q^2+1)^2}\\&=\frac{(q^4-1)(q^6-q^5-2q^3+2q^2+q-2)}{2q^6(q^4+q^2+1)}
    \end{align*}
    and
    \begin{align*}
        \rho_{(3,0),(3,0),(7,0),-1}&=\frac{\frac{1}{2}q^6(q^2+1)(q^4+q^2+1)\frac{1}{2}q(q^2-1)+\frac{1}{2}q^6(q^6-1)\frac{1}{2}q(q^2+1)}{q^{16}(q^4+q^2+1)^2}q^4(q^5-q^4+1)\\&=\frac{q^{11}(q^2+1)(q^6-1)(q^5-q^4+1)}{2q^{16}(q^4+q^2+1)^2}\\&=\frac{(q^4-1)(q^5-q^4+1)}{2q^5(q^4+q^2+1)}
    \end{align*}
    We see that $\rho_{(3,0),(3,0),(7,0),1}+\rho_{(3,0),(3,0),(7,0),-1}=\frac{(q^4-1)(q^6-q^5-q^3+q^2+q-1)}{q^6(q^4+q^2+1)}$.  It follows that the proportion of disjoint conic planes of $\mathcal{Q}(6,q)$ spanning a subspace meeting $\mathcal{Q}(6,q)$ in a non-singular quadric is $1-\frac{1}{q}\frac{q^9+q^8+q^7+q^6-2q^5+q^4-q^3+q^2+q-1}{q^9+q^7+q^5}$ which is at least $ 1-c\frac{1}{q}$ with $c=\frac{28739}{22113}<\frac{13}{10}$ for $q\geq 3$.
\end{example}

We cannot compare the results of Examples \ref{ex:rho1} and \ref{ex:rho3} with previous results since these cases (with simultaneously $j$ or $k$ odd and $n>j+k$) was not discussed in earlier papers.

\begin{remark}
    Similarly to the definition of $\rho_{(j,\delta),(k,\zeta),(n,\eps),\eta}$ as given in Definition \ref{def:rhoquadratic} one could also introduce $\rho_{(j,0,\lambda),(k,\zeta),(n,0),\eta}$, $\rho_{(j,\delta),(k,0,\mu),(n,0),\eta}$ and $\rho_{(j,0,\lambda),(k,0,\mu),(n,0),\eta}$, where one or two of the sets of subspaces considered is restricted to subspaces of certain perp type. We will not discuss this in detail, as it is not relevant for the application we discussed. We will however present an example.
\end{remark}

\begin{example}\label{ex:rho2} 
    Consider the case $j=3$, $k=2$, $\zeta=\lambda=1$ and $n=5$ (so $\delta=\eps=\eta=0$). Then we are looking for the proportion of pairs of conic planes of perp type $1$ and secant lines to a parabolic quadric $\mathcal{Q}(4,q)$ which span the entire space, among the pairs of conic planes of perp type 1 and secant lines to this $\mathcal{Q}(4,q)$. Analogous to Theorem \ref{th:rhoquadratic}, and using Theorem \ref{th:alphaorthogonal} and Example     \ref{ex:gamma} this proportion is
    \begin{align*}
        \rho_{(3,0,1),(2,1),(5,0),0}&=\frac{\gamma_{(0,3,0,1),(5,0),(2,1),0}}{\alpha_{(0,2,1),(5,0)}}=\frac{\frac{1}{2}q^2\left(q^4-q^2+2\right)}{\frac{1}{2}q^3(q+1)\left(q^2+1\right)}=\frac{q^4-q^2+2}{q(q+1)\left(q^2+1\right)}\\
        &=1-\frac{1}{q}\frac{q^3+2q^2+q-2}{(q+1)\left(q^2+1\right)}\:.
    \end{align*}
    Using this expression, it is easy to check that $\rho_{(3,0,1),(2,1),(5,0),0}>1-\frac{23}{20}\frac{1}{q}$ for $q\geq3$. 
\end{example}

\appendix

\section{Appendix: the induction proof calculations}\label{ap:calculations}

In this Appendix, we collect the calculations that were omitted in the proof of Theorem \ref{th:gammaquadratic}. Recall that each of the $4$ different cases based on the parity of $n,j,i$ gave rise to a different recursion relation (described in Lemmas \ref{lem:recursion3}, \ref{lem:recursion4}, \ref{lem:recursion1}, \ref{lem:recursion2}) we give four proofs.

\subsection{\texorpdfstring{$n,j$ odd, $i$ odd (verifying \eqref{rec3})}{n,j,i odd}}

In case $j=i$ (and then automatically also $\delta=1$) the verification is short since only the first term in the recursion relation remains. Note that in the remaining term (in case $j=i$ and only then) there is a factor $\gs{-1}{0}{q^{2}}$. Recall that we set it equal to 1 in Definition \ref{def:gaussianbinomialcoeffient}.

We may assume that $j-i\geq 2$ in the following computation.



The formula \eqref{rec3} now follows since, by Corollary \ref{cor:betas},
\begin{align*}
    \beta_{(0,n-j,\zeta),(n,0),(n-1,\zeta)}&=\frac{1}{2}q^{\frac{1}{2}(j-1)}\left(q^{\frac{1}{2}(j-1)}+1\right)\neq0.
\end{align*}

\subsection{\texorpdfstring{$n,j$ odd, $i$ even (verifying \eqref{rec4})}{n j odd, i even}}

Note that in the proof below, we use the equality $q^{2m}\gs{\frac{j-i-3}{2}}{m}{q^{2}}+\gs{\frac{j-i-3}{2}}{m-1}{q^{2}}=\gs{\frac{j-i-1}{2}}{m}{q^{2}}$ for values $m=0,\dots,\frac{j-i-1}{2}$ thrice. So, in case $j=i+1$, we make us of the equality $\gs{-1}{-1}{q^2}+q^{0}\gs{-1}{0}{q^{2}}=\gs{0}{0}{q^{2}}$. Recall from Definition \ref{def:gaussianbinomialcoeffient} and the note after Lemma \ref{lem:pascalgeneral}, that this is indeed a valid equality.

\begin{align*}
    &\gamma_{(i,j,0,\lambda),(n,0),(n-j,\zeta),0}\beta_{(0,n-j,\zeta),(n,0),(n-1,\zeta)}\\
    &=\frac{1}{2}q^{n-i-1}\left(q^{i}-1\right)\gamma_{(i-1,j-1,0),(n-1,\zeta),(n-j,\zeta),\zeta}\\
    &\qquad+\frac{1}{4}q^{\frac{1}{2}n-\frac{3}{2}}\sum_{\kappa\in\{\pm1\}}\gamma_{(i,j-1,\kappa),(n-1,\zeta),(n-j,\zeta),\zeta}\\&\qquad\qquad\left((q-1)q^{\frac{1}{2}(n-1)-i}+\lambda q^{\frac{1}{2}(j-i-1)}+\zeta\kappa q^{\frac{1}{2}(j-i+1)}+\kappa(q-1)q^{\frac{1}{2}(n-j-i)}+\kappa\lambda+\zeta q\right)\\
    &\qquad+\frac{1}{2}q^{\frac{1}{2}(n-j+i)-1}\left(q^{j-i-1}-1\right)\left(q^{\frac{1}{2}(n-j-i)}-\lambda\right)\gamma_{(i+1,j-1,0),(n-1,\zeta),(n-j,\zeta),\zeta}\\
    &=\frac{1}{2}q^{n-i-1}\left(q^{i}-1\right) \frac{1}{2}q^{(j-1)(n-1)-\frac{5}{4}(j-1)^2}\left(\sum_{m=0}^{\frac{j-i-1}{2}}\chi_{1,\frac{j-1}{2}-m}(q)\gs{\frac{j-i-1}{2}}{m}{q^{2}}q^{m(j+i-n+m-1)}\right.\\&\qquad\qquad+\zeta\:q^{\frac{1}{2}(j-n)}\sum_{m=0}^{\frac{j-i-1}{2}}\chi_{1,\frac{j-1}{2}-m}(q)\gs{\frac{j-i-1}{2}}{m}{q^{2}}q^{m(j+i-n+m-1)}\\&\qquad\qquad+\left.\zeta^{2}q^{-\frac{1}{2}(j-1)}\sum_{m=0}^{\frac{j-i-1}{2}}\chi_{1,\frac{j-1}{2}-m}(q)\gs{\frac{j-i-1}{2}}{m}{q^{2}}q^{m(j+i-n+m-1)}\right)\\
    &\qquad+\frac{1}{4}q^{\frac{1}{2}n-\frac{3}{2}}\sum_{\kappa\in\{\pm1\}}\left((q-1)q^{\frac{1}{2}(n-1)-i}+\lambda q^{\frac{1}{2}(j-i-1)}+\zeta q\right.\\&\qquad\qquad\qquad\left.+\kappa\left(\zeta q^{\frac{1}{2}(j-i+1)}+(q-1)q^{\frac{1}{2}(n-j-i)}+\lambda\right)\right)\\&\qquad\qquad\frac{1}{2}q^{(j-1)(n-1)-\frac{5}{4}(j-1)^2}\left(\sum_{m=0}^{\frac{j-i-1}{2}}\chi_{1,\frac{j-1}{2}-m}(q)\gs{\frac{j-i-1}{2}}{m}{q^{2}}q^{m(j+i-n+m-1)}\right.\\&\qquad\qquad+\zeta\:q^{\frac{1}{2}(j-n)}\sum_{m=0}^{\frac{j-i-1}{2}-1}\chi_{1,\frac{j-1}{2}-m}(q)\gs{\frac{j-i-1}{2}-1}{m}{q^{2}}q^{m(j+i-n+m+1)}\\&\qquad\qquad+\zeta^{2}q^{-\frac{1}{2}(j-1)}\sum_{m=0}^{\frac{j-i-1}{2}}\chi_{1,\frac{j-1}{2}-m}(q)\gs{\frac{j-i-1}{2}-1}{m}{q^{2}}q^{m(j+i-n+m+1)}\\&\qquad\qquad-\zeta\kappa q^{\frac{1}{2}(i-n+1)}\sum_{m=0}^{\frac{j-i-1}{2}-1}\chi_{1,\frac{j-1}{2}-m}(q)\gs{\frac{j-i-1}{2}-1}{m}{q^{2}}q^{m(j+i-n+m+1)}\\&\qquad\qquad\left.+\zeta^{2}\kappa q^{-\frac{1}{2}i}\sum_{m=0}^{\frac{j-i-1}{2}}\chi_{1,\frac{j-1}{2}-m}(q)\gs{\frac{j-i-1}{2}-1}{m-1}{q^{2}}q^{m(j+i-n+m-1)}\right)\\
    &\qquad+\frac{1}{2}q^{\frac{1}{2}(n-j+i)-1}\left(q^{j-i-1}-1\right)\left(q^{\frac{1}{2}(n-j-i)}-\lambda\right)\\&\qquad \qquad \frac{1}{2}q^{(j-1)(n-1)-\frac{5}{4}(j-1)^2}\left(\sum_{m=0}^{\frac{j-i-3}{2}}\chi_{1,\frac{j-1}{2}-m}(q)\gs{\frac{j-i-3}{2}}{m}{q^{2}}q^{m(j+i-n+m+1)}\right.\\&\qquad\qquad+\zeta q^{\frac{1}{2}(j-n)}\sum_{m=0}^{\frac{j-i-3}{2}}\chi_{1,\frac{j-1}{2}-m}(q)\gs{\frac{j-i-3}{2}}{m}{q^{2}}q^{m(j+i-n+m+1)}\\&\qquad\qquad+\left.\zeta^{2}q^{-\frac{1}{2}(j-1)}\sum_{m=0}^{\frac{j-i-3}{2}}\chi_{1,\frac{j-1}{2}-m}(q)\gs{\frac{j-i-3}{2}}{m}{q^{2}}q^{m(j+i-n+m+1)}\right)\\
    &=\frac{1}{4}q^{jn-\frac{5}{4}j^2-n+\frac{3}{2}j-\frac{1}{4}}\left(\vphantom{\sum_{m=0}^{\frac{j-i-1}{2}}}q^{n-i-1}\left(q^{i}-1\right)\left(1+q^{-\frac{1}{2}(j-1)}\right)\right.\\&\qquad\qquad\qquad\sum_{m=0}^{\frac{j-i-1}{2}}\chi_{1,\frac{j-1}{2}-m}(q)\gs{\frac{j-i-1}{2}}{m}{q^{2}}q^{m(j+i-n+m-1)}\\&\qquad\qquad+\zeta q^{\frac{1}{2}(j-n)}q^{n-i-1}\left(q^{i}-1\right)\sum_{m=0}^{\frac{j-i-1}{2}}\chi_{1,\frac{j-1}{2}-m}(q)\gs{\frac{j-i-1}{2}}{m}{q^{2}}q^{m(j+i-n+m-1)}\\
    &\qquad\qquad+q^{\frac{1}{2}n-\frac{3}{2}}\left((q-1)q^{\frac{1}{2}(n-1)-i}+\lambda q^{\frac{1}{2}(j-i-1)}+\zeta q\right)\\&\qquad\qquad\qquad\sum_{m=0}^{\frac{j-i-1}{2}}\chi_{1,\frac{j-1}{2}-m}(q)\gs{\frac{j-i-1}{2}}{m}{q^{2}}q^{m(j+i-n+m-1)}\\&\qquad\qquad+\zeta q^{\frac{1}{2}j-\frac{3}{2}}\left((q-1)q^{\frac{1}{2}(n-1)-i}+\lambda q^{\frac{1}{2}(j-i-1)}+\zeta q\right)\\&\qquad\qquad\qquad\sum_{m=0}^{\frac{j-i-3}{2}}\chi_{1,\frac{j-1}{2}-m}(q)\gs{\frac{j-i-3}{2}}{m}{q^{2}}q^{m(j+i-n+m+1)}\\&\qquad\qquad+q^{\frac{1}{2}n-\frac{1}{2}j-1}\left((q-1)q^{\frac{1}{2}(n-1)-i}+\lambda q^{\frac{1}{2}(j-i-1)}+\zeta q\right)\\&\qquad\qquad\qquad\sum_{m=0}^{\frac{j-i-1}{2}}\chi_{1,\frac{j-1}{2}-m}(q)\gs{\frac{j-i-3}{2}}{m}{q^{2}}q^{m(j+i-n+m+1)}\\&\qquad\qquad-\zeta q^{\frac{1}{2}i-1}\left(\zeta q^{\frac{1}{2}(j-i+1)}+(q-1)q^{\frac{1}{2}(n-j-i)}+\lambda\right)\\&\qquad\qquad\qquad\sum_{m=0}^{\frac{j-i-3}{2}}\chi_{1,\frac{j-1}{2}-m}(q)\gs{\frac{j-i-3}{2}}{m}{q^{2}}q^{m(j+i-n+m+1)}\\&\qquad\qquad+q^{\frac{1}{2}n-\frac{1}{2}i-\frac{3}{2}}\left(\zeta q^{\frac{1}{2}(j-i+1)}+(q-1)q^{\frac{1}{2}(n-j-i)}+\lambda\right)\\&\qquad\qquad\qquad\left.\sum_{m=0}^{\frac{j-i-1}{2}}\chi_{1,\frac{j-1}{2}-m}(q)\gs{\frac{j-i-3}{2}}{m-1}{q^{2}}q^{m(j+i-n+m-1)}\right)\\
    &\qquad \qquad +q^{n-j-1}\left(q^{j-i-1}-1\right)\left(1+q^{-\frac{1}{2}(j-1)}\right)\sum_{m=0}^{\frac{j-i-3}{2}}\chi_{1,\frac{j-1}{2}-m}(q)\gs{\frac{j-i-3}{2}}{m}{q^{2}}q^{m(j+i-n+m+1)}\\&\qquad\qquad+\zeta q^{\frac{1}{2}(n-j-i)}q^{\frac{1}{2}(n-j+i)-1}\left(q^{j-i-1}-1\right)q^{\frac{1}{2}(j-n)}\\&\qquad\qquad\qquad\sum_{m=0}^{\frac{j-i-3}{2}}\chi_{1,\frac{j-1}{2}-m}(q)\gs{\frac{j-i-3}{2}}{m}{q^{2}}q^{m(j+i-n+m+1)}\\&\qquad \qquad -\lambda q^{\frac{1}{2}(n-j+i)-1}\left(q^{j-i-1}-1\right)\left(1+q^{-\frac{1}{2}(j-1)}\right)\sum_{m=0}^{\frac{j-i-3}{2}}\chi_{1,\frac{j-1}{2}-m}(q)\gs{\frac{j-i-3}{2}}{m}{q^{2}}q^{m(j+i-n+m+1)}\\&\qquad\qquad\left.-\lambda\zeta q^{\frac{1}{2}i-1}\left(q^{j-i-1}-1\right)\sum_{m=0}^{\frac{j-i-3}{2}}\chi_{1,\frac{j-1}{2}-m}(q)\gs{\frac{j-i-3}{2}}{m}{q^{2}}q^{m(j+i-n+m+1)}\right)\\
    &=\frac{1}{4}q^{jn-\frac{5}{4}j^2-n+\frac{3}{2}j-\frac{1}{4}}\left(\left(\vphantom{\sum_{m=0}^{\frac{j-i-1}{2}}}q^{n-i-1}\left(q^{i}-1\right)\left(1+q^{-\frac{1}{2}(j-1)}\right)\right.\right.\\&\qquad\qquad\qquad \sum_{m=0}^{\frac{j-i-1}{2}}\chi_{1,\frac{j-1}{2}-m}(q)\gs{\frac{j-i-1}{2}}{m}{q^{2}}q^{m(j+i-n+m-1)}\\&\qquad\qquad+q^{n-i-2}(q-1)\sum_{m=0}^{\frac{j-i-1}{2}}\chi_{1,\frac{j-1}{2}-m}(q)\gs{\frac{j-i-1}{2}}{m}{q^{2}}q^{m(j+i-n+m-1)}\\&\qquad\qquad+q^{\frac{1}{2}j-\frac{1}{2}}\sum_{m=0}^{\frac{j-i-3}{2}}\chi_{1,\frac{j-1}{2}-m}(q)\gs{\frac{j-i-3}{2}}{m}{q^{2}}q^{m(j+i-n+m+1)}\\&\qquad\qquad+q^{n-\frac{1}{2}j-i-\frac{3}{2}}(q-1)\sum_{m=0}^{\frac{j-i-1}{2}}\chi_{1,\frac{j-1}{2}-m}(q)\gs{\frac{j-i-3}{2}}{m}{q^{2}}q^{m(j+i-n+m+1)}\\&\qquad\qquad-q^{\frac{1}{2}j-\frac{1}{2}}\sum_{m=0}^{\frac{j-i-3}{2}}\chi_{1,\frac{j-1}{2}-m}(q)\gs{\frac{j-i-3}{2}}{m}{q^{2}}q^{m(j+i-n+m+1)}\\&\qquad\qquad+q^{n-\frac{1}{2}j-i-\frac{3}{2}}(q-1)\sum_{m=0}^{\frac{j-i-1}{2}}\chi_{1,\frac{j-1}{2}-m}(q)\gs{\frac{j-i-3}{2}}{m-1}{q^{2}}q^{m(j+i-n+m-1)}\\&\qquad \qquad\left. +q^{n-j-1}\left(q^{j-i-1}-1\right)\left(1+q^{-\frac{1}{2}(j-1)}\right)\sum_{m=0}^{\frac{j-i-3}{2}}\chi_{1,\frac{j-1}{2}-m}(q)\gs{\frac{j-i-3}{2}}{m}{q^{2}}q^{m(j+i-n+m+1)}\right)\\
    &\qquad+\zeta\left( q^{\frac{1}{2}n+\frac{1}{2}j-i-1}\left(q^{i}-1\right)\sum_{m=0}^{\frac{j-i-1}{2}}\chi_{1,\frac{j-1}{2}-m}(q)\gs{\frac{j-i-1}{2}}{m}{q^{2}}q^{m(j+i-n+m-1)}\right.\\&\qquad\qquad+ q^{\frac{1}{2}n-\frac{1}{2}}\sum_{m=0}^{\frac{j-i-1}{2}}\chi_{1,\frac{j-1}{2}-m}(q)\gs{\frac{j-i-1}{2}}{m}{q^{2}}q^{m(j+i-n+m-1)}\\&\qquad\qquad+ q^{\frac{1}{2}n+\frac{1}{2}j-i-2}(q-1)\sum_{m=0}^{\frac{j-i-3}{2}}\chi_{1,\frac{j-1}{2}-m}(q)\gs{\frac{j-i-3}{2}}{m}{q^{2}}q^{m(j+i-n+m+1)}\\&\qquad\qquad+ q^{\frac{1}{2}n-\frac{1}{2}j}\sum_{m=0}^{\frac{j-i-1}{2}}\chi_{1,\frac{j-1}{2}-m}(q)\gs{\frac{j-i-3}{2}}{m}{q^{2}}q^{m(j+i-n+m+1)}\\&\qquad\qquad-(q-1)q^{\frac{1}{2}(n-j)-1}\sum_{m=0}^{\frac{j-i-3}{2}}\chi_{1,\frac{j-1}{2}-m}(q)\gs{\frac{j-i-3}{2}}{m}{q^{2}}q^{m(j+i-n+m+1)}\\&\qquad\qquad+ q^{\frac{1}{2}n+\frac{1}{2}j-i-1}\sum_{m=0}^{\frac{j-i-1}{2}}\chi_{1,\frac{j-1}{2}-m}(q)\gs{\frac{j-i-3}{2}}{m-1}{q^{2}}q^{m(j+i-n+m-1)}\\&\qquad\qquad+\left. q^{\frac{1}{2}(n-j)-1}\left(q^{j-i-1}-1\right)\sum_{m=0}^{\frac{j-i-3}{2}}\chi_{1,\frac{j-1}{2}-m}(q)\gs{\frac{j-i-3}{2}}{m}{q^{2}}q^{m(j+i-n+m+1)}\right)\\
    &\qquad+\lambda\left( q^{\frac{1}{2}n+\frac{1}{2}j-\frac{1}{2}i-2}\sum_{m=0}^{\frac{j-i-1}{2}}\chi_{1,\frac{j-1}{2}-m}(q)\gs{\frac{j-i-1}{2}}{m}{q^{2}}q^{m(j+i-n+m-1)}\right.\\&\qquad\qquad+q^{\frac{1}{2}n-\frac{1}{2}i-\frac{3}{2}}\sum_{m=0}^{\frac{j-i-1}{2}}\chi_{1,\frac{j-1}{2}-m}(q)\gs{\frac{j-i-3}{2}}{m}{q^{2}}q^{m(j+i-n+m+1)}\\&\qquad\qquad+q^{\frac{1}{2}n-\frac{1}{2}i-\frac{3}{2}}\sum_{m=0}^{\frac{j-i-1}{2}}\chi_{1,\frac{j-1}{2}-m}(q)\gs{\frac{j-i-3}{2}}{m-1}{q^{2}}q^{m(j+i-n+m-1)}\\&\qquad \qquad -q^{\frac{1}{2}(n-j+i)-1}\left(q^{j-i-1}-1\right)\left(1+q^{-\frac{1}{2}(j-1)}\right)\\&\qquad\qquad\qquad\left. \sum_{m=0}^{\frac{j-i-3}{2}}\chi_{1,\frac{j-1}{2}-m}(q)\gs{\frac{j-i-3}{2}}{m}{q^{2}}q^{m(j+i-n+m+1)}\right)\\
    &\qquad+\zeta\lambda\left( q^{j-\frac{1}{2}i-2}\sum_{m=0}^{\frac{j-i-3}{2}}\chi_{1,\frac{j-1}{2}-m}(q)\gs{\frac{j-i-3}{2}}{m}{q^{2}}q^{m(j+i-n+m+1)}\right.\\&\qquad\qquad- q^{\frac{1}{2}i-1} \sum_{m=0}^{\frac{j-i-3}{2}}\chi_{1,\frac{j-1}{2}-m}(q)\gs{\frac{j-i-3}{2}}{m}{q^{2}}q^{m(j+i-n+m+1)}\\&\qquad\qquad\left.\left.-q^{\frac{1}{2}i-1}\left(q^{j-i-1}-1\right)\sum_{m=0}^{\frac{j-i-3}{2}}\chi_{1,\frac{j-1}{2}-m}(q)\gs{\frac{j-i-3}{2}}{m}{q^{2}}q^{m(j+i-n+m+1)}\right)\right)\\
    &=\frac{1}{4}q^{jn-\frac{5}{4}j^2-n+\frac{3}{2}j-\frac{1}{4}}\left(q^{n-\frac{3}{2}j-\frac{1}{2}}\left(\vphantom{\sum_{m=0}^{\frac{j-i-1}{2}}}q^{j-i}\left(q^{i}-1\right)\left(q^{\frac{1}{2}(j-1)}+1\right)\right.\right.\\&\qquad\qquad\qquad\sum_{m=0}^{\frac{j-i-1}{2}}\chi_{1,\frac{j-1}{2}-m}(q)\gs{\frac{j-i-1}{2}}{m}{q^{2}}q^{m(j+i-n+m-1)}\\&\qquad\qquad+q^{\frac{3}{2}j-i-\frac{3}{2}}(q-1)\sum_{m=0}^{\frac{j-i-1}{2}}\chi_{1,\frac{j-1}{2}-m}(q)\gs{\frac{j-i-1}{2}}{m}{q^{2}}q^{m(j+i-n+m-1)}\\&\qquad\qquad+q^{j-i-1}(q-1)\sum_{m=0}^{\frac{j-i-1}{2}}\chi_{1,\frac{j-1}{2}-m}(q)\left(q^{2m}\gs{\frac{j-i-3}{2}}{m}{q^{2}}+\gs{\frac{j-i-3}{2}}{m-1}{q^{2}}\right)q^{m(j+i-n+m-1)}\\&\qquad\qquad\left. +\left(q^{j-i-1}-1\right)\left(q^{\frac{1}{2}(j-1)}+1\right)\sum_{m=0}^{\frac{j-i-3}{2}}\chi_{1,\frac{j-1}{2}-m}(q)\gs{\frac{j-i-3}{2}}{m}{q^{2}}q^{m(j+i-n+m+1)}\right)\\
    &\qquad+\zeta q^{\frac{1}{2}n-\frac{1}{2}}\left( q^{\frac{1}{2}j-i-\frac{1}{2}}\left(q^{i}-1\right)\sum_{m=0}^{\frac{j-i-1}{2}}\chi_{1,\frac{j-1}{2}-m}(q)\gs{\frac{j-i-1}{2}}{m}{q^{2}}q^{m(j+i-n+m-1)}\right.\\&\qquad\qquad+ \sum_{m=0}^{\frac{j-i-1}{2}}\chi_{1,\frac{j-1}{2}-m}(q)\gs{\frac{j-i-1}{2}}{m}{q^{2}}q^{m(j+i-n+m-1)}\\&\qquad\qquad+ q^{\frac{1}{2}j-i-\frac{1}{2}}\sum_{m=0}^{\frac{j-i-1}{2}}\chi_{1,\frac{j-1}{2}-m}(q)\gs{\frac{j-i-3}{2}}{m}{q^{2}}q^{m(j+i-n+m+1)}\\&\qquad\qquad\left.+ q^{\frac{1}{2}j-i-\frac{1}{2}}\sum_{m=0}^{\frac{j-i-1}{2}}\chi_{1,\frac{j-1}{2}-m}(q)\gs{\frac{j-i-3}{2}}{m-1}{q^{2}}q^{m(j+i-n+m-1)}\right)\\
    &\qquad+\lambda q^{\frac{1}{2}n+\frac{1}{2}i-j-\frac{1}{2}}\left( q^{\frac{3}{2}j-i-\frac{3}{2}}\sum_{m=0}^{\frac{j-i-1}{2}}\chi_{1,\frac{j-1}{2}-m}(q)\gs{\frac{j-i-1}{2}}{m}{q^{2}}q^{m(j+i-n+m-1)}\right.\\&\qquad\qquad+q^{j-i-1}\sum_{m=0}^{\frac{j-i-1}{2}}\chi_{1,\frac{j-1}{2}-m}(q)\left(q^{2m}\gs{\frac{j-i-3}{2}}{m}{q^{2}}+\gs{\frac{j-i-3}{2}}{m-1}{q^{2}}\right)q^{m(j+i-n+m-1)}\\&\qquad \qquad\left. \left. -\left(q^{j-i-1}-1\right)\left(q^{\frac{1}{2}(j-1)}+1\right)\sum_{m=0}^{\frac{j-i-3}{2}}\chi_{1,\frac{j-1}{2}-m}(q)\gs{\frac{j-i-3}{2}}{m}{q^{2}}q^{m(j+i-n+m+1)}\right)\right)\\
    &=\frac{1}{4}q^{jn-\frac{5}{4}j^2-\frac{3}{4}}\left(\left(q^{j-i}\left(q^{i}-1\right)\left(q^{\frac{1}{2}(j-1)}+1\right) \sum_{m=0}^{\frac{j-i-1}{2}}\chi_{1,\frac{j-1}{2}-m}(q)\gs{\frac{j-i-1}{2}}{m}{q^{2}}q^{m(j+i-n+m-1)}\right.\right.\\&\qquad\qquad+q^{j-i-1}(q-1)\left(q^{\frac{1}{2}(j-1)}+1\right)\sum_{m=0}^{\frac{j-i-1}{2}}\chi_{1,\frac{j-1}{2}-m}(q)\gs{\frac{j-i-1}{2}}{m}{q^{2}}q^{m(j+i-n+m-1)}\\&\qquad\qquad\left. +\left(q^{\frac{1}{2}(j-1)}+1\right)\sum_{m=0}^{\frac{j-i-3}{2}}\chi_{1,\frac{j-1}{2}-m}(q)\gs{\frac{j-i-1}{2}}{m}{q^{2}}\left(q^{j-i-1-2m}-1\right)q^{m(j+i-n+m+1)}\right)\\
    &\qquad+\zeta q^{\frac{3}{2}j-\frac{1}{2}n}\left(\left(q^{\frac{1}{2}j-\frac{1}{2}}-q^{\frac{1}{2}j-i-\frac{1}{2}}+1\right)\sum_{m=0}^{\frac{j-i-1}{2}}\chi_{1,\frac{j-1}{2}-m}(q)\gs{\frac{j-i-1}{2}}{m}{q^{2}}q^{m(j+i-n+m-1)}\right.\\&\qquad\qquad\left.+q^{\frac{1}{2}j-i-\frac{1}{2}}\sum_{m=0}^{\frac{j-i-1}{2}}\chi_{1,\frac{j-1}{2}-m}(q)\left(q^{2m}\gs{\frac{j-i-3}{2}}{m}{q^{2}}+\gs{\frac{j-i-3}{2}}{m-1}{q^{2}}\right)q^{m(j+i-n+m-1)}\right)\\
    &\qquad+\lambda q^{-\frac{1}{2}n+\frac{1}{2}i+\frac{1}{2}j}\left(q^{j-i-1}\left(q^{\frac{1}{2}(j-1)}+1\right)\sum_{m=0}^{\frac{j-i-1}{2}}\chi_{1,\frac{j-1}{2}-m}(q)\gs{\frac{j-i-1}{2}}{m}{q^{2}}q^{m(j+i-n+m-1)}\right.\\&\qquad \qquad\left. \left. -\left(q^{\frac{1}{2}(j-1)}+1\right)\sum_{m=0}^{\frac{j-i-1}{2}}\chi_{1,\frac{j-1}{2}-m}(q)\gs{\frac{j-i-1}{2}}{m}{q^{2}}\left(q^{j-i-1-2m}-1\right)q^{m(j+i-n+m+1)}\right)\right)\\
    &=\frac{1}{4}q^{jn-\frac{5}{4}j^2-\frac{3}{4}}\left(q^{\frac{1}{2}(j-1)}+1\right) \left(\sum_{m=0}^{\frac{j-i-1}{2}}\left(q^{j}-q^{2m}\right)\chi_{1,\frac{j-1}{2}-m}(q)\gs{\frac{j-i-1}{2}}{m}{q^{2}}q^{m(j+i-n+m-1)}\right.\\
    &\qquad+\zeta q^{\frac{3}{2}j-\frac{1}{2}n}\sum_{m=0}^{\frac{j-i-1}{2}}\chi_{1,\frac{j-1}{2}-m}(q)\gs{\frac{j-i-1}{2}}{m}{q^{2}}q^{m(j+i-n+m-1)}\\
    &\qquad\left.+\lambda q^{-\frac{1}{2}n+\frac{1}{2}i+\frac{1}{2}j}\sum_{m=0}^{\frac{j-i-1}{2}}\chi_{1,\frac{j-1}{2}-m}(q)\gs{\frac{j-i-1}{2}}{m}{q^{2}}q^{m(j+i-n+m+1)}\right)\\
    &=\frac{1}{4}q^{jn-\frac{5}{4}j^2-\frac{3}{4}}\left(q^{\frac{1}{2}(j-1)}+1\right) \left(\sum_{m=0}^{\frac{j-i-1}{2}}\chi_{1,\frac{j+1}{2}-m}(q)\gs{\frac{j-i-1}{2}}{m}{q^{2}}q^{m(j+i-n+m+1)}\right.\\
    &\qquad+\zeta q^{\frac{3}{2}j-\frac{1}{2}n}\sum_{m=0}^{\frac{j-i-1}{2}}\chi_{1,\frac{j-1}{2}-m}(q)\gs{\frac{j-i-1}{2}}{m}{q^{2}}q^{m(j+i-n+m-1)}\\
    &\qquad\left.+\lambda q^{-\frac{1}{2}n+\frac{1}{2}i+\frac{1}{2}j}\sum_{m=0}^{\frac{j-i-1}{2}}\chi_{1,\frac{j-1}{2}-m}(q)\gs{\frac{j-i-1}{2}}{m}{q^{2}}q^{m(j+i-n+m+1)}\right)
\end{align*}

The formula \eqref{rec4} now follows since, by Corollary \ref{cor:betas},
\begin{align*}
    \beta_{(0,n-j,\zeta),(n,0),(n-1,\zeta)}&=\frac{1}{2}q^{\frac{1}{2}(j-1)}\left(q^{\frac{1}{2}(j-1)}+1\right)\neq0\:.
\end{align*}

\subsection{\texorpdfstring{$n,j,i$ even (verifying \eqref{rec11})}{n,j,i even} }

In case $j=i$ (and then automatically also $\delta=1$) the verification is short since only the first term in the recursion relation remains. Note that in the remaining term (in case $j=i$ and only then) there is a factor $\gs{-1}{0}{q^{2}}$. Recall that we set it equal to 1 in Definition \ref{def:gaussianbinomialcoeffient}.

We may assume that $j-i\geq 2$ in the following computation.



The formula \eqref{rec11} now follows since, by Corollary \ref{cor:betas},
\begin{align*}
    \beta_{(0,n-j,\zeta),(n,\eps),(n-1,0)}&=q^{\frac{1}{2}j-1}\left(q^{\frac{1}{2}j}-\eps\zeta\right)\neq0\:.
\end{align*}

Note that $q^{\frac{1}{2}j}-\eps\zeta\neq0$ since $j\geq2$ is assumed.

\subsection{\texorpdfstring{$n,j$ even, $i$ odd (verifying \eqref{rec2})}{n,j even, i odd}}

\begin{align*}
    &\gamma_{(i,j,0),(n,\eps),(n-j,\zeta),\eps}\beta_{(0,n-j,\zeta),(n,\eps),(n-1,0)}\\
    &=\frac{1}{2}q^{n-i-1}\left(q^{i}-1\right)\sum_{\nu\in\{\pm1\}}\gamma_{(i-1,j-1,0,\nu),(n-1,0),(n-j,\zeta),0}\\
    &\qquad+\frac{1}{2}q^{\frac{1}{2}n-1}\sum_{\kappa\in\{\pm1\}}\left(q^{\frac{1}{2}n-i-1}(q-1)-\eps+\kappa q^{\frac{1}{2}(j-i-1)}\left((q-1)q^{\frac{1}{2}n-j}-\eps\right)\right)\\&\qquad\qquad\qquad\gamma_{(i,j-1,\kappa),(n-1,0),(n-j,\zeta),0}\\
    &\qquad+\frac{1}{2}\left(q^{j-i-1}-1\right)q^{\frac{1}{2}(n-j+i-1)}\sum_{\nu\in\{\pm1\}}\gamma_{(i+1,j-1,0,\nu),(n-1,0),(n-j,\zeta),0}\left(q^{\frac{1}{2}(n-j-i-1)}+\nu\right)\\
    &=\frac{1}{2}q^{n-i-1}\left(q^{i}-1\right)\sum_{\nu\in\{\pm1\}}\frac{1}{2}q^{jn-\frac{5}{4}j^2-n+j}\left(\sum_{m=0}^{\frac{j-i-1}{2}}\chi_{1,\frac{j}{2}-m}(q)\gs{\frac{j-i-1}{2}}{m}{q^{2}}q^{m(j+i-n+m)}\right.\\&\qquad\qquad+\zeta q^{\frac{3}{2}j-\frac{1}{2}n-1}\sum_{m=0}^{\frac{j-i-1}{2}}\chi_{1,\frac{j-2}{2}-m}(q)\gs{\frac{j-i-1}{2}}{m}{q^{2}}q^{m(j+i-n+m-2)}\\&\qquad\qquad\left.+\nu q^{\frac{1}{2}(j+i-n-1)}\sum_{m=0}^{\frac{j-i-1}{2}}\chi_{1,\frac{j-2}{2}-m}(q)\gs{\frac{j-i-1}{2}}{m}{q^{2}}q^{m(j+i-n+m)}\right)\\
    &\qquad+\frac{1}{2}q^{\frac{1}{2}n-1}\sum_{\kappa\in\{\pm1\}}\left(q^{\frac{1}{2}n-i-1}(q-1)-\eps+\kappa q^{\frac{1}{2}(j-i-1)}\left((q-1)q^{\frac{1}{2}n-j}-\eps\right)\right)\\&\qquad\qquad\frac{1}{2}q^{jn-\frac{5}{4}j^2-n+j}\left(\sum_{m=0}^{\frac{j-i-1}{2}}\chi_{1,\frac{j}{2}-m}(q)\gs{\frac{j-i-1}{2}}{m}{q^{2}}q^{m(j+i-n+m)}\right.\\&\qquad\qquad+\zeta q^{\frac{3}{2}j-\frac{1}{2}n-1}\sum_{m=0}^{\frac{j-i-1}{2}-1}\chi_{1,\frac{j-2}{2}-m}(q)\gs{\frac{j-i-1}{2}-1}{m}{q^{2}}q^{m(j+i-n+m)}\\&\qquad\qquad\left.-\zeta\kappa q^{j-\frac{1}{2}n+\frac{1}{2}i-\frac{1}{2}}\sum_{m=0}^{\frac{j-i-1}{2}-1}\chi_{1,\frac{j-2}{2}-m}(q)\gs{\frac{j-i-1}{2}-1}{m}{q^{2}}q^{m(j+i-n+m)}\right)\\
    &\qquad+\frac{1}{2}\left(q^{j-i-1}-1\right)q^{\frac{1}{2}(n-j+i-1)}\sum_{\nu\in\{\pm1\}}\left(q^{\frac{1}{2}(n-j-i-1)}+\nu\right)\frac{1}{2}q^{jn-\frac{5}{4}j^2-n+j}\\&\qquad\qquad\left(\sum_{m=0}^{\frac{j-i-3}{2}}\chi_{1,\frac{j}{2}-m}(q)\gs{\frac{j-i-3}{2}}{m}{q^{2}}q^{m(j+i-n+m+2)}\right.\\&\qquad\qquad+\zeta q^{\frac{3}{2}j-\frac{1}{2}n-1}\sum_{m=0}^{\frac{j-i-3}{2}}\chi_{1,\frac{j-2}{2}-m}(q)\gs{\frac{j-i-3}{2}}{m}{q^{2}}q^{m(j+i-n+m)}\\&\qquad\qquad\left.+\nu q^{\frac{1}{2}(j+i-n+1)}\sum_{m=0}^{\frac{j-i-3}{2}}\chi_{1,\frac{j-2}{2}-m}(q)\gs{\frac{j-i-3}{2}}{m}{q^{2}}q^{m(j+i-n+m+2)}\right)\\
    &=\frac{1}{4}q^{jn-\frac{5}{4}j^2-n+j}\left(2q^{n-i-1}\left(q^{i}-1\right)\sum_{m=0}^{\frac{j-i-1}{2}}\chi_{1,\frac{j}{2}-m}(q)\gs{\frac{j-i-1}{2}}{m}{q^{2}}q^{m(j+i-n+m)}\right.\\&\qquad\qquad+2\zeta \left(q^{i}-1\right)q^{\frac{3}{2}j+\frac{1}{2}n-i-2}\sum_{m=0}^{\frac{j-i-1}{2}}\chi_{1,\frac{j-2}{2}-m}(q)\gs{\frac{j-i-1}{2}}{m}{q^{2}}q^{m(j+i-n+m-2)}\\
    &\qquad\qquad +2q^{n-i-2}(q-1)\sum_{m=0}^{\frac{j-i-1}{2}}\chi_{1,\frac{j}{2}-m}(q)\gs{\frac{j-i-1}{2}}{m}{q^{2}}q^{m(j+i-n+m)}\\&\qquad\qquad+2\zeta q^{\frac{1}{2}n+\frac{3}{2}j-i-3}(q-1)\sum_{m=0}^{\frac{j-i-1}{2}-1}\chi_{1,\frac{j-2}{2}-m}(q)\gs{\frac{j-i-1}{2}-1}{m}{q^{2}}q^{m(j+i-n+m)}\\&\qquad\qquad-2\eps q^{\frac{1}{2}n-1}\sum_{m=0}^{\frac{j-i-1}{2}}\chi_{1,\frac{j}{2}-m}(q)\gs{\frac{j-i-1}{2}}{m}{q^{2}}q^{m(j+i-n+m)}\\&\qquad\qquad-2\eps\zeta q^{\frac{3}{2}j-2}\sum_{m=0}^{\frac{j-i-1}{2}-1}\chi_{1,\frac{j-2}{2}-m}(q)\gs{\frac{j-i-1}{2}-1}{m}{q^{2}}q^{m(j+i-n+m)}\\&\qquad\qquad-2\zeta q^{\frac{1}{2}n+\frac{1}{2}j-2}(q-1) \sum_{m=0}^{\frac{j-i-1}{2}-1}\chi_{1,\frac{j-2}{2}-m}(q)\gs{\frac{j-i-1}{2}-1}{m}{q^{2}}q^{m(j+i-n+m)}\\&\qquad\qquad+2\eps\zeta q^{\frac{3}{2}j-2}\sum_{m=0}^{\frac{j-i-1}{2}-1}\chi_{1,\frac{j-2}{2}-m}(q)\gs{\frac{j-i-1}{2}-1}{m}{q^{2}}q^{m(j+i-n+m)}\\
    &\qquad\qquad +2q^{n-j-1}\left(q^{j-i-1}-1\right)\sum_{m=0}^{\frac{j-i-3}{2}}\chi_{1,\frac{j}{2}-m}(q)\gs{\frac{j-i-3}{2}}{m}{q^{2}}q^{m(j+i-n+m+2)}\\&\qquad\qquad+2\zeta q^{\frac{1}{2}n+\frac{1}{2}j-2}\left(q^{j-i-1}-1\right)\sum_{m=0}^{\frac{j-i-3}{2}}\chi_{1,\frac{j-2}{2}-m}(q)\gs{\frac{j-i-3}{2}}{m}{q^{2}}q^{m(j+i-n+m)}\\&\qquad\qquad\left.+ 2q^{i}\left(q^{j-i-1}-1\right)\sum_{m=0}^{\frac{j-i-3}{2}}\chi_{1,\frac{j-2}{2}-m}(q)\gs{\frac{j-i-3}{2}}{m}{q^{2}}q^{m(j+i-n+m+2)}\right)\\
    &=\frac{1}{2}q^{jn-\frac{5}{4}j^2-n+j}\left(\left(q^{n-i-1}\left(q^{i}-1\right)\sum_{m=0}^{\frac{j-i-1}{2}}\chi_{1,\frac{j}{2}-m}(q)\gs{\frac{j-i-1}{2}}{m}{q^{2}}q^{m(j+i-n+m)}\right.\right.\\&\qquad\qquad+q^{n-i-2}(q-1)\sum_{m=0}^{\frac{j-i-1}{2}}\chi_{1,\frac{j}{2}-m}(q)\gs{\frac{j-i-1}{2}}{m}{q^{2}}q^{m(j+i-n+m)}\\&\qquad\qquad +q^{n-j-1}\left(q^{j-i-1}-1\right)\sum_{m=0}^{\frac{j-i-3}{2}}\chi_{1,\frac{j}{2}-m}(q)\gs{\frac{j-i-3}{2}}{m}{q^{2}}q^{m(j+i-n+m+2)}\\&\qquad\qquad\left.+q^{i}\left(q^{j-i-1}-1\right)\sum_{m=0}^{\frac{j-i-3}{2}}\chi_{1,\frac{j-2}{2}-m}(q)\gs{\frac{j-i-3}{2}}{m}{q^{2}}q^{m(j+i-n+m+2)}\right)\\
    &\qquad+\zeta q^{\frac{1}{2}n+\frac{1}{2}j-2}\left(\left(q^{i}-1\right)q^{j-i}\sum_{m=0}^{\frac{j-i-1}{2}}\chi_{1,\frac{j-2}{2}-m}(q)\gs{\frac{j-i-1}{2}}{m}{q^{2}}q^{m(j+i-n+m-2)}\right.\\&\qquad\qquad+\left(q^{j-i-1}-1\right)(q-1)\sum_{m=0}^{\frac{j-i-3}{2}}\chi_{1,\frac{j-2}{2}-m}(q)\gs{\frac{j-i-3}{2}}{m}{q^{2}}q^{m(j+i-n+m)}\\&\qquad\qquad\left.+ \left(q^{j-i-1}-1\right)\sum_{m=0}^{\frac{j-i-3}{2}}\chi_{1,\frac{j-2}{2}-m}(q)\gs{\frac{j-i-3}{2}}{m}{q^{2}}q^{m(j+i-n+m)}\right)\\
    &\qquad\left.-\eps q^{\frac{1}{2}n-1}\sum_{m=0}^{\frac{j-i-1}{2}}\chi_{1,\frac{j}{2}-m}(q)\gs{\frac{j-i-1}{2}}{m}{q^{2}}q^{m(j+i-n+m)}\right)\\
    &=\frac{1}{2}q^{jn-\frac{5}{4}j^2-n+j}\left(\left(q^{n-i-2}\left(q^{i+1}-1\right)\sum_{m=0}^{\frac{j-i-1}{2}}\chi_{1,\frac{j}{2}-m}(q)\gs{\frac{j-i-1}{2}}{m}{q^{2}}q^{m(j+i-n+m)}\right.\right.\\&\qquad\qquad +q^{n-j-1}\sum_{m=0}^{\frac{j-i-1}{2}}\chi_{1,\frac{j}{2}-m}(q)\left(q^{j-i-1-2m}-1\right)\gs{\frac{j-i-1}{2}}{m}{q^{2}}q^{m(j+i-n+m+2)}\\&\qquad\qquad\left.+q^{i}\sum_{m=0}^{\frac{j-i-3}{2}}\chi_{1,\frac{j-2}{2}-m}(q)\left(q^{j-i-1-2m}-1\right)\gs{\frac{j-i-1}{2}}{m}{q^{2}}q^{m(j+i-n+m+2)}\right)\\
    &\qquad+\zeta q^{\frac{1}{2}n+\frac{1}{2}j-2}\left(\left(q^{i}-1\right)q^{j-i}\sum_{m=0}^{\frac{j-i-1}{2}}\chi_{1,\frac{j-2}{2}-m}(q)\gs{\frac{j-i-1}{2}}{m}{q^{2}}q^{m(j+i-n+m-2)}\right.\\&\qquad\qquad+\left.q\sum_{m=0}^{\frac{j-i-3}{2}}\chi_{1,\frac{j-2}{2}-m}(q)\left(q^{j-i-1-2m}-1\right)\gs{\frac{j-i-1}{2}}{m}{q^{2}}q^{m(j+i-n+m)}\right)\\
    &\qquad\left.-\eps q^{\frac{1}{2}n-1}\sum_{m=0}^{\frac{j-i-1}{2}}\chi_{1,\frac{j}{2}-m}(q)\gs{\frac{j-i-1}{2}}{m}{q^{2}}q^{m(j+i-n+m)}\right)\\
    &=\frac{1}{2}q^{jn-\frac{5}{4}j^2-n+j}\left(\left(q^{n-1}\sum_{m=0}^{\frac{j-i-1}{2}}\chi_{1,\frac{j}{2}-m}(q)\gs{\frac{j-i-1}{2}}{m}{q^{2}}q^{m(j+i-n+m)}\right.\right.\\&\qquad\qquad -q^{n-j-1}\sum_{m=0}^{\frac{j-i-1}{2}}\chi_{1,\frac{j}{2}-m}(q)\gs{\frac{j-i-1}{2}}{m}{q^{2}}q^{m(j+i-n+m+2)}\\&\qquad\qquad\left.+q^{i}\sum_{m=0}^{\frac{j-i-3}{2}}\chi_{1,\frac{j-2}{2}-m}(q)\left(q^{j-i-1-2m}-1\right)\gs{\frac{j-i-1}{2}}{m}{q^{2}}q^{m(j+i-n+m+2)}\right)\\
   &\qquad+\zeta q^{\frac{1}{2}n+\frac{1}{2}j-2}\left(\left(q^{i}-1\right)q^{j-i}\sum_{m=0}^{\frac{j-i-1}{2}}\chi_{1,\frac{j-2}{2}-m}(q)\gs{\frac{j-i-1}{2}}{m}{q^{2}}q^{m(j+i-n+m-2)}\right.\\&\qquad\qquad+\left.q\sum_{m=0}^{\frac{j-i-1}{2}}\chi_{1,\frac{j-2}{2}-m}(q)\left(q^{j-i-1-2m}-1\right)\gs{\frac{j-i-1}{2}}{m}{q^{2}}q^{m(j+i-n+m)}\right)\\
    &\qquad\left.-\eps q^{\frac{1}{2}n-1}\sum_{m=0}^{\frac{j-i-1}{2}}\chi_{1,\frac{j}{2}-m}(q)\gs{\frac{j-i-1}{2}}{m}{q^{2}}q^{m(j+i-n+m)}\right)\\
    &=\frac{1}{2}q^{jn-\frac{5}{4}j^2-n+j}\left(\left(q^{n-1}\sum_{m=0}^{\frac{j-i-1}{2}}\chi_{1,\frac{j}{2}-m}(q)\gs{\frac{j-i-1}{2}}{m}{q^{2}}q^{m(j+i-n+m)}\right.\right.\\&\qquad\qquad -q^{n-j-1}\sum_{m=0}^{\frac{j-i-1}{2}}\chi_{1,\frac{j}{2}-m}(q)\gs{\frac{j-i-1}{2}}{m}{q^{2}}\left((q^{2m}-1)+1\right)q^{m(j+i-n+m)}\\&\qquad\qquad\left.+q^{i}\sum_{m=0}^{\frac{j-i-3}{2}}\chi_{1,\frac{j-2}{2}-m}(q)\left(q^{j-i-1-2m}-1\right)\gs{\frac{j-i-1}{2}}{m}{q^{2}}q^{m(j+i-n+m+2)}\right)\\
    &\qquad+\zeta q^{\frac{1}{2}n+\frac{1}{2}j-2}\sum_{m=0}^{\frac{j-i-1}{2}}(q^{j-i}(q^i-1)q^{-2m}+q^{j-i-2m}-q)\chi_{1,\frac{j-2}{2}-m}(q)\gs{\frac{j-i-1}{2}}{m}{q^{2}}q^{m(j+i-n+m)}\\
    &\qquad\left.-\eps q^{\frac{1}{2}n-1}\sum_{m=0}^{\frac{j-i-1}{2}}\chi_{1,\frac{j}{2}-m}(q)\gs{\frac{j-i-1}{2}}{m}{q^{2}}q^{m(j+i-n+m)}\right)\\
    &=\frac{1}{2}q^{jn-\frac{5}{4}j^2-n+j}\left(\left(q^{n-j-1}\left(q^{j}-1\right)\sum_{m=0}^{\frac{j-i-1}{2}}\chi_{1,\frac{j}{2}-m}(q)\gs{\frac{j-i-1}{2}}{m}{q^{2}}q^{m(j+i-n+m)}\right.\right.\\&\qquad\qquad -q^{i}\sum_{m=1}^{\frac{j-i-1}{2}}\chi_{1,\frac{j}{2}-m}(q)\gs{\frac{j-i-1}{2}}{m-1}{q^{2}}\left(q^{j-i+1-2m}-1\right)q^{(m-1)(j+i-n+m+1)}\\&\qquad\qquad\left.+q^{i}\sum_{m=0}^{\frac{j-i-3}{2}}\chi_{1,\frac{j-2}{2}-m}(q)\left(q^{j-i-1-2m}-1\right)\gs{\frac{j-i-1}{2}}{m}{q^{2}}q^{m(j+i-n+m+2)}\right)\\
    &\qquad+\zeta q^{\frac{1}{2}n+\frac{1}{2}j-2}\sum_{m=0}^{\frac{j-i-1}{2}}q(q^{j-2m-1}-1)\chi_{1,\frac{j-2}{2}-m}(q)\gs{\frac{j-i-1}{2}}{m}{q^{2}}q^{m(j+i-n+m)}\\
    &\qquad\left.-\eps q^{\frac{1}{2}n-1}\sum_{m=0}^{\frac{j-i-1}{2}}\chi_{1,\frac{j}{2}-m}(q)\gs{\frac{j-i-1}{2}}{m}{q^{2}}q^{m(j+i-n+m)}\right)\\
    &=\frac{1}{2}q^{jn-\frac{5}{4}j^2-n+j}\left(\left(q^{n-j-1}\left(q^{j}-1\right)\sum_{m=0}^{\frac{j-i-1}{2}}\chi_{1,\frac{j}{2}-m}(q)\gs{\frac{j-i-1}{2}}{m}{q^{2}}q^{m(j+i-n+m)}\right.\right.\\&\qquad\qquad -q^{i}\sum_{m=0}^{\frac{j-i-3}{2}}\chi_{1,\frac{j}{2}-m-1}(q)\gs{\frac{j-i-1}{2}}{m}{q^{2}}\left(q^{j-i-1-2m}-1\right)q^{m(j+i-n+m+2)}\\&\qquad\qquad\left.+q^{i}\sum_{m=0}^{\frac{j-i-3}{2}}\chi_{1,\frac{j-2}{2}-m}(q)\left(q^{j-i-1-2m}-1\right)\gs{\frac{j-i-1}{2}}{m}{q^{2}}q^{m(j+i-n+m+2)}\right)\\
    &\qquad+\zeta q^{\frac{1}{2}n+\frac{1}{2}j-2}\sum_{m=0}^{\frac{j-i-1}{2}}q\chi_{1,\frac{j}{2}-m}(q)\gs{\frac{j-i-1}{2}}{m}{q^{2}}q^{m(j+i-n+m)}\\
    &\qquad\left.-\eps q^{\frac{1}{2}n-1}\sum_{m=0}^{\frac{j-i-1}{2}}\chi_{1,\frac{j}{2}-m}(q)\gs{\frac{j-i-1}{2}}{m}{q^{2}}q^{m(j+i-n+m)}\right)\\
    &=\frac{1}{2}q^{jn-\frac{5}{4}j^2-n+j}\left( q^{n-j-1}(q^j-1)\sum_{m=0}^{\frac{j-i-3}{2}}\chi_{1,\frac{j}{2}-m}(q)\gs{\frac{j-i-1}{2}}{m}{q^{2}}q^{m(j+i-n+m)}\right.\\
    &\qquad+\zeta q^{\frac{1}{2}n+\frac{1}{2}j-1}\sum_{m=0}^{\frac{j-i-1}{2}}\chi_{1,\frac{j}{2}-m}(q)\gs{\frac{j-i-1}{2}}{m}{q^{2}}q^{m(j+i-n+m)}\\
    &\qquad\left.-\eps q^{\frac{1}{2}n-1}\sum_{m=0}^{\frac{j-i-1}{2}}\chi_{1,\frac{j}{2}-m}(q)\gs{\frac{j-i-1}{2}}{m}{q^{2}}q^{m(j+i-n+m)}\right)\\
    &=\frac{1}{2}q^{jn-\frac{5}{4}j^2}\left(q^{-1}(q^j-1)\sum_{m=0}^{\frac{j-i-3}{2}}\chi_{1,\frac{j}{2}-m}(q)\gs{\frac{j-i-1}{2}}{m}{q^{2}}q^{m(j+i-n+m)}\right.\\
    &\qquad+\zeta q^{-\frac{1}{2}n+\frac{3}{2}j-1}\sum_{m=0}^{\frac{j-i-1}{2}}\chi_{1,\frac{j}{2}-m}(q)\gs{\frac{j-i-1}{2}}{m}{q^{2}}q^{m(j+i-n+m)}\\
    &\qquad\left.-\eps q^{j-\frac{1}{2}n-1}\sum_{m=0}^{\frac{j-i-1}{2}}\chi_{1,\frac{j}{2}-m}(q)\gs{\frac{j-i-1}{2}}{m}{q^{2}}q^{m(j+i-n+m)}\right)
\end{align*}

The formula \eqref{rec2} now follows since, by Corollary \ref{cor:betas},
\begin{align*}
    \beta_{(0,n-j,\zeta),(n,\eps),(n-1,0)}&=q^{\frac{1}{2}j-1}\left(q^{\frac{1}{2}j}-\eps\zeta\right)\neq0\:.
\end{align*}

Note that $q^{\frac{1}{2}j}-\eps\zeta\neq0$ since $j\geq2$ is assumed.

\section{The formulae for \texorpdfstring{$n$}{n} even and \texorpdfstring{$j$}{j} odd}\label{ap:morecalculations}

This appendix contains the details of the proof of Theorem \ref{th:gammaquadratic0even}. Recall that throughout this appendix $n$ is even and $j$ is odd.

\subsection{\texorpdfstring{$i$} is odd}\label{sec:appendixB1}

We assume that $i$ is odd and apply Lemma \ref{lem:recursion1} and Corollary \ref{cor:gammaquadratic0odd}. We distinguish between the cases $n=i+j$ and $n\geq i+j+2$. We first look at the latter case and we find that
\begin{align*}
    &\gamma_{(i,j,\delta),(n,\eps),(n-j,0),\eps}\beta_{(0,n-j,0),(n,\eps),(n-1,0)}\\
    &=\left(q^{i}-1\right)q^{n-i-1}\gamma_{(i-1,j-1,\delta),(n-1,0),(n-j,0),0}\\
    &\qquad+\frac{1}{2}q^{\frac{1}{2}n-2}\left(q^{\frac{1}{2}(j-i)}-\delta\right)\sum_{\nu\in\{\pm1\}}\left((q-1)q^{\frac{1}{2}(n-j-i)}+\nu q+\delta\eps\right)\gamma_{(i,j-1,0,\nu),(n-1,0),(n-j,0),0}\\
    &\qquad+q^{\frac{1}{2}(n-j+i)-1}\left(q^{\frac{1}{2}(j-i-1-\delta)}+1\right)\left(q^{\frac{1}{2}(j-i-1+\delta)}-1\right)\left(q^{\frac{1}{2}(n-j-i)}-\delta\eps\right)\\&\qquad\qquad\qquad\gamma_{(i+1,j-1,\delta),(n-1,0),(n-j,0),0}\\
    &=\left(q^{i}-1\right)q^{n-i-1}q^{\frac{3}{2}jn-\frac{1}{4}n^{2}-\frac{5}{4}j^{2}-\frac{3}{2}n+\frac{3}{2}j-\frac{1}{4}}\left(\sum_{m=0}^{\frac{n-j-i}{2}}\chi_{1,\frac{n-j+1}{2}-m}(q)\gs{\frac{n-j-i}{2}}{m}{q^{2}}q^{m(m-j+i+1)}\right.\\        &\qquad\qquad\left.+\delta\:q^{\frac{1}{2}(i-j)}\sum_{m=0}^{\frac{n-j-i}{2}}\chi_{1,\frac{n-j-1}{2}-m}(q)\gs{\frac{n-j-i}{2}}{m}{q^{2}}q^{m(m-j+i+1)}\right)\\%\gamma_{(i-1,j-1,\delta),(n-1,0),(n-j,0),0}
    &\qquad+\frac{1}{2}q^{\frac{1}{2}n-2}\left(q^{\frac{1}{2}(j-i)}-\delta\right)\sum_{\nu\in\{\pm1\}}\left((q-1)q^{\frac{1}{2}(n-j-i)}+\nu q+\delta\eps\right)\\&\qquad\qquad q^{\frac{3}{2}jn-\frac{1}{4}n^{2}-\frac{5}{4}j^{2}-\frac{3}{2}n+\frac{3}{2}j-\frac{1}{4}}\sum_{m=0}^{\frac{n-j-i}{2}}\chi_{1,\frac{n-j+1}{2}-m}(q)\gs{\frac{n-j-i}{2}}{m}{q^{2}}q^{m(m-j+i+1)}\\%\gamma_{(i,j-1,0,\nu),(n-1,0),(n-j,0),0}
    &\qquad+q^{\frac{1}{2}(n-j+i)-1}\left(q^{\frac{1}{2}(j-i-1-\delta)}+1\right)\left(q^{\frac{1}{2}(j-i-1+\delta)}-1\right)\left(q^{\frac{1}{2}(n-j-i)}-\delta\eps\right)\\&\qquad\qquad q^{\frac{3}{2}jn-\frac{1}{4}n^{2}-\frac{5}{4}j^{2}-\frac{3}{2}n+\frac{3}{2}j-\frac{1}{4}}\left(\sum_{m=0}^{\frac{n-j-i}{2}-1}\chi_{1,\frac{n-j+1}{2}-m}(q)\gs{\frac{n-j-i}{2}-1}{m}{q^{2}}q^{m(m-j+i+3)}\right.\\ &\qquad\qquad\left.+\delta\:q^{\frac{1}{2}(i-j)+1}\sum_{m=0}^{\frac{n-j-i}{2}-1}\chi_{1,\frac{n-j-1}{2}-m}(q)\gs{\frac{n-j-i}{2}-1}{m}{q^{2}}q^{m(m-j+i+3)}\right)\\%\gamma_{(i+1,j-1,\delta),(n-1,0),(n-j,0),0}
    &=q^{\frac{3}{2}jn-\frac{1}{4}n^{2}-\frac{5}{4}j^{2}-\frac{3}{2}n+\frac{3}{2}j-\frac{1}{4}}\left(\left(q^{i}-1\right)q^{n-i-1}\sum_{m=0}^{\frac{n-j-i}{2}}\chi_{1,\frac{n-j+1}{2}-m}(q)\gs{\frac{n-j-i}{2}}{m}{q^{2}}q^{m(m-j+i+1)}\right.\\&\qquad\qquad+\delta\left(q^{i}-1\right)q^{n-\frac{1}{2}(i+j)-1}\sum_{m=0}^{\frac{n-j-i}{2}}\chi_{1,\frac{n-j-1}{2}-m}(q)\gs{\frac{n-j-i}{2}}{m}{q^{2}}q^{m(m-j+i+1)}\\
    &\qquad+q^{\frac{1}{2}n-2}\left(q^{\frac{1}{2}(j-i)}-\delta\right)\left((q-1)q^{\frac{1}{2}(n-j-i)}+\delta\eps\right)\sum_{m=0}^{\frac{n-j-i}{2}}\chi_{1,\frac{n-j+1}{2}-m}(q)\gs{\frac{n-j-i}{2}}{m}{q^{2}}q^{m(m-j+i+1)}\\
    &\qquad+q^{\frac{1}{2}(n-j+i)-1}\left(q^{j-i-1}-1+\delta q^{\frac{1}{2}(j-i)-1}(q-1)\right)\left(q^{\frac{1}{2}(n-j-i)}-\delta\eps\right)\\&\qquad\qquad \left(\sum_{m=0}^{\frac{n-j-i}{2}-1}\chi_{1,\frac{n-j+1}{2}-m}(q)\gs{\frac{n-j-i}{2}-1}{m}{q^{2}}q^{m(m-j+i+3)}\right.\\ &\qquad\qquad\left.\left.+\delta\:q^{\frac{1}{2}(i-j)+1}\sum_{m=0}^{\frac{n-j-i}{2}-1}\chi_{1,\frac{n-j-1}{2}-m}(q)\gs{\frac{n-j-i}{2}-1}{m}{q^{2}}q^{m(m-j+i+3)}\right)\right)\\
    &=q^{\frac{3}{2}jn-\frac{1}{4}n^{2}-\frac{5}{4}j^{2}-\frac{3}{2}n+\frac{3}{2}j-\frac{1}{4}}\left(\left(q^{i}-1\right)q^{n-i-1}\sum_{m=0}^{\frac{n-j-i}{2}}\chi_{1,\frac{n-j+1}{2}-m}(q)\gs{\frac{n-j-i}{2}}{m}{q^{2}}q^{m(m-j+i+1)}\right.\\        &\qquad+\delta\left(q^{i}-1\right)q^{n-\frac{1}{2}(i+j)-1}\sum_{m=0}^{\frac{n-j-i}{2}}\chi_{1,\frac{n-j-1}{2}-m}(q)\gs{\frac{n-j-i}{2}}{m}{q^{2}}q^{m(m-j+i+1)}\\
    &\qquad+\left(q-1\right)q^{n-i-2}\sum_{m=0}^{\frac{n-j-i}{2}}\chi_{1,\frac{n-j+1}{2}-m}(q)\gs{\frac{n-j-i}{2}}{m}{q^{2}}q^{m(m-j+i+1)}\\&\qquad+\delta\eps q^{\frac{1}{2}(n+j-i)-2}\sum_{m=0}^{\frac{n-j-i}{2}}\chi_{1,\frac{n-j+1}{2}-m}(q)\gs{\frac{n-j-i}{2}}{m}{q^{2}}q^{m(m-j+i+1)}\\&\qquad-\delta(q-1)q^{n-\frac{1}{2}(j+i)-2}\sum_{m=0}^{\frac{n-j-i}{2}}\chi_{1,\frac{n-j+1}{2}-m}(q)\gs{\frac{n-j-i}{2}}{m}{q^{2}}q^{m(m-j+i+1)}\\ &\qquad-\eps q^{\frac{1}{2}n-2}\sum_{m=0}^{\frac{n-j-i}{2}}\chi_{1,\frac{n-j+1}{2}-m}(q)\gs{\frac{n-j-i}{2}}{m}{q^{2}}q^{m(m-j+i+1)}\\
    &\qquad+\left(q^{j-i-1}-1\right)q^{n-j-1}\sum_{m=0}^{\frac{n-j-i}{2}-1}\chi_{1,\frac{n-j+1}{2}-m}(q)\gs{\frac{n-j-i}{2}-1}{m}{q^{2}}q^{m(m-j+i+3)}\\&\qquad+\delta\left(q^{j-i-1}-1\right)q^{n-\frac{3}{2}j+\frac{1}{2}i}\sum_{m=0}^{\frac{n-j-i}{2}-1}\chi_{1,\frac{n-j-1}{2}-m}(q)\gs{\frac{n-j-i}{2}-1}{m}{q^{2}}q^{m(m-j+i+3)}\\        &\qquad-\delta\eps q^{\frac{1}{2}(n-j+i)-1}\left(q^{j-i-1}-1\right)\sum_{m=0}^{\frac{n-j-i}{2}-1}\chi_{1,\frac{n-j+1}{2}-m}(q)\gs{\frac{n-j-i}{2}-1}{m}{q^{2}}q^{m(m-j+i+3)}\\&\qquad-\eps \left(q^{j-i-1}-1\right)q^{\frac{1}{2}n-j+i}\sum_{m=0}^{\frac{n-j-i}{2}-1}\chi_{1,\frac{n-j-1}{2}-m}(q)\gs{\frac{n-j-i}{2}-1}{m}{q^{2}}q^{m(m-j+i+3)}\\ &\qquad+\delta (q-1)q^{n-\frac{1}{2}j-\frac{1}{2}i-2}\sum_{m=0}^{\frac{n-j-i}{2}-1}\chi_{1,\frac{n-j+1}{2}-m}(q)\gs{\frac{n-j-i}{2}-1}{m}{q^{2}}q^{m(m-j+i+3)}\\&\qquad+(q-1)q^{n-j-1}\sum_{m=0}^{\frac{n-j-i}{2}-1}\chi_{1,\frac{n-j-1}{2}-m}(q)\gs{\frac{n-j-i}{2}-1}{m}{q^{2}}q^{m(m-j+i+3)}\\&\qquad-\eps(q-1)q^{\frac{1}{2}n-2}\sum_{m=0}^{\frac{n-j-i}{2}-1}\chi_{1,\frac{n-j+1}{2}-m}(q)\gs{\frac{n-j-i}{2}-1}{m}{q^{2}}q^{m(m-j+i+3)}\\&\qquad\left.-\delta\eps (q-1)q^{\frac{1}{2}(n-j+i)-1}\sum_{m=0}^{\frac{n-j-i}{2}-1}\chi_{1,\frac{n-j-1}{2}-m}(q)\gs{\frac{n-j-i}{2}-1}{m}{q^{2}}q^{m(m-j+i+3)}\right)\\
    &=q^{\frac{3}{2}jn-\frac{1}{4}n^{2}-\frac{5}{4}j^{2}-\frac{3}{2}n+\frac{3}{2}j-\frac{1}{4}}\left(q^{n-j-1}\left(\left(q^{i}-1\right)q^{j-i}\sum_{m=0}^{\frac{n-j-i}{2}}\chi_{1,\frac{n-j+1}{2}-m}(q)\gs{\frac{n-j-i}{2}}{m}{q^{2}}q^{m(m-j+i+1)}\right.\right.\\&\qquad\qquad+\left(q-1\right)q^{j-i-1}\sum_{m=0}^{\frac{n-j-i}{2}}\chi_{1,\frac{n-j+1}{2}-m}(q)\gs{\frac{n-j-i}{2}}{m}{q^{2}}q^{m(m-j+i+1)}\\&\qquad\qquad+\left(q^{j-i-1}-1\right)\sum_{m=0}^{\frac{n-j-i}{2}-1}\chi_{1,\frac{n-j+1}{2}-m}(q)\gs{\frac{n-j-i}{2}-1}{m}{q^{2}}q^{m(m-j+i+3)}\\&\qquad\qquad\left.+(q-1)\sum_{m=0}^{\frac{n-j-i}{2}-1}\chi_{1,\frac{n-j-1}{2}-m}(q)\gs{\frac{n-j-i}{2}-1}{m}{q^{2}}q^{m(m-j+i+3)}\right)\\
    &\qquad+\delta q^{n-\frac{1}{2}j-\frac{1}{2}i-2}\left(\left(q^{i}-1\right)q\sum_{m=0}^{\frac{n-j-i}{2}}\chi_{1,\frac{n-j-1}{2}-m}(q)\gs{\frac{n-j-i}{2}}{m}{q^{2}}q^{m(m-j+i+1)}\right.\\&\qquad\qquad-(q-1)\sum_{m=0}^{\frac{n-j-i}{2}}\chi_{1,\frac{n-j+1}{2}-m}(q)\gs{\frac{n-j-i}{2}}{m}{q^{2}}q^{m(m-j+i+1)}\\&\qquad\qquad+\left(q^{j-i-1}-1\right)q^{i-j+2}\sum_{m=0}^{\frac{n-j-i}{2}-1}\chi_{1,\frac{n-j-1}{2}-m}(q)\gs{\frac{n-j-i}{2}-1}{m}{q^{2}}q^{m(m-j+i+3)}\\&\qquad\qquad+\left.(q-1)\sum_{m=0}^{\frac{n-j-i}{2}-1}\chi_{1,\frac{n-j+1}{2}-m}(q)\gs{\frac{n-j-i}{2}-1}{m}{q^{2}}q^{m(m-j+i+3)}\right)\\
    &\qquad-\eps q^{\frac{1}{2}n-j+i}\left( q^{j-i-2}\sum_{m=0}^{\frac{n-j-i}{2}}\chi_{1,\frac{n-j+1}{2}-m}(q)\gs{\frac{n-j-i}{2}}{m}{q^{2}}q^{m(m-j+i+1)}\right.\\&\qquad\qquad+ \left(q^{j-i-1}-1\right)\sum_{m=0}^{\frac{n-j-i}{2}-1}\chi_{1,\frac{n-j-1}{2}-m}(q)\gs{\frac{n-j-i}{2}-1}{m}{q^{2}}q^{m(m-j+i+3)}\\&\left.\qquad\qquad+(q-1)q^{j-i-2}\sum_{m=0}^{\frac{n-j-i}{2}-1}\chi_{1,\frac{n-j+1}{2}-m}(q)\gs{\frac{n-j-i}{2}-1}{m}{q^{2}}q^{m(m-j+i+3)}\right)\\
    &\qquad+\delta\eps q^{\frac{1}{2}(n-j+i)-1}\left(q^{j-i-1}\sum_{m=0}^{\frac{n-j-i}{2}}\chi_{1,\frac{n-j+1}{2}-m}(q)\gs{\frac{n-j-i}{2}}{m}{q^{2}}q^{m(m-j+i+1)}\right.\\&\qquad\qquad-\left(q^{j-i-1}-1\right)\sum_{m=0}^{\frac{n-j-i}{2}-1}\chi_{1,\frac{n-j+1}{2}-m}(q)\gs{\frac{n-j-i}{2}-1}{m}{q^{2}}q^{m(m-j+i+3)}\\&\qquad\qquad\left.\left.-(q-1)\sum_{m=0}^{\frac{n-j-i}{2}-1}\chi_{1,\frac{n-j-1}{2}-m}(q)\gs{\frac{n-j-i}{2}-1}{m}{q^{2}}q^{m(m-j+i+3)}\right)\right)\\
    &=q^{\frac{3}{2}jn-\frac{1}{4}n^{2}-\frac{5}{4}j^{2}-\frac{3}{2}n+\frac{3}{2}j-\frac{1}{4}}\\
    &\qquad\left(q^{n-j-1}\left(\left(q^{i+1}-1\right)q^{j-i-1}\sum_{m=0}^{\frac{n-j-i}{2}}\chi_{1,\frac{n-j+1}{2}-m}(q)\gs{\frac{n-j-i}{2}}{m}{q^{2}}q^{m(m-j+i+1)}\right.\right.\\&\qquad\qquad+\left(q^{j-i-1}-1\right)\sum_{m=0}^{\frac{n-j-i}{2}-1}\chi_{1,\frac{n-j+1}{2}-m}(q)\gs{\frac{n-j-i}{2}-1}{m}{q^{2}}q^{m(m-j+i+3)}\\&\qquad\qquad+q\sum_{m=1}^{\frac{n-j-i}{2}}\chi_{1,\frac{n-j+1}{2}-m}(q)\gs{\frac{n-j-i}{2}-1}{m-1}{q^{2}}q^{m(m-j+i+1)}q^{j-i-2}\\&\qquad\qquad\left.-\sum_{m=0}^{\frac{n-j-i}{2}-1}\chi_{1,\frac{n-j-1}{2}-m}(q)\gs{\frac{n-j-i}{2}-1}{m}{q^{2}}q^{m(m-j+i+3)}\right)\\
    &\qquad+\delta q^{n-\frac{1}{2}j-\frac{1}{2}i-2}\left(\left(q^{i}-1\right)q\sum_{m=0}^{\frac{n-j-i}{2}}\chi_{1,\frac{n-j-1}{2}-m}(q)\gs{\frac{n-j-i}{2}}{m}{q^{2}}q^{m(m-j+i+1)}\right.\\&\qquad\qquad-(q-1)\sum_{m=0}^{\frac{n-j-i}{2}}\chi_{1,\frac{n-j+1}{2}-m}(q)\left(\gs{\frac{n-j-i}{2}}{m}{q^{2}}-q^{2m}\gs{\frac{n-j-i}{2}-1}{m}{q^{2}}\right)q^{m(m-j+i+1)}\\&\qquad\qquad+\left.\left(q-q^{i-j+2}\right)\sum_{m=0}^{\frac{n-j-i}{2}-1}\chi_{1,\frac{n-j-1}{2}-m}(q)\gs{\frac{n-j-i}{2}-1}{m}{q^{2}}q^{m(m-j+i+3)}\right)\\
    &\qquad-\eps q^{\frac{1}{2}n-j+i}\\
    &\qquad\qquad\left(q^{j-i-2}\sum_{m=0}^{\frac{n-j-i}{2}}\chi_{1,\frac{n-j+1}{2}-m}(q)\left(q^{2m}\gs{\frac{n-j-i}{2}-1}{m}{q^{2}}+\gs{\frac{n-j-i}{2}-1}{m-1}{q^{2}}\right)q^{m(m-j+i+1)}\right.\\&\qquad\qquad+ q^{j-i-1}\sum_{m=0}^{\frac{n-j-i}{2}-1}\chi_{1,\frac{n-j-1}{2}-m}(q)\gs{\frac{n-j-i}{2}-1}{m}{q^{2}}q^{m(m-j+i+3)}\\&\qquad\qquad-\sum_{m=1}^{\frac{n-j-i}{2}}\chi_{1,\frac{n-j+1}{2}-m}(q)\gs{\frac{n-j-i}{2}-1}{m-1}{q^{2}}q^{m(m-j+i+1)}q^{j-i-2}\\&\left.\qquad\qquad+(q-1)q^{j-i-2}\sum_{m=0}^{\frac{n-j-i}{2}-1}\chi_{1,\frac{n-j+1}{2}-m}(q)\gs{\frac{n-j-i}{2}-1}{m}{q^{2}}q^{m(m-j+i+3)}\right)\\
    &\qquad+\delta\eps q^{\frac{1}{2}(n-j+i)-1}\\&\qquad\qquad\left(q^{j-i-1}\sum_{m=0}^{\frac{n-j-i}{2}}\chi_{1,\frac{n-j+1}{2}-m}(q)\left(q^{2m}\gs{\frac{n-j-i}{2}-1}{m}{q^{2}}+\gs{\frac{n-j-i}{2}-1}{m-1}{q^{2}}\right)q^{m(m-j+i+1)}\right.\\&\qquad\qquad-\left(q^{j-i-1}-1\right)\sum_{m=0}^{\frac{n-j-i}{2}-1}\chi_{1,\frac{n-j+1}{2}-m}(q)\gs{\frac{n-j-i}{2}-1}{m}{q^{2}}q^{m(m-j+i+3)}\\&\qquad\qquad-q\sum_{m=1}^{\frac{n-j-i}{2}}\chi_{1,\frac{n-j+1}{2}-m}(q)\gs{\frac{n-j-i}{2}-1}{m-1}{q^{2}}q^{m(m-j+i+1)}q^{j-i-2}\\&\qquad\qquad\left.\left.+\sum_{m=0}^{\frac{n-j-i}{2}-1}\chi_{1,\frac{n-j-1}{2}-m}(q)\gs{\frac{n-j-i}{2}-1}{m}{q^{2}}q^{m(m-j+i+3)}\right)\right)\\
    &=q^{\frac{3}{2}jn-\frac{1}{4}n^{2}-\frac{5}{4}j^{2}-\frac{3}{2}n+\frac{3}{2}j-\frac{1}{4}}\\
    &\qquad\left(q^{n-j-1}\left(\left(q^{i+1}-1\right)q^{j-i-1}\sum_{m=0}^{\frac{n-j-i}{2}}\chi_{1,\frac{n-j+1}{2}-m}(q)\gs{\frac{n-j-i}{2}}{m}{q^{2}}q^{m(m-j+i+1)}\right.\right.\\&\qquad\qquad+q^{j-i-1}\sum_{m=0}^{\frac{n-j-i}{2}}\chi_{1,\frac{n-j+1}{2}-m}(q)\left(q^{2m}\gs{\frac{n-j-i}{2}-1}{m}{q^{2}}+\gs{\frac{n-j-i}{2}-1}{m-1}{q^{2}}\right)q^{m(m-j+i+1)}\\&\qquad\qquad\left.-\sum_{m=0}^{\frac{n-j-i}{2}-1}\chi_{1,\frac{n-j-1}{2}-m}(q)\left(\left(q^{n-j-2m}-1\right)+1\right)\gs{\frac{n-j-i}{2}-1}{m}{q^{2}}q^{m(m-j+i+3)}\right)\\
    &\qquad+\delta q^{n-\frac{1}{2}j-\frac{1}{2}i-2}\left(\left(q^{i}-1\right)q\sum_{m=0}^{\frac{n-j-i}{2}}\chi_{1,\frac{n-j-1}{2}-m}(q)\gs{\frac{n-j-i}{2}}{m}{q^{2}}q^{m(m-j+i+1)}\right.\\&\qquad\qquad-(q-1)\sum_{m=1}^{\frac{n-j-i}{2}}\chi_{1,\frac{n-j+1}{2}-m}(q)\gs{\frac{n-j-i}{2}-1}{m-1}{q^{2}}q^{m(m-j+i+1)}\\&\qquad\qquad+\left.\left(q-q^{i-j+2}\right)\sum_{m=0}^{\frac{n-j-i}{2}-1}\chi_{1,\frac{n-j-1}{2}-m}(q)\gs{\frac{n-j-i}{2}-1}{m}{q^{2}}q^{m(m-j+i+3)}\right)\\
    &\qquad-\eps q^{\frac{1}{2}n-j+i}\:q^{j-i-1}\sum_{m=0}^{\frac{n-j-i}{2}-1}\chi_{1,\frac{n-j-1}{2}-m}(q)\left(1+\left(q^{n-j-2m}-1\right)\right)\gs{\frac{n-j-i}{2}-1}{m}{q^{2}}q^{m(m-j+i+3)}\\
    &\qquad+\left.\delta\eps q^{\frac{1}{2}(n-j+i)-1}\sum_{m=0}^{\frac{n-j-i}{2}-1}\chi_{1,\frac{n-j-1}{2}-m}(q)\left(\left(q^{n-j-2m}-1\right)+1\right)\gs{\frac{n-j-i}{2}-1}{m}{q^{2}}q^{m(m-j+i+3)}\right)\\
    &=q^{\frac{3}{2}jn-\frac{1}{4}n^{2}-\frac{5}{4}j^{2}-\frac{3}{2}n+\frac{3}{2}j-\frac{1}{4}}\left(q^{n-j-1}\left(q^{j}\sum_{m=0}^{\frac{n-j-i}{2}}\chi_{1,\frac{n-j+1}{2}-m}(q)\gs{\frac{n-j-i}{2}}{m}{q^{2}}q^{m(m-j+i+1)}\right.\right.\\&\qquad\qquad\left.-q^{n-j}\sum_{m=0}^{\frac{n-j-i}{2}-1}\chi_{1,\frac{n-j-1}{2}-m}(q)\gs{\frac{n-j-i}{2}-1}{m}{q^{2}}q^{m(m-j+i+1)}\right)\\
    &\qquad+\delta q^{n-\frac{1}{2}j-\frac{1}{2}i-2}\left(\left(q^{i}-1\right)q\sum_{m=0}^{\frac{n-j-i}{2}}\chi_{1,\frac{n-j-1}{2}-m}(q)\gs{\frac{n-j-i}{2}}{m}{q^{2}}q^{m(m-j+i+1)}\right.\\&\qquad\qquad-(q-1)\sum_{m=0}^{\frac{n-j-i}{2}-1}\chi_{1,\frac{n-j-1}{2}-m}(q)\gs{\frac{n-j-i}{2}-1}{m}{q^{2}}q^{m(m-j+i+3)}q^{i-j+2}\\&\qquad\qquad+\left.\left(q-q^{i-j+2}\right)\sum_{m=0}^{\frac{n-j-i}{2}-1}\chi_{1,\frac{n-j-1}{2}-m}(q)\gs{\frac{n-j-i}{2}-1}{m}{q^{2}}q^{m(m-j+i+3)}\right)\\
    &\qquad-\eps q^{\frac{3}{2}n-j-1}\left(\sum_{m=0}^{\frac{n-j-i}{2}-1}\chi_{1,\frac{n-j-1}{2}-m}(q)\gs{\frac{n-j-i}{2}-1}{m}{q^{2}}q^{m(m-j+i+1)}\right)\\
    &\qquad+\left.\delta\eps q^{\frac{3}{2}n-\frac{3}{2}j+\frac{1}{2}i-1}\left(\sum_{m=0}^{\frac{n-j-i}{2}-1}\chi_{1,\frac{n-j-1}{2}-m}(q)\gs{\frac{n-j-i}{2}-1}{m}{q^{2}}q^{m(m-j+i+1)}\right)\right)\\
    &=q^{\frac{3}{2}jn-\frac{1}{4}n^{2}-\frac{5}{4}j^{2}-\frac{3}{2}n+\frac{3}{2}j-\frac{1}{4}}\left(q^{n-j-1}\left(q^{j}\sum_{m=0}^{\frac{n-j-i}{2}}\chi_{1,\frac{n-j+1}{2}-m}(q)\gs{\frac{n-j-i}{2}}{m}{q^{2}}q^{m(m-j+i+1)}\right.\right.\\&\qquad\qquad\left.-q^{n-j}\sum_{m=1}^{\frac{n-j-i}{2}}\chi_{1,\frac{n-j+1}{2}-m}(q)\gs{\frac{n-j-i}{2}-1}{m-1}{q^{2}}q^{m(m-j+i-1)}q^{j-i}\right)\\
    &\qquad+\delta q^{n-\frac{1}{2}j-\frac{1}{2}i-2}\left(\left(q^{i}-1\right)q\sum_{m=0}^{\frac{n-j-i}{2}}\chi_{1,\frac{n-j-1}{2}-m}(q)\gs{\frac{n-j-i}{2}}{m}{q^{2}}q^{m(m-j+i+1)}\right.\\&\qquad\qquad-\left.q\left(q^{i-j+2}-1\right)\sum_{m=0}^{\frac{n-j-i}{2}-1}\chi_{1,\frac{n-j-1}{2}-m}(q)\gs{\frac{n-j-i}{2}-1}{m}{q^{2}}q^{m(m-j+i+3)}\right)\\
    &\qquad-\left.\eps q^{\frac{3}{2}n-j-1}\left(1-\delta q^{\frac{1}{2}(i-j)}\right)\sum_{m=0}^{\frac{n-j-i}{2}-1}\chi_{1,\frac{n-j-1}{2}-m}(q)\gs{\frac{n-j-i}{2}-1}{m}{q^{2}}q^{m(m-j+i+1)}\right)\\
    &=q^{\frac{3}{2}jn-\frac{1}{4}n^{2}-\frac{5}{4}j^{2}-\frac{3}{2}n+\frac{3}{2}j-\frac{1}{4}}\\
    &\qquad \left(q^{n-j-1}\left(q^{j}\sum_{m=0}^{\frac{n-j-i}{2}}\chi_{1,\frac{n-j+1}{2}-m}(q)\left(\gs{\frac{n-j-i}{2}}{m}{q^{2}}-q^{n-j-i-2m}\gs{\frac{n-j-i}{2}-1}{m-1}{q^{2}}\right)q^{m(m-j+i+1)}\right)\right.\\
    &\qquad+\delta q^{n-\frac{1}{2}j-\frac{1}{2}i-1}\left(q^{i}\sum_{m=0}^{\frac{n-j-i}{2}}\chi_{1,\frac{n-j-1}{2}-m}(q)\gs{\frac{n-j-i}{2}}{m}{q^{2}}q^{m(m-j+i+1)}\right.\\&\qquad\qquad-\sum_{m=0}^{\frac{n-j-i}{2}}\chi_{1,\frac{n-j-1}{2}-m}(q)\left(\gs{\frac{n-j-i}{2}}{m}{q^{2}}-q^{2m}\gs{\frac{n-j-i}{2}-1}{m}{q^{2}}\right)q^{m(m-j+i+1)}\\&\qquad\qquad-\left.q^{i-j+2}\sum_{m=1}^{\frac{n-j-i}{2}}\chi_{1,\frac{n-j+1}{2}-m}(q)\gs{\frac{n-j-i}{2}-1}{m-1}{q^{2}}q^{m(m-j+i+1)}q^{j-i-2}\right)\\
    &\qquad-\left.\eps q^{\frac{3}{2}n-j-1}\left(1-\delta q^{\frac{1}{2}(i-j)}\right)\sum_{m=0}^{\frac{n-j-i}{2}-1}\chi_{1,\frac{n-j-1}{2}-m}(q)\gs{\frac{n-j-i}{2}-1}{m}{q^{2}}q^{m(m-j+i+1)}\right)\\
    &=q^{\frac{3}{2}jn-\frac{1}{4}n^{2}-\frac{5}{4}j^{2}-\frac{3}{2}n+\frac{3}{2}j-\frac{1}{4}}\left(q^{n-1}\sum_{m=0}^{\frac{n-j-i}{2}-1}\chi_{1,\frac{n-j+1}{2}-m}(q)\gs{\frac{n-j-i}{2}-1}{m}{q^{2}}q^{m(m-j+i+1)}\right.\\
    &\qquad+\delta q^{n-\frac{1}{2}j-\frac{1}{2}i-1}\left(q^{i}\sum_{m=0}^{\frac{n-j-i}{2}}\chi_{1,\frac{n-j-1}{2}-m}(q)\gs{\frac{n-j-i}{2}}{m}{q^{2}}q^{m(m-j+i+1)}\right.\\&\qquad\qquad-\left.\sum_{m=0}^{\frac{n-j-i}{2}}\chi_{1,\frac{n-j-1}{2}-m}(q)\left(1+\left(q^{n-j-2m}-1\right)\right)\gs{\frac{n-j-i}{2}-1}{m-1}{q^{2}}q^{m(m-j+i+1)}\right)\\
    &\qquad-\left.\eps q^{\frac{3}{2}n-j-1}\left(1-\delta q^{\frac{1}{2}(i-j)}\right)\sum_{m=0}^{\frac{n-j-i}{2}-1}\chi_{1,\frac{n-j-1}{2}-m}(q)\gs{\frac{n-j-i}{2}-1}{m}{q^{2}}q^{m(m-j+i+1)}\right)\\
    &=q^{\frac{3}{2}jn-\frac{1}{4}n^{2}-\frac{5}{4}j^{2}-\frac{1}{2}n+\frac{3}{2}j-\frac{5}{4}}\left(\sum_{m=0}^{\frac{n-j-i}{2}-1}\chi_{1,\frac{n-j+1}{2}-m}(q)\gs{\frac{n-j-i}{2}-1}{m}{q^{2}}q^{m(m-j+i+1)}\right.\\
    &\qquad+\delta q^{-\frac{1}{2}j-\frac{1}{2}i}\:q^{i}\sum_{m=0}^{\frac{n-j-i}{2}}\chi_{1,\frac{n-j-1}{2}-m}(q)\left(\gs{\frac{n-j-i}{2}}{m}{q^{2}}-q^{n-j-i-2m}\gs{\frac{n-j-i}{2}-1}{m-1}{q^{2}}\right)q^{m(m-j+i+1)}\\
    &\qquad-\left.\eps q^{\frac{1}{2}n-j}\left(1-\delta q^{\frac{1}{2}(i-j)}\right)\sum_{m=0}^{\frac{n-j-i}{2}-1}\chi_{1,\frac{n-j-1}{2}-m}(q)\gs{\frac{n-j-i}{2}-1}{m}{q^{2}}q^{m(m-j+i+1)}\right)\\
    &=q^{\frac{3}{2}jn-\frac{1}{4}n^{2}-\frac{5}{4}j^{2}-\frac{1}{2}n+\frac{3}{2}j-\frac{5}{4}}\left(\sum_{m=0}^{\frac{n-j-i}{2}-1}\chi_{1,\frac{n-j+1}{2}-m}(q)\gs{\frac{n-j-i}{2}-1}{m}{q^{2}}q^{m(m-j+i+1)}\right.\\
    &\qquad+\delta q^{\frac{1}{2}i-\frac{1}{2}j}\sum_{m=0}^{\frac{n-j-i}{2}-1}\chi_{1,\frac{n-j-1}{2}-m}(q)\gs{\frac{n-j-i}{2}-1}{m}{q^{2}}q^{m(m-j+i+1)}\\
    &\qquad-\left.\eps q^{\frac{1}{2}n-j}\left(1-\delta q^{\frac{1}{2}(i-j)}\right)\sum_{m=0}^{\frac{n-j-i}{2}-1}\chi_{1,\frac{n-j-1}{2}-m}(q)\gs{\frac{n-j-i}{2}-1}{m}{q^{2}}q^{m(m-j+i+1)}\right)\\
    &=q^{\frac{3}{2}jn-\frac{1}{4}n^{2}-\frac{5}{4}j^{2}-\frac{1}{2}n+\frac{3}{2}j-\frac{5}{4}}\left(\sum_{m=0}^{\frac{n-j-i}{2}}\chi_{1,\frac{n-j+1}{2}-m}(q)\gs{\frac{n-j-i}{2}-1}{m}{q^{2}}q^{m(m-j+i+1)}\right.\\
    &\qquad+\left.\left(\delta q^{\frac{1}{2}i-\frac{1}{2}j}-\eps q^{\frac{1}{2}n-j}+\delta\eps q^{\frac{1}{2}n-\frac{3}{2}j+\frac{1}{2}i}\right)\sum_{m=0}^{\frac{n-j-i}{2}-1}\chi_{1,\frac{n-j-1}{2}-m}(q)\gs{\frac{n-j-i}{2}-1}{m}{q^{2}}q^{m(m-j+i+1)}\right)\:.
\end{align*}

Now we look at the case $n=i+j$. Note that necessarily $\delta=\eps$. Again, using Lemma \ref{lem:recursion1} and Corollary \ref{cor:gammaquadratic0odd} we find that
\begin{align*}
    &\gamma_{(n-j,j,\eps),(n,\eps),(n-j,0),\eps}\beta_{(0,n-j,0),(n,\eps),(n-1,0)}\\
    &=\left(q^{n-j}-1\right)q^{j-1}\gamma_{(n-j-1,j-1,\eps),(n-1,0),(n-j,0),0}\\
    &\qquad+\frac{1}{2}q^{\frac{1}{2}n-2}\left(q^{\frac{1}{2}(2j-n)}-\eps\right)\sum_{\nu\in\{\pm1\}}\left((q-1)+\nu q+\eps^{2}\right)\gamma_{(n-j,j-1,0,\nu),(n-1,0),(n-j,0),0}\\
    &\qquad+q^{(n-j)-1}\left(q^{\frac{1}{2}(2j-n-1-\eps)}+1\right)\left(q^{\frac{1}{2}(2j-n-1+\eps)}-1\right)\left(q^{0}-\eps^{2}\right)\\&\qquad\qquad\qquad\gamma_{(n-j+1,j-1,\delta),(n-1,0),(n-j,0),0}\\
    &=\left(q^{n-j}-1\right)q^{j-1}q^{\frac{3}{2}jn-\frac{1}{4}n^{2}-\frac{5}{4}j^{2}-\frac{3}{2}n+\frac{3}{2}j-\frac{1}{4}}\left(\chi_{1,\frac{n-j+1}{2}}(q)+\eps q^{\frac{1}{2}n-j}\chi_{1,\frac{n-j-1}{2}}(q)\right)\\%\gamma_{(i-1,j-1,\delta),(n-1,0),(n-j,0),0}
    &\qquad+\frac{1}{2}q^{\frac{1}{2}n-2}\left(q^{j-\frac{1}{2}n}-\eps\right)\sum_{\nu\in\{\pm1\}}q\left(1+\nu\right) q^{\frac{3}{2}jn-\frac{1}{4}n^{2}-\frac{5}{4}j^{2}-\frac{3}{2}n+\frac{3}{2}j-\frac{1}{4}}\chi_{1,\frac{n-j+1}{2}}(q)\\%\gamma_{(i,j-1,0,\nu),(n-1,0),(n-j,0),0}
    &=q^{\frac{3}{2}jn-\frac{1}{4}n^{2}-\frac{5}{4}j^{2}-\frac{3}{2}n+\frac{3}{2}j-\frac{1}{4}}q^{j-1}\left(\left(q^{n-j}-1\right)\left(\chi_{1,\frac{n-j+1}{2}}(q)+\eps q^{\frac{1}{2}n-j}\chi_{1,\frac{n-j-1}{2}}(q)\right)\right.\\
    &\qquad\left.+q^{\frac{1}{2}n-j}\left(q^{j-\frac{1}{2}n}-\eps\right) \chi_{1,\frac{n-j+1}{2}}(q)\right)\\
    &=q^{\frac{3}{2}jn-\frac{1}{4}n^{2}-\frac{5}{4}j^{2}-\frac{3}{2}n+\frac{5}{2}j-\frac{5}{4}}\left(q^{n-j}\chi_{1,\frac{n-j+1}{2}}(q)+\eps q^{\frac{1}{2}n-j}\left(\left(q^{n-j}-1\right)\chi_{1,\frac{n-j-1}{2}}(q)-\chi_{1,\frac{n-j+1}{2}}(q)\right)\right)\\
    &=q^{\frac{3}{2}jn-\frac{1}{4}n^{2}-\frac{5}{4}j^{2}-\frac{3}{2}n+\frac{5}{2}j-\frac{5}{4}}\left(q^{n-j}\chi_{1,\frac{n-j+1}{2}}(q)+\eps q^{\frac{1}{2}n-j}\left(\chi_{1,\frac{n-j+1}{2}}(q)-\chi_{1,\frac{n-j+1}{2}}(q)\right)\right)\\
    &=q^{\frac{3}{2}jn-\frac{1}{4}n^{2}-\frac{5}{4}j^{2}-\frac{1}{2}n+\frac{3}{2}j-\frac{5}{4}}\chi_{1,\frac{n-j+1}{2}}(q)\:,
\end{align*}
which agrees with the formula for the general case since $\gs{-1}{0}{q^{2}}=1$.

\subsection{\texorpdfstring{$i$} is even}\label{sec:appendixB2}

We assume that $i$ is even and apply Lemma \ref{lem:recursion2} and Corollary \ref{cor:gammaquadratic0odd}. We find that
\begin{align*}
    &\gamma_{(i,j,0),(n,\eps),(n-j,0),\eps}\beta_{(0,n-j,0),(n,\eps),(n-1,0)}\\
    &=\frac{1}{2}q^{n-i-1}\left(q^{i}-1\right)\sum_{\nu\in\{\pm1\}}\gamma_{(i-1,j-1,0,\nu),(n-1,0),(n-j,0),0}\\
    &\qquad+\frac{1}{2}q^{\frac{1}{2}n-1}\sum_{\kappa\in\{\pm1\}}\left(q^{\frac{1}{2}n-i-1}(q-1)-\eps+\kappa q^{\frac{1}{2}(j-i-1)}\left((q-1)q^{\frac{1}{2}n-j}-\eps\right)\right)\\&\qquad\qquad\qquad\qquad\qquad\gamma_{(i,j-1,\kappa),(n-1,0),(n-j,0),0}\\
    &\qquad+\frac{1}{2}\left(q^{j-i-1}-1\right)q^{\frac{1}{2}(n-j+i-1)}\sum_{\nu\in\{\pm1\}}\gamma_{(i+1,j-1,0,\nu),(n-1,0),(n-j,0),0}\left(q^{\frac{1}{2}(n-j-i-1)}+\nu\right)\\
    &=\frac{1}{2}q^{n-i-1}\left(q^{i}-1\right)\sum_{\nu\in \{\pm 1\}} q^{\frac{3}{2}jn-\frac{1}{4}n^{2}-\frac{5}{4}j^{2}-\frac{3}{2}n+\frac{3}{2}j-\frac{1}{4}}\sum_{m=0}^{\frac{n-j-i+1}{2}}\chi_{1,\frac{n-j+1}{2}-m}(q)\gs{\frac{n-j-i+1}{2}}{m}{q^{2}}q^{m(m-j+i)}\\
    &\qquad+\frac{1}{2}q^{\frac{1}{2}n-1}\sum_{\kappa\in\{\pm1\}}\left(q^{\frac{1}{2}n-i-1}(q-1)-\eps+\kappa q^{\frac{1}{2}(j-i-1)}\left((q-1)q^{\frac{1}{2}n-j}-\eps\right)\right)\\& \qquad\qquad q^{\frac{3}{2}jn-\frac{1}{4}n^{2}-\frac{5}{4}j^{2}-\frac{3}{2}n+\frac{3}{2}j-\frac{1}{4}}\left(\sum_{m=0}^{\frac{n-j-i-1}{2}}\chi_{1,\frac{n-j+1}{2}-m}(q)\gs{\frac{n-j-i-1}{2}}{m}{q^{2}}q^{m(m-j+i+2)}\right.\\ &\qquad\qquad\qquad\left.+\kappa\:q^{\frac{1}{2}(i-j+1)}\sum_{m=0}^{\frac{n-j-i-1}{2}}\chi_{1,\frac{n-j-1}{2}-m}(q)\gs{\frac{n-j-i-1}{2}}{m}{q^{2}}q^{m(m-j+i+2)}\right)\\
    &\qquad+\frac{1}{2}\left(q^{j-i-1}-1\right)q^{\frac{1}{2}(n-j+i-1)}\sum_{\nu\in\{\pm1\}}\left(q^{\frac{1}{2}(n-j-i-1)}+\nu\right)\\& \qquad \qquad q^{\frac{3}{2}jn-\frac{1}{4}n^{2}-\frac{5}{4}j^{2}-\frac{3}{2}n+\frac{3}{2}j-\frac{1}{4}}\sum_{m=0}^{\frac{n-j-i-1}{2}}\chi_{1,\frac{n-j+1}{2}-m}(q)\gs{\frac{n-j-i-1}{2}}{m}{q^{2}}q^{m(m-j+i+2)}\\
    &=q^{\frac{3}{2}jn-\frac{1}{4}n^{2}-\frac{5}{4}j^{2}-\frac{3}{2}n+\frac{3}{2}j-\frac{1}{4}}\left(q^{n-i-1}\left(q^{i}-1\right) \sum_{m=0}^{\frac{n-j-i+1}{2}}\chi_{1,\frac{n-j+1}{2}-m}(q)\gs{\frac{n-j-i+1}{2}}{m}{q^{2}}q^{m(m-j+i)}\right.\\
    &\qquad+q^{\frac{1}{2}n-1}\left( \left(q^{\frac{1}{2}n-i-1}(q-1)-\eps\right)\sum_{m=0}^{\frac{n-j-i-1}{2}}\chi_{1,\frac{n-j+1}{2}-m}(q)\gs{\frac{n-j-i-1}{2}}{m}{q^{2}}q^{m(m-j+i+2)}\right.\\&\left. \qquad\qquad +\left((q-1)q^{\frac{1}{2}n-j}-\eps\right)\sum_{m=0}^{\frac{n-j-i-1}{2}}\chi_{1,\frac{n-j-1}{2}-m}(q)\gs{\frac{n-j-i-1}{2}}{m}{q^{2}}q^{m(m-j+i+2)}\right)\\
    &\qquad+\left.\left(q^{j-i-1}-1\right)q^{\frac{1}{2}(n-j+i-1)}q^{\frac{1}{2}(n-j-i-1)} \sum_{m=0}^{\frac{n-j-i-1}{2}}\chi_{1,\frac{n-j+1}{2}-m}(q)\gs{\frac{n-j-i-1}{2}}{m}{q^{2}}q^{m(m-j+i+2)}\right)\\
    &=q^{\frac{3}{2}jn-\frac{1}{4}n^{2}-\frac{5}{4}j^{2}-\frac{3}{2}n+\frac{3}{2}j-\frac{1}{4}}\left(\left(q^{n-i-1}\left(q^{i}-1\right) \sum_{m=0}^{\frac{n-j-i+1}{2}}\chi_{1,\frac{n-j+1}{2}-m}(q)\gs{\frac{n-j-i+1}{2}}{m}{q^{2}}q^{m(m-j+i)}\right.\right.\\&\qquad\qquad+q^{n-i-2}(q-1)\sum_{m=0}^{\frac{n-j-i-1}{2}}\chi_{1,\frac{n-j+1}{2}-m}(q)\gs{\frac{n-j-i-1}{2}}{m}{q^{2}}q^{m(m-j+i+2)}\\& \qquad\qquad + (q-1)q^{n-j-1}\sum_{m=0}^{\frac{n-j-i-1}{2}}\chi_{1,\frac{n-j-1}{2}-m}(q)\gs{\frac{n-j-i-1}{2}}{m}{q^{2}}q^{m(m-j+i+2)}\\ &\qquad\qquad+\left.\left(q^{j-i-1}-1\right)q^{n-j-1} \sum_{m=0}^{\frac{n-j-i-1}{2}}\chi_{1,\frac{n-j+1}{2}-m}(q)\gs{\frac{n-j-i-1}{2}}{m}{q^{2}}q^{m(m-j+i+2)}\right)\\
    & \qquad-\eps q^{\frac{1}{2}n-1}\left(\sum_{m=0}^{\frac{n-j-i-1}{2}}\chi_{1,\frac{n-j+1}{2}-m}(q)\gs{\frac{n-j-i-1}{2}}{m}{q^{2}}q^{m(m-j+i+2)}\right.\\ &\qquad\qquad\left.\left.+\sum_{m=0}^{\frac{n-j-i-1}{2}}\chi_{1,\frac{n-j-1}{2}-m}(q)\gs{\frac{n-j-i-1}{2}}{m}{q^{2}}q^{m(m-j+i+2)}\right)\right)\\
    &=q^{\frac{3}{2}jn-\frac{1}{4}n^{2}-\frac{5}{4}j^{2}-\frac{3}{2}n+\frac{3}{2}j-\frac{1}{4}}\left(q^{n-j-1}\left(q^{j} \sum_{m=0}^{\frac{n-j-i+1}{2}}\chi_{1,\frac{n-j+1}{2}-m}(q)\gs{\frac{n-j-i+1}{2}}{m}{q^{2}}q^{m(m-j+i)}\right.\right.\\&\qquad\qquad -q^{j-i} \sum_{m=0}^{\frac{n-j-i+1}{2}}\chi_{1,\frac{n-j+1}{2}-m}(q)\gs{\frac{n-j-i+1}{2}}{m}{q^{2}}q^{m(m-j+i)}\\&\qquad\qquad+\left(q^{j-i}-1\right)\sum_{m=0}^{\frac{n-j-i-1}{2}}\chi_{1,\frac{n-j+1}{2}-m}(q)\gs{\frac{n-j-i-1}{2}}{m}{q^{2}}q^{m(m-j+i+2)}\\& \qquad\qquad\left. + (q-1)\sum_{m=0}^{\frac{n-j-i-1}{2}}\chi_{1,\frac{n-j-1}{2}-m}(q)\gs{\frac{n-j-i-1}{2}}{m}{q^{2}}q^{m(m-j+i+2)}\right)\\
    &\qquad\left. -\eps q^{\frac{1}{2}n-1} \sum_{m=0}^{\frac{n-j-i-1}{2}}\chi_{1,\frac{n-j-1}{2}-m}(q)((q^{n-j-2m}-1)+1)\gs{\frac{n-j-i-1}{2}}{m}{q^{2}}q^{m(m-j+i+2)} \right)\\
    &=q^{\frac{3}{2}jn-\frac{1}{4}n^{2}-\frac{5}{4}j^{2}-\frac{3}{2}n+\frac{3}{2}j-\frac{1}{4}}\left(q^{n-j-1}\left(q^{j} \sum_{m=0}^{\frac{n-j-i+1}{2}}\chi_{1,\frac{n-j+1}{2}-m}(q)\gs{\frac{n-j-i+1}{2}}{m}{q^{2}}q^{m(m-j+i)}\right.\right.\\&\qquad\qquad -q^{j-i} \sum_{m=0}^{\frac{n-j-i+1}{2}}\chi_{1,\frac{n-j+1}{2}-m}(q)\left(\gs{\frac{n-j-i+1}{2}}{m}{q^{2}}-q^{2m}\gs{\frac{n-j-i-1}{2}}{m}{q^{2}}\right)q^{m(m-j+i)}\\&\qquad\qquad-\sum_{m=0}^{\frac{n-j-i-1}{2}}\chi_{1,\frac{n-j+1}{2}-m}(q)\gs{\frac{n-j-i-1}{2}}{m}{q^{2}}q^{m(m-j+i+2)}\\& \qquad\qquad + q\sum_{m=0}^{\frac{n-j-i-1}{2}}\chi_{1,\frac{n-j-1}{2}-m}(q)\gs{\frac{n-j-i-1}{2}}{m}{q^{2}}q^{m(m-j+i+2)}\\& \qquad\qquad\left. -\sum_{m=0}^{\frac{n-j-i-1}{2}}\chi_{1,\frac{n-j-1}{2}-m}(q)\gs{\frac{n-j-i-1}{2}}{m}{q^{2}}q^{m(m-j+i+2)}\right)\\
    &\qquad\left. -\eps q^{\frac{3}{2}n-j-1}\sum_{m=0}^{\frac{n-j-i-1}{2}}\chi_{1,\frac{n-j-1}{2}-m}(q)\gs{\frac{n-j-i-1}{2}}{m}{q^{2}}q^{m(m-j+i)} \right)\\
    &=q^{\frac{3}{2}jn-\frac{1}{4}n^{2}-\frac{5}{4}j^{2}-\frac{1}{2}n+\frac{1}{2}j-\frac{5}{4}}\left(\left(q^{j} \sum_{m=0}^{\frac{n-j-i+1}{2}}\chi_{1,\frac{n-j+1}{2}-m}(q)\gs{\frac{n-j-i+1}{2}}{m}{q^{2}}q^{m(m-j+i)}\right.\right.\\&\qquad\qquad -q^{j-i} \sum_{m=0}^{\frac{n-j-i+1}{2}}\chi_{1,\frac{n-j+1}{2}-m}(q)\gs{\frac{n-j-i-1}{2}}{m-1}{q^{2}}q^{m(m-j+i)}\\&\qquad\qquad-\sum_{m=0}^{\frac{n-j-i-1}{2}}\chi_{1,\frac{n-j-1}{2}-m}(q)\left(\left(q^{n-j-2m}-1\right)+1\right)\gs{\frac{n-j-i-1}{2}}{m}{q^{2}}q^{m(m-j+i+2)}\\& \qquad\qquad\left. + q\sum_{m=0}^{\frac{n-j-i-1}{2}}\chi_{1,\frac{n-j-1}{2}-m}(q)\gs{\frac{n-j-i-1}{2}}{m}{q^{2}}q^{m(m-j+i+2)}\right)\\
    &\qquad\left. -\eps q^{\frac{1}{2}n}\sum_{m=0}^{\frac{n-j-i-1}{2}}\chi_{1,\frac{n-j-1}{2}-m}(q)\gs{\frac{n-j-i-1}{2}}{m}{q^{2}}q^{m(m-j+i)} \right)\\
    &=q^{\frac{3}{2}jn-\frac{1}{4}n^{2}-\frac{5}{4}j^{2}-\frac{1}{2}n+\frac{1}{2}j-\frac{5}{4}}\left(\left(q^{j} \sum_{m=0}^{\frac{n-j-i+1}{2}}\chi_{1,\frac{n-j+1}{2}-m}(q)\gs{\frac{n-j-i+1}{2}}{m}{q^{2}}q^{m(m-j+i)}\right.\right.\\&\qquad\qquad -q^{j-i} \sum_{m=0}^{\frac{n-j-i-1}{2}}\chi_{1,\frac{n-j-1}{2}-m}(q)\gs{\frac{n-j-i-1}{2}}{m}{q^{2}}q^{m(m-j+i+2)}q^{-j+i+1}\\&\qquad\qquad-q^{n-j}\sum_{m=0}^{\frac{n-j-i-1}{2}}\chi_{1,\frac{n-j-1}{2}-m}(q)\gs{\frac{n-j-i-1}{2}}{m}{q^{2}}q^{m(m-j+i)}\\& \qquad\qquad\left. + q\sum_{m=0}^{\frac{n-j-i-1}{2}}\chi_{1,\frac{n-j-1}{2}-m}(q)\gs{\frac{n-j-i-1}{2}}{m}{q^{2}}q^{m(m-j+i+2)}\right)\\
    &\qquad\left. -\eps q^{\frac{1}{2}n}\sum_{m=0}^{\frac{n-j-i-1}{2}}\chi_{1,\frac{n-j-1}{2}-m}(q)\gs{\frac{n-j-i-1}{2}}{m}{q^{2}}q^{m(m-j+i)} \right)\\%&\qquad\left. -\eps q^{\frac{1}{2}n}\sum_{m=0}^{\frac{n-j-i-1}{2}}\chi_{1,\frac{n-j-1}{2}-m}(q)\gs{\frac{n-j-i-1}{2}}{m}{q^{2}}q^{m(m-j+i)} \right)\\
    &=q^{\frac{3}{2}jn-\frac{1}{4}n^{2}-\frac{5}{4}j^{2}-\frac{1}{2}n+\frac{1}{2}j-\frac{5}{4}}\left(\left(q^{j} \sum_{m=0}^{\frac{n-j-i+1}{2}}\chi_{1,\frac{n-j+1}{2}-m}(q)\gs{\frac{n-j-i+1}{2}}{m}{q^{2}}q^{m(m-j+i)}\right.\right.\\&\qquad\qquad\left.-q^{n-j}\sum_{m=1}^{\frac{n-j-i+1}{2}}\chi_{1,\frac{n-j+1}{2}-m}(q)\gs{\frac{n-j-i-1}{2}}{m-1}{q^{2}}q^{m(m-j+i-2)}q^{j-i+1}\right)\\
    &\qquad\left. -\eps q^{\frac{1}{2}n}\sum_{m=0}^{\frac{n-j-i-1}{2}}\chi_{1,\frac{n-j-1}{2}-m}(q)\gs{\frac{n-j-i-1}{2}}{m}{q^{2}}q^{m(m-j+i)} \right)\\
    &=q^{\frac{3}{2}jn-\frac{1}{4}n^{2}-\frac{5}{4}j^{2}-\frac{1}{2}n+\frac{1}{2}j-\frac{5}{4}}\\&\quad\left(q^{j} \sum_{m=0}^{\frac{n-j-i+1}{2}}\chi_{1,\frac{n-j+1}{2}-m}(q)\left(\gs{\frac{n-j-i+1}{2}}{m}{q^{2}}-q^{n-j-i+1-2m}\gs{\frac{n-j-i-1}{2}}{m-1}{q^{2}}\right)q^{m(m-j+i)}\right.\\
    &\qquad\left. -\eps q^{\frac{1}{2}n}\sum_{m=0}^{\frac{n-j-i-1}{2}}\chi_{1,\frac{n-j-1}{2}-m}(q)\gs{\frac{n-j-i-1}{2}}{m}{q^{2}}q^{m(m-j+i)} \right)\\
    &=q^{\frac{3}{2}jn-\frac{1}{4}n^{2}-\frac{5}{4}j^{2}-\frac{1}{2}n+\frac{3}{2}j-\frac{5}{4}}\left(\sum_{m=0}^{\frac{n-j-i-1}{2}}\chi_{1,\frac{n-j+1}{2}-m}(q)\gs{\frac{n-j-i-1}{2}}{m}{q^{2}}q^{m(m-j+i)}\right.\\
    &\qquad\left. -\eps q^{\frac{1}{2}n-j}\sum_{m=0}^{\frac{n-j-i-1}{2}}\chi_{1,\frac{n-j-1}{2}-m}(q)\gs{\frac{n-j-i-1}{2}}{m}{q^{2}}q^{m(m-j+i)} \right)\:.
\end{align*}

\end{document}